\documentclass[11pt,oneside]{article}
\usepackage{tikz}
\usetikzlibrary{arrows.meta, positioning, quotes}
\usetikzlibrary{arrows}
\usepackage{amssymb,amsmath,latexsym,amsfonts,amscd,amsthm,multirow,ctable,colortbl,mathdots,caption,array}
\usepackage{upgreek}
\usepackage{amsmath}
\usepackage{geometry}
\usepackage{wasysym}
\usepackage{mathtools}
\usepackage{fancyhdr,tabularx,cite,mathrsfs,graphicx}
\usepackage{authblk}
\usepackage[headings]{fullpage}
\usepackage{stmaryrd}
\usepackage[all]{xy}
\usepackage{stmaryrd,rotating}
\usepackage{dsfont}
\usepackage{pdflscape,longtable}

\usepackage[font=small]{caption}

\usepackage{setspace}

\usepackage{hyperref}

\newcommand{\bea}{\begin{eqnarray}} 
\newcommand{\eea}{\end{eqnarray}} 
\newcommand{\bee}{\begin{eqnarray*}} 
\newcommand{\eee}{\end{eqnarray*}} 
\newcommand{\al}{\begin{align*}} 
\newcommand{\eal}{\end{align*}} 
\newcommand{\be}{\begin{equation}} 
\newcommand{\ee}{\end{equation}} 
 
\newcommand{\bem}{\begin{pmatrix}} 
\newcommand{\eem}{\end{pmatrix}}

\def\a{\alpha} 
 
\def\c{\gamma}

\def\D{\Delta}

\newcolumntype{R}{ >{$}r <{$}}
\newcolumntype{C}{ >{$}c <{$}}
\newcolumntype{L}{ >{$}l <{$}}
\newcolumntype{F}{>{\centering\arraybackslash}m{1.5cm}}

\newcommand{\gt}[1]{\mathfrak{#1}}
\newcommand{\mc}[1]{\mathcal{#1}}

\newcommand{\comment}[1]{}

\newcommand{\RR}{{\mathbb R}}%Reals
\newcommand{\CC}{{\mathbb C}}%Complex
\newcommand{\ZZ}{{\mathbb Z}}%Integers
\newcommand{\QQ}{{\mathbb Q}}%Rationals
\newcommand{\HH}{{\mathbb H}}%quaternions
\newcommand{\FF}{{\mathbb F}}%finite field
\newcommand{\Inv}{\operatorname{Inv}}

\newcommand{\Id}{\operatorname{Id}}
\newcommand{\tr}{\operatorname{{tr}}}
\newcommand{\Tr}{\operatorname{{Tr}}}

\newcommand{\ex}{\operatorname{e}} %Number theory exp

\newcommand{\Ex}{\operatorname{Ex}}

\newcommand{\wh}{{\rm wh}}	%weakly holomorphic
\newcommand{\xmod}{{\rm \;mod\;}}

\newcommand{\Th}{\textsl{Th}}

\newcommand{\PSL}{\operatorname{\textsl{PSL}}}    %PSL group
\newcommand{\SL}{\operatorname{\textsl{SL}}}      %SL group
\newcommand{\mpt}{\widetilde{\SL}_2}      %Metaplectic double cover of \SL_2
\newcommand{\vn}{V^{\natural}} %Moonshine module
\newcommand{\MM}{\mathbb{M}}	%monster group
\newcommand{\SQ}{\operatorname{\textsl{SQ}}}

\newcommand{\pd}{\uplambda}

\DeclarePairedDelimiter\ceil{\lceil}{\rceil}
\DeclarePairedDelimiter\floor{\lfloor}{\rfloor}

\newtheorem{thm}{Theorem}[section]
\newtheorem*{thm*}{Theorem}

\newtheorem{pro}[thm]{Proposition}
\newtheorem{con}[thm]{Conjecture}

\theoremstyle{definition}

\theoremstyle{remark}
\newtheorem{rmk}[thm]{Remark}

\numberwithin{equation}{subsection}

\fancypagestyle{tables}{
    \fancyhead{}
    
}
\begin{document}

\setstretch{1.26}

\renewcommand{\thefootnote}{\fnsymbol{footnote}} 

\title{
\vspace{-35pt}
\textsc{\huge{ 
Two New Avatars of 
Moonshine for\\ 
\hspace{2.7pt} 
the Thompson Group
}}
    }

\renewcommand{\thefootnote}{\arabic{footnote}}

\author[1]{John F. R. Duncan\thanks{jduncan@gate.sinica.edu.tw}}
\author[2]{Jeffrey A. Harvey\thanks{j-harvey@uchicago.edu}}
\author[3]{Brandon C. Rayhaun\thanks{brayhaun@stanford.edu}}

\affil[1]{Institute of Mathematics, Academia Sinica, Taipei, Taiwan.}
\affil[2]{Enrico Fermi Institute and Department of Physics, University of Chicago, Chicago, IL 60637, U.S.A.}
\affil[3]{Institute for Theoretical Physics, Stanford University, Palo Alto, CA 94305, U.S.A.}

\date{\vspace{-.45in}}

\maketitle

\begin{center}
\vspace{-.12in}
\emph{Dedicated to the memory of Simon P.~Norton.}
\vspace{.12in}
\end{center}

\abstract{
The Thompson sporadic group admits special relationships to modular forms of two kinds. 
On the one hand, last century's generalized moonshine for the monster equipped the Thompson group with a module for which the associated McKay--Thompson series are distinguished weight zero modular functions. 
On the other hand, Griffin and Mertens verified the existence of a module for which the McKay--Thompson series are distinguished modular forms of weight one-half, that were assigned to the Thompson group in this century by the last two authors of this work. 
In this paper we round out this picture by proving the existence of two new avatars of Thompson moonshine: a new module giving rise to weight zero
modular functions, and a new module giving rise to forms of weight one-half. 
We explain how the newer modules are related 
to the 
older ones by 
Borcherds products and traces of singular moduli. 
In so doing 
we 
clarify the relationship between the previously known modules, and expose a new arithmetic aspect to moonshine for the Thompson group. 
We also present evidence that this phenomenon extends to a correspondence between other cases of generalized monstrous moonshine and penumbral moonshine, and thereby enriches these phenomena with counterparts in weight one-half and weight zero, respectively. 
}

\clearpage

\tableofcontents

%------------------------------------------------------------------%
\section{Introduction}\label{sec:int}
%------------------------------------------------------------------%

A main goal of the moonshine program is to explain connections between distinguished objects from 
finite group theory 
and 
modular form theory 
by constructing suitably rich 
objects in mathematical physics and algebra.
This goal has been attained to great effect 
in the setting of monstrous and generalized monstrous moonshine, 
where the modular forms arising have weight zero, 
but less so in the more recently appearing setup of umbral and penumbral moonshine, where the relevant weight is one-half. 
In a companion paper \cite{pmp} we 
construct a ``lift'' of the singular theta lift of \cite{Harvey:1995fq,Borcherds:1996uda} that can be used to translate between moonshine 
in weight one-half and moonshine in weight zero. 
In this work we explore the consequences of this construction in a special case, and in so doing uncover two new avatars of moonshine for the Thompson group. 
We also illuminate a more expansive landscape that connects the moonshine in these two weights.

\subsection{Weight Zero}\label{sec:int-zer}

The story of moonshine in weight zero begins with
monstrous moonshine, as formulated in \cite{Conway:1979qga}. 
Here, the connection to be explained is one between the representation theory of the sporadic simple monster group, $\MM$, and the Fourier expansions of a collection of principal moduli $\{T_g\}_{g\in\MM}$ (also known as Hauptmoduln) for genus zero subgroups of $\SL_2(\RR)$. (See \S~\ref{sec:pre-mod} for the relevant definitions.) 
This connection is detailed in the {``moonshine conjectures''} of op.\ cit. 
It entails, 
in particular, the initial observations of McKay and Thompson which gave birth to the subject (see \cite{Tho_FinGpsModFns,Tho_NmrlgyMonsEllModFn}), wherein
coefficients of the normalized elliptic modular invariant, 
\begin{align}\label{normalizedj}
J(\tau) = \sum_{n\geq -1} c(n)q^n  = q^{-1} + 196884q + 21493760q^2+864299970q^3+\dots \ \ \  (q=\ex(\tau)),
\end{align}
are decomposed into non-trivial sums of dimensions of irreducible representations of $\MM$. E.g.\ 
\begin{gather}
\begin{split}\label{eqn:initialobservations}
196884 &= 196883 + 1,\\ 
21493760 &= 21296876+196883 + 1,\\
864299970&=
842609326
+
21296876
+
2\times 196883
+
2\times 1
,
\end{split}
\end{gather} 
and so on. (In (\ref{normalizedj}) and henceforth, we employ the notation $\ex(x):= e^{2\pi i x}$.) 

The algebraic construction which provides the stage upon which 
the monster group $\MM$ and the principal moduli $T_g$ both act---and which also has an interpretation in physics---is a particular 
vertex operator algebra (VOA) $V^\natural$, known as the moonshine module. 
It was constructed explicitly by Frenkel, Lepowsky and Meurman \cite{flm,FLMBerk,frenkel1989vertex}, and in physical terms can be described as a chiral conformal field theory (CFT) defined by an asymmetric $\ZZ/2\ZZ$-orbifold of a $c=24$ theory of compact chiral bosons based on the Leech lattice. The symmetries of 
$V^\natural$ are described by the monster group $\MM$, and---as was ultimately shown by Borcherds \cite{Borcherds1992}---its {\em twined} partition functions,
\begin{align}\label{McKayThompsonSeries}
T_g(\tau) := \tr\left(gq^{L_0-1}\big|{V^\natural}\right), \ \ \ \ \ \ g\in\mathbb{M},
\end{align}
also known as McKay--Thompson series, recover the principal moduli predicted in \cite{Conway:1979qga} precisely.
(In (\ref{McKayThompsonSeries}) $L_0$ denotes the zero mode of the stress-energy tensor of the chiral CFT that underlies $V^\natural$.) Specializing $g$ to be the identity element in (\ref{McKayThompsonSeries}) recovers $T_{e} = J$, which endows the normalized elliptic modular invariant (\ref{normalizedj}) with an interpretation as the graded dimension of $V^\natural$, and explains the observations (\ref{eqn:initialobservations}) of McKay and Thompson.

It develops that a ``twist'' of this construction gives context to a broader set of conjectures 
formulated by Norton (see \cite{norton,norton-rev}), 
which together are known as {generalized (monstrous) moonshine}.
These conjectures extend the relationship between $\MM$ and $\{T_g\}_{g\in \MM}$ to one between central extensions $K_h.C_\MM(h)$ of centralizer subgroups $C_\MM(h)$ of $\MM$, and also a larger collection of functions $\{T_{h,g}\}_{g\in K_h.C_\MM(h)}$, 
each of which is either a principal modulus or constant. 
Here $h$ is any element of $\MM$, and the content of monstrous moonshine is obtained by specializing $h$ to be the identity element, $T_{g} = T_{e,g}$. 

The physical intuition \cite{dixon1988} behind generalized monstrous moonshine comes from orbifold conformal field theory \cite{dixon1985strings,dixon1986strings,Dixon:1986qv,Hamidi:1986vh}: 
Roughly speaking, the idea is that $T_{h,g}$ should be recovered as the $g$-twined partition function of the moonshine module CFT, but where the usual periodic boundary conditions, $\varphi(e^{2\pi i}z) =\varphi(z)$, are replaced by more general twisted boundary conditions in which operators are only required to come back to themselves up to the action of a symmetry group element, $\varphi(e^{2\pi i}z)=h\cdot \varphi(z)$. 
Denoting this $h$-twisted sector of the theory by $V_h^\natural$, we then have the expectation that the {\em $h$-twisted $g$-twined} partition functions agree
with the functions described by Norton, 
\begin{align}\label{eqn:int-zer:Thg}
T_{h,g}(\tau) = \tr\left(gq^{L_0-1}\big|V^\natural_h\right), \ \ \ \ \ h\in\MM,\ g\in K_h.C_\MM(h).
\end{align}
(That the right-hand side of (\ref{eqn:int-zer:Thg}) is indeed a principal modulus or constant was proven by Carnahan in \cite{Carnahan:2012gx}. 
Many of the $T_{h,g}$ belong to the list of replicable functions given in \cite{for-mck-nor}.)

\subsection{Infinite Products}\label{sec:int-inf}

Although the introduction of the moonshine module provided a conceptually satisfying explanation for why there should be interaction between objects from such disparate areas of mathematics in the first place, there still remained work to be done to prove the moonshine conjectures of \cite{Conway:1979qga}. 
For example, although arguments from VOA theory may be applied to the moonshine module to establish that its twisted and twined series (\ref{eqn:int-zer:Thg}) should have modular properties (cf.\ e.g.\ \cite{MR2793423,Gaberdiel2010}, and see e.g.\ \cite{Dong2000,MR3435813,MR2046807,Zhu_ModInv} for rigorous results along these lines), these arguments are insufficient for demonstrating that the $T_g=T_{e,g}$, let alone the non-constant $T_{h,g}$, are \emph{Hauptmoduln} for genus zero groups. This property makes them much more special than generic modular functions.

To prove that the (untwisted) McKay--Thompson series $T_g$ possess the ``genus zero property'' just mentioned, Borcherds first invented the notion of a generalized Kac--Moody algebra \cite{Borcherds1988,Borcherds1991}---which we henceforth refer to as a Borcherds--Kac--Moody algebra (BKM)---and then constructed a particular example $\mathfrak{m}$, called the {monster Lie algebra} \cite{Borcherds1992} (cf.\ also \cite{Borcherds1990}), which is obtained by applying a suitably defined ``second quantization'' functor to $V^\natural$. As its name suggests, $\mathfrak{m}$ naturally inherits an action of $\MM$ from its construction.  

The twined denominator formulae of the monster Lie algebra BKM impose infinitely many identities upon the $q$-expansion coefficients of the McKay--Thompson series $T_g$, and it develops that these are strong enough to determine them from their first five Fourier coefficients. For example, the denominator formula for $\mathfrak{m}$ is the well-known infinite product formula 
\begin{align}\label{eqn:int-inf:Jsigma-Jtau}
J(\sigma)-J(\tau) = p^{-1}\prod_{m>0}\prod_{n\geq -1}(1-p^mq^n)^{c(mn)} \ \ \ \ \ \  (p = \ex(\sigma)),
\end{align}
which was discovered independently by Koike, Norton and Zagier (cf.\ \cite{koikereplication,ZagierTSM}).
The $c(mn)$ in (\ref{eqn:int-inf:Jsigma-Jtau}) are the coefficients of $J$, as in \eqref{normalizedj}, and according to (\ref{eqn:int-inf:Jsigma-Jtau}) recover dimensions of root spaces in $\gt{m}$. 
Remarkably, there are no terms with mixed powers of $p$ and $q$ on the left-hand side of (\ref{eqn:int-inf:Jsigma-Jtau}), so expanding out the right-hand side of this equation, 
and setting the coefficients of all mixed-power terms to zero, we recover restrictive recursion relations (see e.g.\ (9.1) of \cite{Borcherds1992}) satisfied by the $c(n)$ of \eqref{normalizedj}.

Similar formulae to (\ref{eqn:int-inf:Jsigma-Jtau}) can be obtained for the McKay--Thompson series $T_g$. Namely, we have
\begin{align}\label{eqn:int-inf:Tgsigma-Tgtau}
T_g(\sigma)-T_g(\tau) = p^{-1}\exp\left(
-
\sum_{m>0}
\sum_{n\geq -1}
\sum_{k>0}
c_{g^k}(mn)\frac{p^{mk}q^{nk}}{k}   \right), 
\end{align}
where now $c_{g}(n)$ is the coefficient of $q^n$ in the Fourier expansion of $T_g$. After noting that Hauptmoduln satisfy the same recursion relations, and that the $T_g$ all match known Hauptmoduln to low orders in their $q$-expansions, the conjectures of Conway and Norton follow. Roughly speaking, the proof of generalized monstrous moonshine, due to Carnahan \cite{Carnahan2010,Carnahan2012,Carnahan:2012gx}, proceeds similarly, using twisted versions $\mathfrak{m}_h$ of the monster Lie algebra and their twined denominator formulae to pin down the functions $T_{h,g}$ (cf.\ e.g.\ \cite{hoehn_babymonster} for a special case). Note though that the implementation of this idea involves serious challenges in the general case, and was only relatively recently achieved.

Given the evident power of the monster Lie algebra, it is natural to ask about its physical interpretation. In fact, inspired by work of Frenkel \cite{fre85}, Borcherds made use of the no-ghost theorem \cite{noghost,Brower:1972wj} (see also \cite{Kato:1982im,Banks:1985ff,FGZ86}) in his construction of $\mathfrak{m}$, which is suggestive of an embedding into string theory. Along these lines, we mention recent work of Paquette, Persson and Volpato \cite{Paquette:2016xoo,Paquette:2017xui}, which interprets $\mathfrak{m}$ as an algebra of spontaneously broken gauge symmetries in a particular low-dimensional heterotic string compactification. (See also \cite{Harrison:2018joy,Harrison:2020wxl,Harrison:2021gnp,finitealldirections} for related results.)

\subsection{Weight One-Half}\label{sec:int-woh}

The last decade has witnessed a resurgence of activity in the study of moonshine phenomena, starting in 2010 with the discovery of Mathieu moonshine by Eguchi, Ooguri, and Tachikawa \cite{eguchi2011notes}  in the elliptic genus of K3, which counts supersymmetric (BPS) states in superconformal sigma models with a K3 surface target space. (See \cite{MR3539377} for the proof of the Mathieu moonshine conjecture.) Mathieu moonshine is the first example in a family of similar phenomena known as umbral moonshine \cite{cheng2014umbral,Cheng:2012tq,cheng2018weight,DGO15}, which feature connections between (among other things) finite groups and 
mock 
Jacobi forms. 

Naively, the right objects are at play here for an explanation of these connections similar in spirit to the one appearing in monstrous and generalized monstrous moonshine. 
However, ideas with as much power as those found in the ``classical moonshine'' of \S\S~\ref{sec:int-zer}--\ref{sec:int-inf} are still lacking. For example, one might hope that superconformal field theories based on K3 surfaces would play a role in the umbral theory analogous to the one played by the moonshine module $V^\natural$ in classical moonshine, but a 
theorem of Gaberdiel, Hohenegger and Volpato \cite{Gaberdiel:2011fg} shows that no point in the moduli space of K3 sigma models realizes the full Mathieu group $\textsl{M}_{24}$ as the group preserving $(4,4)$ superconformal symmetry of the theory. There has been some progress in extending the automorphism groups arising from K3 sigma models beyond those allowed by \cite{Gaberdiel:2011fg} by using the symmetry surfing approach of Taormina and Wendland \cite{Taormina:2011rr} for families of K3 sigma models, or by looking at groups preserving only $(4,1)$ superconformal symmetry \cite{Harvey:2020jvu}. 
There is also a kind of universal VOA for K3 surfaces coming from moonshine for the Conway group (see \cite{MR3465528} and cf.\ also \cite{fourtorus}) which comes close to solving the Mathieu moonshine module problem, 
and there are concrete module constructions for some cases of umbral moonshine (see \cite{MR3922534,MR3995918,MR3649360,MR3859972}),
but we are still lacking a full understanding of the emergence of $\textsl{M}_{24}$ and the other umbral groups.

More recently, an even more mysterious moonshine phenomenon,
which relates the 
sporadic simple group 
$\Th$ 
of Thompson 
\cite{MR399257,MR0399193} to a certain family of weight $\frac12$ modular forms,
\begin{align}\label{eqn:penumbralthMTseries}
    \breve{F}^{(-3,1)}_g(\tau) = \sum_{{D\geq -3}}\breve{C}^{(-3,1)}_g(D)q^D, \ \ \ \ \ g\in\Th,
\end{align}
was conjectured in \cite{Harvey:2015mca}.
This conjecture was proven in \cite{Griffin2016}, wherein the authors showed that the $\breve{F}^{(-3,1)}_g$ consistently form the McKay--Thompson series of a $\Th$-module\footnote{In \cite{Harvey:2015mca} the functions $\breve{F}_g^{(-3,1)}$ were labeled $\mathcal{F}_{3,[g]}$, and the module $\breve{W}^{(-3,1)}$ was simply labeled $W$.} 
\begin{align}\label{eqn:breveW31}
    \breve{W}^{(-3,1)} = \bigoplus_{D\geq-3} \breve{W}^{(-3,1)}_D
\end{align} 
in the sense that $\breve{C}_g^{(-3,1)}(D) = \tr(g|\breve{W}^{(-3,1)}_D)$ for all $g\in \Th$ and $D\geq -3$, though an explicit construction is still missing.
(See also \cite{khaqan2020elliptic} for work which illuminates Thompson moonshine phenomena in weight $\frac32$.) 

This Thompson moonshine 
(\ref{eqn:penumbralthMTseries}--\ref{eqn:breveW31})
plays a role similar to the one played by Mathieu moonshine in the broader umbral theory. Namely, it arises as the first example in a more general family called \emph{penumbral moonshine} \cite{pmo,pmz}, which relates finite groups to collections of so-called {skew-holomorphic Jacobi forms}. Umbral and penumbral moonshine are two halves of a whole. Indeed, skew-holomorphic Jacobi forms are natural partners of holomorphic Jacobi forms in that their theta-coefficients transform under conjugate representations of the metaplectic double cover of the modular group (or a congruence subgroup, more generally).
Moreover, just as the mock Jacobi forms of umbral moonshine can be organized by \emph{non-Fricke} groups of genus zero (see \cite{Cheng:2016klu}), the skew-holomorphic Jacobi forms which play a role in the penumbral theory enjoy a relationship (see \cite{pmo,pmz}) to the remaining \emph{Fricke} genus zero groups which did not participate in the umbral classification. 
(See \S~\ref{sec:pre-fuc} for the definition of a Fricke genus zero group.) 
An important objective is to find structures, like the moonshine module $V^\natural$ and the monster Lie algebra $\mathfrak{m}$, which ``explain'' Thompson and penumbral moonshine in the sense described above.

\subsection{A Point of Comparison}\label{sec:int-com}

Note now that the Thompson sporadic simple group is related to the centralizer of any element $h$ belonging to the conjugacy class of $\MM$ called 3C in \cite{atlas}, by 
\begin{align}\label{eqn:CMh}
C_\MM(h) \cong \ZZ/3\ZZ\times\Th.
\end{align}
Thus, according to our discussion in \S~\ref{sec:int-zer}, the group $\Th$ should play a role in generalized monstrous moonshine. 
Indeed, taking $h$ as in (\ref{eqn:CMh}), the space $V^\natural_{h}$ we called the $h$-twisted sector naturally admits an action by $\Th$, 
as this group admits no non-trivial central extensions (see e.g.\ \cite{atlas}).
Just looking at the graded dimension of $V^\natural_h$, which is given (after rescaling) by\footnote{The graded dimension of $V^\natural_{\rm 3C}$ is $T_{{\rm 3C},e}(\tau)=j(\frac\tau 3)^{\frac13}$, but to simplify our analysis, we use the rescaled series $\widetilde{T}_{{\rm 3C},g}(\tau) := T_{{\rm 3C},g}(3\tau)$ for the remainder of the paper.}  
\begin{align}\label{eqn:int:tildeT3Ce}
\widetilde{T}_{\mathrm{3C},e}(\tau):= T_{\mathrm{3C},e}(3\tau) = j(\tau)^{\frac13} = q^{-\frac13}(1+248q + 4124q^{2}+34752q^{3}
	+\dots),
\end{align}
we see 
counterparts 
\begin{gather}
\begin{split}\label{eqn:Th3Cobservations}
248 &= 248,\\ 
4124 &= 4123 + 1,\\
34752&=
30628
+
4123
+
1
,
\end{split}
\end{gather} 
to the observations (\ref{eqn:initialobservations}) of McKay and Thompson.
(See Tables \ref{character_table_th_1}--\ref{character_table_th_2} for the irreducible characters of $\Th$, and see \S~3.5 of \cite{Carnahan2012} for further discussion of the 3C case of generalized monstrous moonshine). 

At this point we mention that the twisted sector $V^\natural_h$ depends only on the conjugacy class of $h$, and similarly for the graded characters $T_{h,g}$ that it defines. So from now on we write $V^\natural_{\rm 3C}$ for $V^\natural_h$, and ${T}_{{\rm 3C},g}$ for ${T}_{h,g}$, \&c., when $h$ belongs to the class 3C as in (\ref{eqn:CMh}),
and this explains our choice of notation in (\ref{eqn:int:tildeT3Ce}).
We also mention that candidates for the graded characters ${T}_{{\rm 3C},g}$ for $g\in \Th$ (cf.\ (\ref{eqn:int-zer:Thg})) were first written down in \cite{queen}.

It is natural to ask now if there is any connection between the $\Th$-module $V^\natural_{\rm 3C}$ realized by 3C-generalized moonshine for the monster, and the 
penumbral $\Th$-module $\breve W^{(-3,1)}$ of (\ref{eqn:breveW31}). In fact, at the level of graded dimensions at least, there is a close connection, in that the 
weakly holomorphic modular form of weight $0$ (\ref{eqn:int:tildeT3Ce}),
that serves as the rescaled graded dimension
of $V^\natural_{\rm 3C}$,
can be
realized as an infinite product constructed from the Fourier coefficients of the weakly holomorphic modular form of weight $\frac12$,
\begin{align}\label{eqn:int-com:breveF-31}
\breve{F}^{(-3,1)}(\tau):=\breve{F}^{(-3,1)}_e(\tau)  = 2q^{-3}+248+2\cdot 27000 q^4-2\cdot 85995 q^5+\dots,
\end{align} 
that serves as the (signed) graded dimension function
of $\breve W^{(-3,1)}$.
(See \S~\ref{sec:pre-mod} for the definition of a weakly holomorphic modular form.) 
Specifically, it follows from the results of \cite{Borcherds1995} that we have a product formula
\begin{align}\label{eqn:int-com:infiniteproductT3C}
\widetilde{T}_{\mathrm{3C},e}(\tau)^2 = 
{q^{20}\left(\prod_{n>0} (1-q^n)^{\breve{C}^{(-3,1)}(n^2)}\right)}{\eta(\tau)^{-496}},
\end{align}
valid for $\Im(\tau)$ sufficiently large, where 
$\breve{C}^{(-3,1)}(D):=\breve{C}^{(-3,1)}_e(D)$ (cf.\ (\ref{eqn:penumbralthMTseries}), (\ref{eqn:int-com:breveF-31})) is the (signed) dimension of $\breve{W}^{(-3,1)}_D$, 
and $\eta(\tau):=q^{\frac{1}{24}}\prod_{n>0} (1-q^n)$ is the Dedekind eta function. 

We may wonder whether this relationship (\ref{eqn:int-com:infiniteproductT3C}) between $V^\natural_{\mathrm{3C}}$ and $\breve W^{(-3,1)}$ extends beyond their graded dimensions to a relationship as $\Th$-modules. In light of the established CFT and VOA interpretations of $V_{\mathrm{3C}}^\natural$, there is hope that it might be possible to leverage such a relationship so as to uncover algebraic and perhaps also physical aspects to penumbral Thompson moonshine. Thus we take the problem of elucidating the relationship between $V^\natural_{\mathrm{3C}}$ and $\breve W^{(-3,1)}$ as the main motivation for this paper.

To appreciate the approach we take observe that the infinite product formula \eqref{eqn:int-com:infiniteproductT3C} may be thought of as a Thompson-moonshine analog of the denominator formula \eqref{eqn:int-inf:Jsigma-Jtau} for the monster Lie algebra, since $\breve{F}^{(-3,1)}$ (\ref{eqn:int-com:breveF-31}) serves as a Thompson-moonshine analog of the graded dimension function (\ref{normalizedj}) of the moonshine module. 
This suggests, in analogy with the statement that $\mathfrak{m}$ is a second quantization of $V^\natural$, that 
$(V_{\mathrm{3C}}^\natural)^{\otimes 2}$
might be recovered by applying a suitable second-quantization functor to $\breve{W}^{(-3,1)}$. Now the definition of the second-quantization functor introduced in \cite{Borcherds1992}, which links $V^\natural$ to $\mathfrak{m}$, relies upon 
the 
VOA
structure of  the former. 
Without access to similar structure on $\breve{W}^{(-3,1)}$ at present, it is not clear yet how to carry this procedure out---at an algebraic level---in the context of penumbral Thompson moonshine. 
Nonetheless, we \emph{do} know the structures of 
$\breve{W}^{(-3,1)}$ 
and 
$(V^\natural_{\mathrm{3C}})^{\otimes 2}$ 
as graded modules for the Thompson group, and so we may ask if they are compatible or not. 
That is, we may ask if there are appropriate Thompson-moonshine analogs of 
the twined denominator formulae for the monster Lie algebra \eqref{eqn:int-inf:Tgsigma-Tgtau}, and if so, 
we may ask if they connect the functions $\breve{F}^{(-3,1)}_g$ (\ref{eqn:penumbralthMTseries}) to the functions $\widetilde{T}_{\mathrm{3C},g}^2$ (cf.\ (\ref{eqn:int:tildeT3Ce})), in analogy with (\ref{eqn:int-com:infiniteproductT3C}).

\subsection{A General Construction}\label{sec:int-con}

To explore the relationship between $\breve{W}^{(-3,1)}$ (\ref{eqn:breveW31}) and $V^\natural_{\rm 3C}$ (\ref{eqn:int:tildeT3Ce})
we
apply the general construction 
\begin{gather}\label{eqn:int-con:SQ}
W\mapsto \SQ(W) 
\end{gather} 
of \cite{pmp},
which translates graded modules for which the associated McKay--Thompson series are modular forms of weight $\frac12$, into graded modules whose McKay--Thompson series have weight $0$. 
This construction (\ref{eqn:int-con:SQ}) 
serves as a ``lift'' of the singular theta lift of 
\cite{Harvey:1995fq,Borcherds:1996uda} 
(in the $O_{2,1}$ setting) from modular forms to $G$-modules, for $G$ a finite group. It is analogous to the second-quantization functor of \cite{Borcherds1992} mentioned above, 
but operates at the level of (virtual) $G$-modules rather than Lie algebras. 

Applying the construction (\ref{eqn:int-con:SQ}) to the penumbral Thompson moonshine module $\breve{W}^{(-3,1)}$ we obtain 
a graded $\Th$-module $V^{(-3,1)}$
with the same graded dimension as $(V_{\rm 3C}^\natural)^{\otimes 2}$ (up to a rescaling of the latter), whose McKay--Thompson series $T^{(-3,1)}_g$ are weakly holomorphic modular forms of weight $0$
(see Theorem \ref{thm:res-ava:wz}).
We also obtain
(also in Theorem \ref{thm:res-ava:wz}) 
that these McKay--Thompson series of $V^{(-3,1)}$ admit infinite product expressions
\begin{align}\label{eqn:int-com:T31g}
T^{(-3,1)}_g(\tau) = q^{20}\exp\left(-\sum_{n>0}\sum_{k>0}\breve{C}^{(-3,1)}_{g^k}(n^2)\frac{q^{nk}}{k}     \right)\eta_g(\tau)^{-2} 
\end{align}
for $g\in \Th$ (cf.\ (\ref{eqn:int-com:infiniteproductT3C})), valid for $\Im(\tau)$ sufficiently large, where 
\begin{align}\label{eqn:int-com:etag}
\eta_g(\tau) := \prod_{b>0} \eta(b\tau)^{v_b(g)},
\end{align}
for $\pi(g)=\prod_{b>0}b^{v_b(g)}$ the Frame shape
defined by the action of $\Th$ on 
its unique irreducible representation of dimension 248.
(Cf.\ (\ref{eqn:Th3Cobservations}), and see Table \ref{tab:res-ava:pig} for the Frame shapes $\pi(g)$.)

Note that 
it is not a priori clear that the right-hand side of (\ref{eqn:int-com:T31g}) makes sense as a function on the upper half-plane.
We establish the convergence of this expression (for $\Im(\tau)$ sufficiently large) and also establish its modularity, 
by applying the results of \cite{pmp}. These results in turn depend upon the singular theta lift developed in \cite{Borcherds:1996uda,Harvey:1995fq}.

We note that, while Thompson moonshine is the main focus of this work, the construction (\ref{eqn:int-con:SQ}) is quite general, and may be applied, for example, to any of the modules of penumbral moonshine \cite{pmo}. 
See \S~\ref{sec:res-gen} for more discussion of this.

At this point we ask whether or not $V^{(-3,1)} \stackrel{?}{=} V_{\mathrm{3C}}^\natural\otimes V_{\mathrm{3C}}^\natural$ as $\Th$-modules. That is, we ask whether or not
\begin{align}\label{putativetwine}
T^{(-3,1)}_{g}\stackrel{?}{=} \widetilde{T}_{\mathrm{3C},g}^2 
\end{align}
for all $g\in \Th$. 
If true 
we would obtain
a Thompson-moonshine analog of the pair $(V^\natural,\mathfrak{m})$, with $\breve{W}^{(-3,1)}$ playing the role of the moonshine module CFT $V^\natural$, and $V^\natural_{\mathrm{3C}}\otimes V^\natural_{\mathrm{3C}}$ 
arising from the root spaces of a penumbral Thompson-moonshine analog of the monster Lie algebra $\mathfrak{m}$.

\subsection{The Avatars}\label{sec:int-ava}

It develops that \eqref{putativetwine} holds for 3-regular elements of the Thompson group (i.e.\ elements $g\in\Th$ such that $o(g)$ is coprime to $3$), but in general fails for 3-singular elements (see Theorem \ref{thm:res-ava:wz}). 
However, 
we prove (see Theorem \ref{thm:res-ava:wh}) that there is another graded $\Th$-module, 
\begin{align}\label{eqn:breveW3C}
    \breve{W}_{\mathrm{3C}} = \bigoplus_{{D\geq -3}} (\breve{W}_{\mathrm{3C}})_D,
\end{align} 
whose McKay--Thompson series 
\begin{align}\label{eqn:breveF3Cg}
\breve{F}_{\mathrm{3C},g}(\tau) 
= \sum_{{D\geq -3}} \breve{C}_{\mathrm{3C},g}(D) q^D
:= \sum_{{D\geq -3}} \tr(g|(\breve{W}_{\mathrm{3C}})_D)q^D
\end{align}
are weakly holomorphic modular forms of weight $\frac12$,
such that when we apply the construction (\ref{eqn:int-con:SQ}) to $\breve W_{\rm 3C}$ we recover the $\Th$-module $(V^\natural_{\rm 3C})^{\otimes 3}$.
Moreover, the associated McKay--Thompson series of $\breve W_{\rm 3C}$
admit infinite product formulae 
similar to (\ref{eqn:int-com:T31g}). 
We prove (see Theorem \ref{thm:res-ava:wh}) that these infinite products converge for $\Im(\tau)$ sufficiently large, and define weakly holomorphic modular forms of weight $0$. 
By virtue of this we obtain product formulae
\begin{align}\label{twined_th_product}
\widetilde{T}_{\mathrm{3C},g}(\tau) = q^{-\frac13} \exp\left( -\sum_{n>0}\sum_{k>0}\frac13\breve{C}_{\mathrm{3C},g^k}(n^2)\frac{q^{nk}}{k}    \right)
\end{align}
for all the functions $\widetilde{T}_{{\rm 3C},g}$ of ${\rm 3C}$-generalized monstrous moonshine. 

Note that we are distinguishing the McKay--Thompson series of $\breve{W}_{\mathrm{3C}}$ from those of $\breve{W}^{(-3,1)}$ by using the 
``lambdency''  superscript $(-3,1)$ of the latter to emphasize the connection to penumbral moonshine, and using 
the 3C subscript of the former 
to emphasize a relationship to the 3C-twisted sector $V_{\mathrm{3C}}^\natural$ of monstrous moonshine. 
Another remark is that we consider $\widetilde{T}_{\mathrm{3C},g}^3$ as opposed to $\widetilde{T}_{\mathrm{3C},g}$  or $\widetilde{T}_{\mathrm{3C},g}^2$ in \eqref{twined_th_product} because this is the smallest power of $\widetilde{T}_{\mathrm{3C},g}$ for which such a correspondence works.

\begin{figure}[ht]
\begin{center}
\begin{tikzpicture}
\tikzset{vertex/.style = {shape=rectangle,rounded corners,draw,minimum size=1.5em, text width = 12em,align=center,semithick}}
\tikzset{edge/.style = {->,> = latex',semithick}}
% vertices
\node[vertex] (1) at  (0,0) {Weight one-half \\ 3C-generalized \\ moonshine module, \\
$(\breve{W}_{\mathrm{3C}}$, $\breve{F}_{\mathrm{3C},g})$};
\node[vertex] (2) at  (7,0) {Weight zero \\ 3C-generalized \\ moonshine module, \\
 $((V_{\mathrm{3C}}^\natural)^{\otimes 3},\widetilde{T}_{\mathrm{3C},g}^3)$};
\node[vertex] (3) at  (0,-3.5) {Weight one-half \\ penumbral moonshine \\ module, $(\breve{W}^{(-3,1)},\breve{F}^{(-3,1)}_g)$};
\node[vertex] (4) at  (7,-3.5) {Weight zero \\ penumbral moonshine \\
module, $(V^{(-3,1)},T^{(-3,1)}_g)$};
%edges
\draw[edge] (1.008) to ["$\SQ$/BP"] (2.172);
\draw[edge] (2.190) to ["TSM",swap] (1.350);
\draw[edge] (3.008) to ["$\SQ$/BP"] (4.172);
\draw[edge] (4.188) to ["TSM",swap] (3.352);
\draw[edge] (1.270) to ["3-regular"] (3.090);
\draw[edge] (2.270) to (4.090);
\draw[edge] (3.090) to (1.270);
\draw[edge] (4.090) to ["3-regular",swap] (2.270);
\end{tikzpicture}
\end{center} 
\caption{The four avatars of Thompson moonshine. 
We can pass from the weight one-half avatars to the weight zero avatars using 
a construction ($\SQ$)
that produces Borcherds products (BP) for the McKay--Thompson series in weight zero. 
We conjecture that we can pass from weight zero back to weight one-half 
using traces of singular moduli (TSM). 
The avatars of the same weight have McKay--Thompson series which essentially agree on 3-regular elements of $\Th$.}\label{fouravatars}
\end{figure}
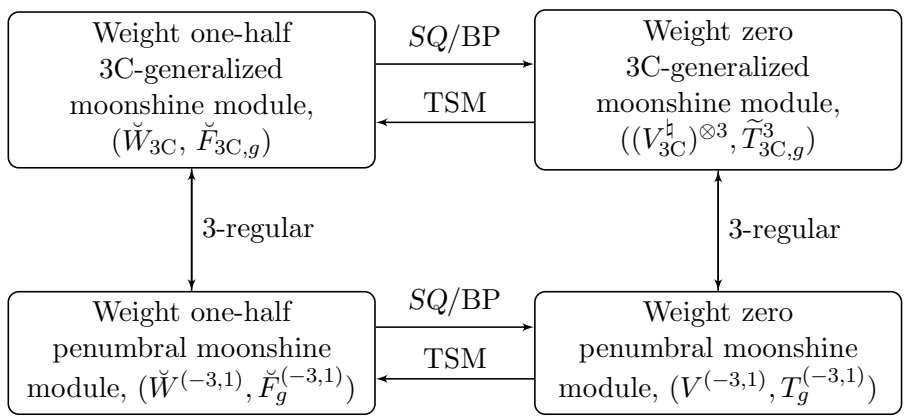

In light of its relationship to 
the 
penumbral Thompson module $\breve{W}^{(-3,1)}$, whose McKay--Thompson series have weight $\frac12$, we call 
$V^{(-3,1)}$, 
whose McKay--Thompson series have weight $0$, the \emph{weight zero penumbral Thompson moonshine module};  
it is the first of two new Thompson moonshine phenomena we consider in this paper. 
By a similar token, we call the new Thompson module $\breve{W}_{\mathrm{3C}}$ 
the \emph{weight one-half 3C-generalized monstrous moonshine module}, in light of its relationship (cf.~(\ref{eqn:breveW3C}--\ref{twined_th_product}))
to $V_{\rm 3C}^\natural$;
it is the second of the two new Thompson moonshine phenomena that we study. 
In total, we arrive at the four avatars of Thompson moonshine that are summarized in Figure \ref{fouravatars}. 

\subsection{An Arithmetic Aspect}

In \S~\ref{sec:int-con} we have described a passage from weight one-half to weight zero, which is facilitated by the machinery of Borcherds products. It is interesting to ask whether or not there is a prescription for going in the reverse direction. It turns out that there is beautiful mathematics which allows one to define the kinds of weight one-half forms that arise in 3C-generalized Thompson moonshine using weight zero modular functions. Indeed, in \S~\ref{sec:res-tra} we give conjectural descriptions for most of the Fourier coefficients of the McKay--Thompson series $\breve{F}_{\mathrm{3C},g}$ of $\breve{W}_{\mathrm{3C}}$ by summing the $\widetilde{T}_{\mathrm{3C},g}^3$ over CM points (i.e.\ quadratic algebraic integers) in the upper half-plane, a method we refer to as taking ``traces of singular moduli'' (TSM), following \cite{ZagierTSM}.

It is important to note that this TSM construction is more subtle than the passage 
from weight one-half to weight zero that is furnished by $\SQ$ (\ref{eqn:int-con:SQ}). 
As we have mentioned, the construction (\ref{eqn:int-con:SQ}) applies to general modules $W$ with McKay--Thompson series that are modular of weight $\frac12$, so as to produce corresponding modules $\SQ(W)$ with McKay--Thompson series that have weight $0$. 
By contrast, the TSM construction we describe in \S~\ref{sec:res-tra} is mostly conjectural, and the $\Th$-equivariance is far from obvious. 
For a broader perspective on this see \S~5 of \cite{pmp}, wherein we approach the relationship between $\SQ$ and TSM in a general way (see \S~5.1 of op.\ cit.), and then highlight the special features of this approach that manifest in the setting of moonshine (see \S~5.2 of op.\ cit.).

In particular, it is remarkable that $\breve W_{{\rm 3C}}$ exists.
To better appreciate this, note first that
the subspaces $(\breve{W}_{\mathrm{3C}})_{D}$ (see (\ref{eqn:breveW3C})) at perfect-square values of the grading $D$ are directly related to $V_{\mathrm{3C}}^\natural$
by virtue of the Borcherds product construction \eqref{twined_th_product}, and in particular, inherit their structure as $\Th$-modules from weight zero in an explicit way.
Now contrast this with the situation at square-free values of $D$, where the corresponding components \emph{also} conspire to be $\Th$-modules, despite being defined by highly non-trivial reorganizations of the information encoded in $V^\natural_{\mathrm{3C}}$. More specifically, the TSM machinery we develop demonstrates that the existence of $\breve W_{{\rm 3C}}$ depends not only on the fact that the Fourier coefficients of the functions $\widetilde{T}_{\mathrm{3C},g}^3$ are organized by the $\Th$ group, but also that their \emph{values} are similarly organized, when summed over special points in the upper half-plane. We therefore consider the incorporation of these square-free components as representing an additional layer of arithmetic richness to Thompson moonshine that is distilled by working in weight one-half, but hidden in weight zero.

We note that the discussion of \S~5 of \cite{pmp} provides evidence that the TSM construction we present here may be generalized so as to invert the application of $\SQ$ to the penumbral Thompson moonshine module $\breve W^{(-3,1)}$ (and in fact, invert its application to any of the modules of penumbral moonshine), and it is on the strength of this that we include an arrow from right to left in the bottom of Figure \ref{fouravatars}. However, we leave the details of this to future work.

\subsection{Two Questions}\label{sec:int-two}

Probably the most important question arising from this work is the elucidation of the algebraic structures and physical notions that underpin the avatars of Thompson moonshine represented in Figure \ref{fouravatars}. 
Certainly an important clue in this regard is the fact that 
the top and bottom rows of this figure essentially collapse---at the level of characters---when we work in characteristic $3$ (cf.\ (\ref{eqn:res-ava:3F31g=2F3Cgplustheta})). 
We also have the analogy between the left- and right-hand sides of Figure \ref{fouravatars}, and the relationship between the monster module $V^\natural$ and the monster Lie algebra $\mathfrak{m}$. 
Thus a compelling possibility is that the weight one-half modules on the left-hand side are two naturally defined lifts to characteristic zero of a single VOA-like object $\breve W$, with manifest $\Th$ symmetry, that is naturally defined in characteristic $3$, while the weight zero modules on the right-hand side are two corresponding lifts arising from a BKM-like counterpart to $\breve W$. 
In particular, we have the intriguing possibility that 
the McKay--Thompson series of $(V^\natural_{\mathrm{3C}})^{\otimes 3}$ have alternate lives as the twined denominators of a BKM (or BKM-like object) with Thompson symmetry. 
Can such a BKM-like structure be realized in a string theoretic setting, along the lines of that developed in \cite{Paquette:2016xoo,Paquette:2017xui}?

We emphasize that the main puzzle of Thompson moonshine, being the failure of (\ref{putativetwine}) for elements with order divisible by $3$, disappears when we pass to characteristic $3$. 
Thus a key point of this paper is that it prepares the way for a monstrous-moonshine-based approach to the module problem for penumbral Thompson moonshine, where the starting point is a realization of the 3C-twisted module for the moonshine vertex operator algebra in characteristic $3$. 
At the time of writing the module problem for umbral Mathieu moonshine remains unsolved. 
Given the analogy between umbral moonshine and penumbral moonshine, it is possible that follow-up work on Thompson moonshine, along the lines we have just sketched, may shed light on umbral moonshine modules too.

We conclude this introduction by highlighting another natural question: 
To what extent does the story we tell here generalize? 
More specifically, we may ask if there are counterparts to Figure \ref{fouravatars} where the lambdency $\pd=(-3,1)$ in the lower row is replaced by other lambdencies from penumbral moonshine \cite{pmo}, and where the conjugacy class 3C in the upper row is replaced by other conjugacy classes of the monster. In \S~\ref{sec:res-gen} 
we provide evidence that the answer to this question is positive, at least for some cases. Thus we anticipate, and lay foundations for, a more extensive interplay between penumbral and generalized monstrous moonshine. It is our hope that these results, building on prior work \cite{ORTL,Cheng:2016klu} on umbral moonshine 
(see \S\S~1.5 and 5.2 of \cite{pmp} for more discussion of this), may serve as a springboard for understanding the relationships between moonshine phenomena in weight one-half and weight zero more broadly. 

\subsection{Overview}\label{sec:int:struc}

The structure of this article is as follows. 
Following this introduction we give a guide to notation in \S~\ref{sec:notation}. 
Then in \S~\ref{sec:pre} we offer a review of relevant background material.
Specifically, we review the Fuchsian groups that arise in generalized monstrous moonshine in \S~\ref{sec:pre-fuc}, 
and we review 
the weight one-half modular forms that form the basis of penumbral moonshine \cite{Harvey:2015mca,pmo,pmz}, and also support the new weight one-half Thompson moonshine avatar $\breve W_{\rm 3C}$ (\ref{eqn:breveW3C}--\ref{twined_th_product}) that we introduce in this paper, in \S~\ref{sec:pre-mod}.
We review the necessary structures to discuss traces of singular moduli in \S~\ref{sec:pre-cmp}, and we complete the presentation of preliminary notions with a discussion of Rademacher sums in \S~\ref{sec:pre-rad}.

Our new results appear in \S~\ref{sec:res}.  
We start in \S~\ref{sec:res-pro} by recalling the defining features of the general construction $\SQ$ from \cite{pmp}.
Then in \S~\ref{sec:res-ava} we 
apply $\SQ$ to
the weight one-half penumbral Thompson module $\breve{W}^{(-3,1)}$, and thus arrive at its weight zero avatar $V^{(-3,1)}$.
We prove 
(see Theorem \ref{thm:res-ava:wz})
that the two weight zero avatars of moonshine for the Thompson group, $V^{(-3,1)}$ and $V^\natural_{\mathrm{3C}}$, have McKay--Thompson series which agree on all 3-regular elements of $\Th$.

We then move on to the weight one-half avatar 
of 3C-generalized Thompson moonshine, 
and prove (see Theorem \ref{thm:res-ava:wh}) the existence of a virtual graded $\Th$-module 
$W_{\rm 3C}$ 
that maps to $(V^\natural_{\mathrm{3C}})^{\otimes 3}$ under the construction $\SQ$.

In \S~\ref{sec:res-tra} we consider the problem of recovering the coefficients of the McKay--Thompson series of 3C-generalized monstrous moonshine in weight one-half as traces of singular moduli. 
Specifically, we 
give a conjectural description (see Conjecture \ref{con:res-tra}) for how our construction $\SQ$ 
can be inverted in this case---at the level of graded characters at least---so as to yield expressions for 
all of the coefficients of (most of) the weight $\frac12$ modular forms $\breve{F}_{\mathrm{3C},g}$ as sums of the weight zero principal moduli $\widetilde{T}_{\mathrm{3C},g}^3$ over CM points in the upper half-plane, in the spirit of work of Zagier \cite{ZagierTSM}. 

In \S~\ref{sec:res-gen}, being the final part of \S~\ref{sec:res}, we provide evidence that the relationship between 
weight one-half moonshine
and 
weight zero 
moonshine that we studied in the special case of the Thompson group generalize to correspondences between instances of penumbral moonshine and instances of generalized monstrous moonshine which are governed by (nearly) the same group (cf.\ Tables \ref{tab:penumbralgroupsD0m3}--\ref{tab:penumbralgroupsD0m4}). In particular, we demonstrate that the Borcherds products associated to the 
graded dimension functions of certain penumbral modules give rise to the weight zero graded dimension functions of corresponding generalized monstrous moonshine modules. 
We also offer a characterization (see Proposition \ref{pro:generalBorPrds}) of the Borcherds products associated to a family of vector-valued modular forms that 
are candidates for identity McKay--Thompson series in yet-to-be-realized instances of penumbral moonshine.

We conclude the main body of this work
by taking stock of our results,
and offering directions for future research,
in \S~\ref{sec:sum}.

Several appendices include additional data and information. 
In the first, \S~\ref{app:ser}, we provide specifications for the main examples of modular forms that appear in this work. 
We specify the McKay--Thompson series of both avatars of weight one-half Thompson moonshine in \S~\ref{app:ser-thm}, 
and specify the McKay--Thompson series of weight zero that arise from 3C-generalized monstrous moonshine in \S~\ref{app:ser-thz}.
These prescriptions play a role in the proofs of the theorems of \S~\ref{sec:res-ava}.
In \S~\ref{app:ser-gen}
we specify functions that supplement the proof of Proposition \ref{pro:generalBorPrds}.

The last three appendices are composed of numerical data. 
We reproduce the character table of the Thompson group 
in \S~\ref{app:tgd}, we print the low-lying coefficients of the McKay--Thompson series of all four of the avatars of $\Th$-moonshine that appear in this work in \S~\ref{app:cff}, 
and we display decompositions of the corresponding (virtual) $\Th$-modules into irreducible constituents in \S~\ref{app:dec}.

\section*{Acknowledgements}

We thank Nathan Benjamin, Miranda Cheng, Sarah Harrison, Shamit Kachru, Michael Mertens, Ken Ono, Natalie Paquette, Roberto Volpato, Max Zimet, and especially Scott Carnahan, for helpful communication and discussions, and we especially thank Simon Norton for sharing comments and computations pertaining to generalized monstrous moonshine.
We also thank anonymous referees for their comments and suggestions.
J.D.\ acknowledges support from the 
U.S.\ National Science Foundation (NSF) 
(DMS-1203162, DMS-1601306), 
the Simons Foundation (\#708354), and the Ministry of Science and Technology of Taiwan (111-2115-M-001-001-MY2).
J.H.\  acknowledges support from the NSF under grant PHY-1520748, and the hospitality of the Aspen Center for Physics supported by the NSF under grant PHY-1607611. 
B.R.\ acknowledges support from the 
NSF 
under grant PHY-1720397. 
This research was also supported in part by the NSF under grant PHY-1748958.

\section{Notation}\label{sec:notation}

\begin{footnotesize}

\begin{list}{}{
	\itemsep -1pt
	\labelwidth 23ex
	\leftmargin 13ex	
	}

\item
[$\cdot*\cdot$]
We set $n*n':=nn'/\gcd(n,n')^2$ for integers $n$ and $n'$, not both zero. 
See (\ref{eqn:estarf}).

\item
[$\left(\frac{\,\cdot\,}{\,\cdot\,}\right)$]
The Kronecker symbol. Cf.\ (\ref{eqn:pre-mod:thetatransform}).

\item
[${\bf 248}$]
A copy of the unique irreducible $\Th$-module of dimension $248$. 
Cf.\ (\ref{eqn:res-ava:3F31g=2F3Cgplustheta}).

\item
[$(-3,1)$]
The lambdency of penumbral Thompson moonshine.
See \S\S~\ref{sec:int-woh}--\ref{sec:int-ava} and cf.\ \S~\ref{sec:res-gen}.

\item[$a(n)$]
The map $n\mapsto a(n)$ defines an isomorphism $\Ex_m\to O_m$.
Cf.\ (\ref{eqn:Om}--\ref{eqn:Omegaalpham}).

\item
[$\alpha$]
A character of $O_m$. Cf.\ (\ref{eqn:Fartau}) and (\ref{eqn:Omegaalpham}).

\item[$\a_Q$]
A CM point in $\HH$. 
See \S~\ref{sec:pre-cmp}.

\item
[$c(n)$]
A Fourier coefficient of $J=T_e$. See (\ref{normalizedj}).

\item
[$c_g(n)$]
A Fourier coefficient of $T_g$. See (\ref{eqn:int-inf:Tgsigma-Tgtau}).

\item
[$\breve{C}_{{\rm 3C},g}(D)$]
A Fourier coefficient of $\breve{F}_{{\rm 3C},g}$. 
See (\ref{eqn:breveF3Cg}).

\item
[$\breve{C}^{(-3,1)}_{g}(D)$]
A Fourier coefficient of $\breve F^{(-3,1)}_g$. See (\ref{eqn:penumbralthMTseries}).

\item
[$C^W(D,r)$]
Shorthand for $C^W_e(D,r)$. 
Cf.\ (\ref{eqn:HW}).

\item
[$C^W_g(D,r)$]
A Fourier coefficient of $F^W_{g,r}$. 
See (\ref{eqn:SQ-BP:FWgr}) and (\ref{eqn:SQ-BP:CWgDr}).

\item
[$\ex(\,\cdot\,)$]
We set $\ex(x):=e^{2\pi i x}$.

\item
[$\Ex_m$]
The group of exact divisors of $m$. 
See (\ref{eqn:Exn}).

\item
[$\Ex_{n|h}$]
The $E\in \Ex_{nh}$ such that every prime divisor of $E$ also divides $n/h$.
See \S~\ref{sec:pre-fuc}.

\item
[$\eta_g$]
An eta product defined by the action of $g\in \Th$ on ${\bf 248}$.
See (\ref{eqn:int-com:etag}).

\item
[$\eta^W_g$]
An eta product defined by the action of $g\in G$ on $W_{0,0}$, for $W$ as in (\ref{eqn:SQ-BP:W}).
See (\ref{eqn:SQ-BP:etaWg}).

\item
[$f^V_g$]
A McKay--Thompson series associated to a graded $G$-module $V$ as in \S~\ref{sec:res-pro}.
See (\ref{eqn:SQ-BP:fVg}).

\item
[$F^W_g$]
A McKay--Thompson series associated to a graded $G$-module $W$ as in (\ref{eqn:SQ-BP:W}). 
Cf.\ (\ref{eqn:SQ-BP:FWgr}).

\item
[$F^W_{g,r}$]
A component function of $F^W_g$. See (\ref{eqn:SQ-BP:FWgr}).

\item
[$\check{F}^W_g$]
A repackaged McKay--Thompson series. See (\ref{eqn:FWgntocheckFWg}).

\item
[$F^{(\pd)}_g$] 
A McKay--Thompson series of the penumbral moonshine module $W^{(\pd)}$.
Cf.\ (\ref{eqn:breveFuplambdag}).

\item
[$\breve F^{(-3,1)}$]
The function $\breve F^{(\pd)}_g$ for $\pd=(-3,1)$ and $g=e$. 
See (\ref{eqn:int-com:breveF-31}) and cf.\ \S~\ref{sec:res-gen}.

\item
[$\breve F^{(-3,1)}_g$]
The function $\breve F^{(\pd)}_g$ for $\pd=(-3,1)$. 
See (\ref{eqn:penumbralthMTseries}).

\item
[$F_{{\rm 3C},g}$]
The McKay--Thompson series defined by the action of $g\in \Th$ on $W_{\rm 3C}$.
See Theorem \ref{thm:res-ava:wh}.

\item
[$\breve{F}_{{\rm 3C},g}$]
The McKay--Thompson series defined by the action of $g\in \Th$ on $\breve W_{\rm 3C}$.
See (\ref{eqn:breveF3Cg}).

\item
[$F_{h}$] 
A function closely related to $F^{(\pd)}=F^{(\pd)}_e$ for $h=h^{(\pd)}$. 
See \S~\ref{sec:res-gen}, especially (\ref{eqn:breveFhuplambda}).

\item
[$\breve{F}_{h}$]
The image of $F_{h}$ under the map (\ref{eqn:prebreveF}--\ref{eqn:breveF}).
See (\ref{eqn:breveFhuplambda}).

\item
[$(\gamma,u)$]
An element of $\widetilde{\SL}_2(\RR)$. 
Cf.\ (\ref{eqn:widetildeSL2R}--\ref{eqn:widetildeSL2R-mlt}).

\item
[$\overline{\Gamma}$]
The quotient $\Gamma/\{\pm\Id\}$ for a subgroup $\Gamma$ of $\SL_2(\RR)$ that contains $\{\pm\Id\}$.

\item[$\overline\Gamma_Q$]
The stabilizer in $\overline\Gamma$ of some binary quadratic form $Q$.
See \S~\ref{sec:pre-cmp}, especially (\ref{sec:pre-cmp:TrGammachiD0}).

\item
[$\widetilde{\Gamma}$]
The preimage of $\Gamma\subset\SL_2(\RR)$ under the natural map $\widetilde{\SL}_2(\RR)\to \SL_2(\RR)$.
See \S~\ref{sec:pre-fuc}.

\item
[$\widetilde{\Gamma}_0(N)$]
The preimage of $\Gamma_0(N)$ in $\widetilde{\SL}_2(\ZZ)$ under the natural map $\widetilde{\SL}_2(\ZZ)\to \SL_2(\ZZ)$. 
See \S~\ref{sec:pre-fuc}.

\item
[$\Gamma_0(N|h)$] 
A shorthand for $\Gamma_0(N|h)+S$ in the case that $S=\{1\}$.
See (\ref{eqn:Gamma0Nverth}).

\item
[$\Gamma_0(N)+S$] 
A shorthand for $\Gamma_0(N|1)+S$. 
Cf.\ (\ref{eqn:Gamma0NverthplusS}).

\item
[$\Gamma_0(N|h)+S$] 
A subgroup of $\SL_2(\RR)$ that normalizes $\Gamma_0(Nh)$. 
See (\ref{eqn:Gamma0NverthplusS}).

\item
[$H$]
A shorthand for $H^W$. 
Cf.\ 
(\ref{eqn:HW})
and 
(\ref{eqn:res-pro:PsiWg}).

\item
[$H^W$]
A generalized class number associated to $W$ as in (\ref{eqn:SQ-BP:W}). 
See (\ref{eqn:HW}).

\item
[$H^W_\ell(g)$]
A generalized class number associated to $W$ as in (\ref{eqn:SQ-BP:W}).
See (\ref{eqn:HWellg}).

\item
[$\ell$]
A cusp representative in \S~\ref{sec:res-ava}, and a lambency of penumbral moonshine in \S~\ref{sec:res-gen}.

\item
[$\pd$] A lambdency of penumbral moonshine. 
See \S~\ref{sec:res-gen}.

\item
[$m+n,n',\dots$]
A shorthand for $\Gamma_0(m)+S$ where $S=\{1,n,n',\dots\}$. 
Cf.\ (\ref{eqn:Nplusnnprime}) and (\ref{eqn:ell}).

\item
[$M^{\wh,+}_{\frac12,m}(N,\chi)$]
A certain space of weakly holomorphic modular forms of weight $\frac12$. 
See \S~\ref{sec:pre-mod}.

\item
[$M^{+}_{\frac12,m}(N,\chi)$]
The subspace of $M^{\wh,+}_{\frac12,m}(N,\chi)$ composed of forms that are holomorphic.
See \S~\ref{sec:pre-mod}.

\item
[$N_g$]
A positive integer attached to a group element $g\in G$, in the setup of \S~\ref{sec:res-pro}.
Cf.\ (\ref{eqn:SQ-BP:FWgnVwh}).

\item
[$N+n,n',\dots$]
A shorthand for $\Gamma_0(N)+S$ where $S=\{1,n,n',\dots\}$. 
See (\ref{eqn:Nplusnnprime}).

\item
[$N|h+n,n',\dots$]
A subgroup of index $h$ in $\Gamma_0(N|h)+S$, where $S=\{1,n,n',\dots\}$. 
See (\ref{eqn:Nverthplusnnprime}).

\item
[$O_m$] The multiplicative group composed of the $a\xmod 2m$ such that $a^2\equiv 1\xmod 4m$. 
See (\ref{eqn:Om}).

\item 
[$\Omega_m^\alpha$]
A projection operator on $\textsl{V}_{\frac12,m}^{\wh}(N,\chi)$. 
See \eqref{eqn:Omegaalpham}.

\item
[$p$]
Either $p=\ex(\sigma)$ for $\sigma\in \HH$, as in (\ref{eqn:int-inf:Jsigma-Jtau}), or $p$ denotes a prime.

\item
[$\pi(g)$]
A shorthand for $\pi(g|{\bf 248})$.
Cf.\ (\ref{eqn:int-com:etag}).

\item
[$\pi(g|U)$]
The Frame shape associated to the action of $g\in G$ on a virtual $G$-module $U$. 
Cf.\ (\ref{eqn:res-pro:pigW00}).

\item 
[$\Psi^W_g$]
The  Borcherds product defined by the action of $g\in G$ on $W$ as in (\ref{eqn:SQ-BP:W}). 
See \eqref{eqn:res-pro:PsiWg}.

\item
[$q$]
We set $q:=\ex(\tau)$ for $\tau\in \HH$. Cf.\ (\ref{normalizedj}).

\item[$\mathcal{Q}_{D}$]
A set of binary quadratic forms of discriminant $D$ with integer coefficients.
Cf.\ (\ref{eqn:pre-cmp:Qgammaxy}).

\item[$\mc{Q}^{(N)}_{D}$]
A particular set of binary quadratic forms with integer coefficients. 
Cf.\ (\ref{eqn:pre-cmp:Qgammaxy}).

\item
[$\varrho_m$]
A Weil representation. 
See \S~\ref{sec:pre-mod}.

\item
[$\SQ$]
Our  lift of the $O_{2,1}$-type singular theta lift.
See (\ref{eqn:int-con:SQ}) and (\ref{eqn:SQ-BP:SQW}).

\item
[$\mpt(\RR)$]
The metaplectic double cover of $\SL_2(\RR)$. 
See (\ref{eqn:widetildeSL2R}--\ref{eqn:widetildeSL2R-mlt}).

\item
[$\mpt(\ZZ)$]
The preimage of $\SL_2(\ZZ)$ under the natural projection $\mpt(\RR)\to \SL_2(\RR)$. 
See \S~\ref{sec:pre-mod}.

\item
[$T_{g}$]
A McKay--Thompson series of monstrous moonshine.
See (\ref{McKayThompsonSeries}) and cf.\ (\ref{eqn:int-zer:Thg}).

\item
[$T_{h,g}$]
A McKay--Thompson series of generalized monstrous moonshine.  
See (\ref{eqn:int-zer:Thg}).

\item
[$T_{{\rm 3C},g}$]
The function $T_{h,g}$ for any $h$ in the 3C class of $\MM$.
See \S~\ref{sec:int-com}.

\item
[$\widetilde{T}_{\mathrm{3C},g}$]
We define $\widetilde{T}_{\mathrm{3C},g}(\tau) := T_{\mathrm{3C},g}(3\tau)$.
See \S~\ref{sec:int-com}, especially (\ref{eqn:int:tildeT3Ce}).

\item
[$T^{(-3,1)}_g$]
A shorthand for $T^W_g$ in case $W=W^{(-3,1)}$. 
See (\ref{eqn:int-com:T31g}) and Theorem \ref{thm:res-ava:wz}.

\item
[$T^W_g$]
The McKay--Thompson series associated to $V^W=\SQ(W)$.
See (\ref{eqn:SQ-BP:TWg}).

\item 
[$\Tr_{D_0,\Gamma}(T,D,\chi)$]
A trace of singular moduli. See \eqref{sec:pre-cmp:TrGammachiD0}.

\item
[$\theta$]
The theta series $\theta(\tau) := \theta_{1,0}^0(\tau)=\sum_{n}q^{n^2}$. Cf.\ (\ref{eqn:pre-mod:thetatransform}).

\item
[$\theta_m^0$]
The 
vector-valued modular form $\theta_m^0:=(\theta_{m,r}^0)$. 
Cf.\ (\ref{eqn:pre-mod:thetamr0}--\ref{eqn:pre-mod:varrhomgammauthetam0}).

\item
[$\theta^{0}_{m,r}$]
An index $m$ Thetanullwert. 
See (\ref{eqn:pre-mod:thetamr0}).

\item
[$v_b(g)$]
A shorthand for $v_b(g|{\bf 248})$.
See (\ref{eqn:int-com:etag}).

\item
[$v_b(g|U)$]
Integers that define the Frame shape $\pi(g|U)$. 
Cf.\ (\ref{eqn:res-pro:pigW00}).

\item
[$V^W$]
See Theorem \ref{thm:SQ} and (\ref{eqn:SQ-BP:SQW}).

\item
[$\textsl{V}^{\wh}_{\frac12,m}$]
We abbreviate $\textsl{V}^{\wh}_{\frac12,m}=\textsl{V}^{\wh}_{\frac12,m}(1,1)$.
See \S~\ref{sec:pre-mod}.

\item
[$\textsl{V}^{\wh,\a}_{\frac12,m}$]
We abbreviate $\textsl{V}^{\wh,\a}_{\frac12,m}=\textsl{V}^{\wh,\a}_{\frac12,m}(1,1)$.
See \S~\ref{sec:pre-mod}.

\item
[$\textsl{V}^{\wh}_{\frac12,m}(N,\chi)$]
A certain space of weakly holomorphic vector-valued modular forms of weight $\frac12$. 
See \S~\ref{sec:pre-mod}.

\item
[$\textsl{V}^{\wh,\alpha}_{\frac12,m}(N,\chi)$]
The subspace of $\textsl{V}^{\wh}_{\frac12,m}(N,\chi)$ defined by a character $\alpha$ of $O_m$.
Cf.\ (\ref{eqn:Fartau}).

\item
[$V_h^\natural$]
The $h$-twisted sector of the moonshine module VOA. 
Cf.\ (\ref{eqn:int-zer:Thg}).

\item
[$V_{\rm 3C}^\natural$]
The 3C-twisted sector of the moonshine module VOA. 
See \S~\ref{sec:int-com}.

\item
[$V^{(-3,1)}$]
The weight zero penumbral Thompson moonshine module. 
See \S~\ref{sec:int-con} and (\ref{eqn:V-31}).

\item
[$W$]
A weakly holomorphic $G$-module of weight $\frac12$ and some index. 
See (\ref{eqn:SQ-BP:W}).

\item
[$\breve W$]
A virtual graded $G$-module obtained from $W$ as in (\ref{eqn:SQ-BP:W}).
See (\ref{eqn:WtobreveW}--\ref{eqn:breveWD}).

\item
[$\breve W_D$]
A homogeneous component of $\breve W$. See (\ref{eqn:breveWD}).

\item
[$W_{r,\frac{D}{4m}}$]
A homogeneous component of $W$. See (\ref{eqn:SQ-BP:W}).

\item
[$W^{(-3,1)}$]
The penumbral Thompson moonshine module.
Cf.\ (\ref{eqn:V-31}).

\item
[$\breve W^{(-3,1)}$]
A form of the penumbral Thompson moonshine module.
See (\ref{eqn:breveW31}).

\item
[$\breve W^{(-3,1)}_D$]
A homogeneous component of $\breve W^{(-3,1)}$. See (\ref{eqn:breveW31}).

\item
[$W_{\rm 3C}$]
The weight one-half 3C-generalized Thompson moonshine module. 
See Theorem \ref{thm:res-ava:wh}.

\item
[$\breve W_{\rm 3C}$]
A form of the weight one-half 3C-generalized Thompson moonshine module. 
See (\ref{eqn:breveW3C}).

\item
[$(\breve W_{\rm 3C})_D$]
A homogeneous component of $\breve W_{\rm 3C}$. 
See (\ref{eqn:breveW3C}).

\item 
[$w_n$]
An Atkin--Lehner coset of $\Gamma_0(N|h)$ for some $N$ and $h$.
See (\ref{eqn:we}-\ref{eqn:Gamma0Nverth}).

\item
[$W_n$]
An Atkin--Lehner coset of $\Gamma_0(N)$ for some $N$.
Cf.\ (\ref{eqn:we}-\ref{eqn:Gamma0Nverth}) and (\ref{eqn:pre-fuc:WEsubsetwn}).

\item
[$\chi$]
A character of $\Gamma_0(N)$ for some $N$, or a (generalized) genus character. See \S\S~\ref{sec:pre-mod}--\ref{sec:pre-cmp}.

\item
[$\chi_{D_0}$]
A genus character. See \eqref{eqn:pre-cmp:chiD0N}.

\item
[$\widetilde{\chi}_{-3}$]
A generalized genus character. See \eqref{eqn:pre-cmp:widetildechiD03}.

\end{list}

\end{footnotesize}

\section{Preliminaries}\label{sec:pre}

In this section we review some preliminary notions which will be useful in the remainder of the paper. 
In \S~\ref{sec:pre-fuc} we define some Fuchsian subgroups of $\SL_2(\RR)$ that are ubiquitous in investigations of moonshine. 
In \S~\ref{sec:pre-mod} we review the main kinds of modular forms which appear in this work.
We review the basic structures underlying traces of singular moduli in \S~\ref{sec:pre-cmp}, and finally present some Rademacher sum expressions in \S~\ref{sec:pre-rad}.

\subsection{Fuchsian Groups}\label{sec:pre-fuc}

Let $\mathbb{H} := \{\tau \in\mathbb{C}\mid \Im(\tau)>0\}$ denote the upper half-plane, equipped with its usual hyperbolic metric.
The group
$\SL_2(\RR)$ 
acts on $\mathbb{H}$ by M\"obius transformations, $\tau\xmapsto{\gamma}\frac{a\tau+b}{c\tau+d}$, and the orientation preserving isometry group of $\mathbb{H}$ is obtained as the image of this action, 
\begin{align}
\mathrm{Isom}^+(\mathbb{H}) = \PSL_2(\RR) = \SL_2(\RR)/\{\pm \Id \}.
\end{align}

Throughout this paper we will work with Fuchsian groups, i.e.\ with discrete subgroups of $\SL_2(\RR)$. 
We will further restrict attention to discrete subgroups $\Gamma$ of $\SL_2(\RR)$ that are commensurable with $\SL_2(\ZZ)$ in the sense that $\Gamma\cap \SL_2(\ZZ)$ has finite index in both $\Gamma$ and $\SL_2(\ZZ)$. 
Of particular interest are the congruence subgroups 
\begin{align}
\begin{split}\label{eqn:Gamma0nGamma1n}
\Gamma_0(N) &:= \left.\left\{\left(\begin{matrix} a & b \\ c & d\end{matrix}\right)\in \SL_2(\ZZ) \,\right|\,  c \equiv 0  \xmod N\right\}, \\
\Gamma_1(N) &:= \left.\left\{\left(\begin{matrix} a & b \\ c & d\end{matrix}\right)\in \SL_2(\ZZ) \,\right|\,  a,d\equiv 1 \xmod N, \  c \equiv 0  \xmod N\right\},
\end{split}
\end{align}
as well as their generalizations to ``$n\vert h$-type groups'' which appear in monstrous and generalized monstrous moonshine. 
Next we describe these generalizations, closely following the conventions of \cite{ferenbaugh1993}
(except that we will use $N$ where op.\ cit.\ uses $n$, and use $n,n',\dots$ where op.\ cit.\ uses $e,e',\dots$ and $e_1,e_2,\dots$).

Denote the set of exact divisors of a positive integer $N$ by 
\begin{gather}\label{eqn:Exn}
\Ex_N := \left\{ n>0  ~\big\vert~  n\vert N,\,\gcd(n,N/n)=1   \right\},
\end{gather} 
and note that this set is naturally endowed with the structure of a finite group with multiplication rule 
\begin{gather}\label{eqn:estarf}
n\ast n' := \frac{nn'}{\gcd(n,n')^2}.
\end{gather} 
Thus every non-identity element has order 2. 
Now choose $h>0$ such that $h$ divides $\gcd(N,24)$, 
and to each exact divisor $n\in\Ex_{N/h}$ attach a set of matrices 
\begin{align}\label{eqn:we}
w_n := \left.\left\{\frac{1}{\sqrt{n}}\left(\begin{matrix}an & b/h \\ cN & dn \end{matrix}\right)\,\right|\,  a,b,c,d\in \ZZ, \ adn^2-bcN/h=n       \right\}.
\end{align}
Then these $w_n$ satisfy the same multiplication rule as do exact divisors of $N/h$, 
i.e., we have 
$w_n w_{n'} = w_{n\ast n'}$ for $n,n'\in \Ex_{N/h}$,
and we also have $w_n\Gamma_0(Nh)=\Gamma_0(Nh)w_n$.
Thus to any subgroup 
$S$ of $\Ex_{N/h}$ we may associate a corresponding subgroup
\begin{gather}\label{eqn:Gamma0NverthplusS}
	\Gamma_0(N|h)+
	S
	:=
	\bigcup_{n\in S}w_n
\end{gather}
of $\SL_2(\RR)$, which normalizes $\Gamma_0(Nh)$, and is commensurable with $\SL_2(\ZZ)$ in the sense just described.
It develops that $\Gamma_0(N|h)+\Ex_{N/h}$ is the full normalizer of $\Gamma_0(Nh)$ in $\SL_2(\RR)$ when $h$ is the largest divisor of $24$ such that $h$ divides $N$ (see \cite{Conway:1979qga}).

The $w_n$ in (\ref{eqn:we}--\ref{eqn:Gamma0NverthplusS}) are referred to as the Atkin--Lehner cosets of 
\begin{gather}\label{eqn:Gamma0Nverth}
\Gamma_0(N|h):=w_1=
\left.\left\{\left(\begin{matrix}a & b/h \\ cN & d \end{matrix}\right)\,\right|\,  a,b,c,d\in \ZZ, \ ad-bcN/h=1       \right\},
\end{gather}
and any element of $w_n$ is called an Atkin--Lehner involution of $\Gamma_0(N|h)$ (even though such elements generally do not have order $2$ with respect to matrix multiplication).
We have $\Gamma_0(N|1)=\Gamma_0(N)$, so we drop the symbols $|1$ from notation in (\ref{eqn:Gamma0NverthplusS}--\ref{eqn:Gamma0Nverth}) when $h=1$.
We follow \cite{ferenbaugh1993} in writing $W_n$ for $w_n$ when $h=1$.

In Tables \ref{tab:mts-T3C}--\ref{tab:bpd:bpp} we use symbols of the form
$N|h+n,n',\dots$,
for $\{1,n,n',\dots\}$ a subgroup of $\Ex_{N/h}$, to specify subgroups of $\SL_2(\RR)$. 
When $h=1$ we simplify this notation to
\begin{gather}\label{eqn:Nplusnnprime}
N+n,n',\dots,
\end{gather} 
and it then just serves as a shorthand for the group $\Gamma_0(N)+S$ as in (\ref{eqn:Gamma0NverthplusS}), for $S=\{1,n,n',\dots\}$. 
When $h>1$ the definition is more subtle. 
To explain this we suppose that $N$ and $h$ are as in (\ref{eqn:we}), with $h>1$,
and for each exact divisor $E$ of $Nh$ write $h_E$ for the largest divisor of $h$ such that $h_E^2|E$. 
Then $E/h_E^2$ is an exact divisor of $N/h$, and the map $E\mapsto E/h_E^2$ defines a surjective homomorphism $\Ex_{Nh}\to \Ex_{N/h}$. 
Moreover, letting $W_E$ denote the Atkin--Lehner coset of $\Gamma_0(Nh)=\Gamma_0(Nh|1)$ defined by $E\in \Ex_{Nh}$, 
we have $W_E\subset w_n$ as subsets of $\SL_2(\RR)$, for $n=E/h_E^2$, since
\begin{align}\label{eqn:pre-fuc:WEsubsetwn}
\frac{1}{\sqrt{E}}\left(\begin{matrix} aE & b \\ cNh & dE \end{matrix}\right) = \frac{1}{\sqrt{n}}\left(\begin{matrix}ah_E n & \frac{bh/h_E}{h} \\ \frac{ch}{h_E}N & dh_E n     \end{matrix}\right) 
\end{align}
for $a,b,c,d\in\ZZ$.
Thus it is that each Atkin--Lehner coset of $\Gamma_0(Nh)$ naturally determines an Atkin--Lehner coset of $\Gamma_0(N\vert h)$. 

As explained in \S~1.3 of \cite{ferenbaugh1993}, the map $\Ex_{Nh}\to \Ex_{N/h}$ just described is generally not injective, but does become so when restricted to the subgroup 
composed of the exact divisors $E$ of $Nh$ with the property that if $p$ is a prime dividing $E$ then $p$ also divides $N/h$.
Here we denote this subgroup by $\Ex_{N|h}$. 

We now let $S$ be a subgroup of $\Ex_{N/h}$ and consider 
a homomorphism 
$\lambda:\Gamma_0(N\vert h) +S \to \CC^*$ that satisfies
\begin{align}\label{lambdahom}
\lambda(\gamma) = 
	\begin{cases}
		1 
		& 
		\text{if $\gamma\in W_E$ for $E\in \Ex_{N|h}$},\\
		\ex(-\frac{1}{h}) 
		& 
		\text{if }\gamma \in \Gamma_0(Nh)\left(\begin{smallmatrix} 1 & 1/h \\ 0 & 1 \end{smallmatrix}\right),\\
		\ex(\pm \frac{1}{h}) 
		& 
		\text{if }\gamma\in \Gamma_0(Nh)\left(\begin{smallmatrix} 1 & 0 \\ N & 1 \end{smallmatrix}\right), 
		 \end{cases}
\end{align}
where the sign in the third case of (\ref{lambdahom}) is $+$ if 
$N/h$ is in $S$, and $-$ if not.
The group $\Gamma_0(N|h)$ is generated by $\Gamma_0(Nh)$ together with the matrices 
$\left(\begin{smallmatrix} 1 & 1/h \\ 0 & 1 \end{smallmatrix}\right)$ and $\left(\begin{smallmatrix} 1 & 0 \\ N & 1 \end{smallmatrix}\right)$ according to \cite{MR1073672}, 
so a homomorphism 
$\lambda$ as in (\ref{lambdahom}) 
is unique if it exists. 

Necessary and sufficient conditions on $N$, $h$ and $S$ for the existence of 
$\lambda$ as in (\ref{lambdahom}) 
are given in Theorem 2.8 of \cite{ferenbaugh1993}. 
In the case that these conditions hold, and $S=\{1,n,n',\dots\}$, we take
\begin{gather}\label{eqn:Nverthplusnnprime}
	N|h+n,n',\dots,
\end{gather}
to mean the kernel of this homomorphism $\lambda$.
The groups (\ref{eqn:Nverthplusnnprime}) so defined are referred to as $n|h$-type groups, 
and we adopt the convention of writing $N|h+{}$ as a shorthand for (\ref{eqn:Nverthplusnnprime}) in case $S=\Ex_{N/h}$, and also write $N|h$ in the case that $S=\{1\}$.
An $n\vert h$-type group (\ref{eqn:Nverthplusnnprime}) is called Fricke if it contains the element 
\begin{gather}\label{eqn:Fricke}
\sqrt{\frac{h}{N}}\left(\begin{matrix}0 & -{1}/{h} \\ N & 0 \end{matrix}\right),
\end{gather} 
and non-Fricke otherwise.

We will find it useful in some places to work with metaplectic double covers of Fuchsian groups. The reason is that the weight $\frac12$ modular forms we consider, while transforming only projectively with respect to congruence subgroups $\Gamma_0(N)$, transform with respect to honest representations of their metaplectic double covers. 
To describe these covers, recall that $\SL_2(\RR)$ admits a $2$-fold central extension
\begin{align}\label{eqn:widetildeSL2R}
\widetilde{\SL}_2(\RR) := \left.\left\{ (\gamma,u)\,\right|\, \gamma\in \SL_2(\RR),\, u(\tau)^4{\rm d}(\gamma\tau)={\rm d}\tau     \right\},
\end{align}
where the $u$ in each pair $(\gamma,u)$ in (\ref{eqn:widetildeSL2R}) is a holomorphic function on $\HH$ (and is therefore a square root of $u(\tau)^2=c\tau+d$ when $(c,d)$ is the lower row of $\gamma$), and where the multiplication rule is
\begin{align}\label{eqn:widetildeSL2R-mlt}
(\gamma_1,u_1)(\gamma_2,u_2) = (\gamma_1\gamma_2,(u_1\circ\gamma_2)u_2).
\end{align}
We call this group (\ref{eqn:widetildeSL2R}--\ref{eqn:widetildeSL2R-mlt}) the metaplectic double cover of $\SL_2(\RR)$. 

For any Fuchsian group $\Gamma<\SL_2(\RR)$, we reserve the notation $\widetilde{\Gamma}$ for the inverse image of $\Gamma$ under the natural map $\widetilde{\SL}_2(\RR)\to\SL_2(\RR)$, and abuse this notation slightly so as to write $\widetilde{\Gamma}_0(N)$ and $\widetilde{\Gamma}_1(N)$ for the preimages of $\Gamma_0(N)$ and $\Gamma_1(N)$, respectively. We also write $\widetilde{\SL}_2(\ZZ)$ for the preimage of the modular group $\SL_2(\ZZ)=\Gamma_0(1)=\Gamma_1(1)$.

The Fuchsian groups which appear in generalized monstrous moonshine are distinguished by the fact that they are genus zero, in the following sense. Assume $\Gamma$ is commensurable with $\SL_2(\ZZ)$ (as described at the beginning of this section). Then the action of $\Gamma$ on $\HH$ extends naturally to $\QQ\cup\{i\infty\}$ (see e.g.\ Proposition 2.14 of \cite{Dun_ArthGrpsAffE8Dyn}),
and there is a naturally defined Riemann surface structure on
\begin{gather}\label{eqn:XGamma}
X_\Gamma:=\Gamma\backslash \HH\cup\QQ \cup \{i\infty\}
\end{gather}
(see e.g.\ \S~1.5 of \cite{ShiATAF}).
In general, a Fuchsian group $\Gamma$ is said to be genus zero if $X_\Gamma$ is isomorphic to the Riemann sphere. The genus zero $n\vert h$-type groups are classified in \cite{ferenbaugh1993}. We refer interested readers to op.\ cit.\ for a more thorough treatment.

\subsection{Modular Forms}\label{sec:pre-mod}

In this section we review the three main kinds of modular forms that appear in this work. Firstly we recall that 
a weakly holomorphic modular form of weight $0$ for a discrete group $\Gamma<\SL_2(\RR)$ is 
a $\Gamma$-invariant 
holomorphic function 
on $\HH$ 
that has at most exponential growth near the cusps of $\Gamma$. 
If $\Gamma$ is commensurable with $\SL_2(\ZZ)$, 
this may be formulated as the condition that
there exists a $C>0$ such that $f(\gamma\tau)=O(e^{C\Im(\tau)})$ as $\Im(\tau)\to \infty$, for all $\gamma\in \SL_2(\ZZ)$.
A weakly holomorphic modular form of weight $0$ is called a principal modulus, or Hauptmodul, if the 
induced map $\Gamma\backslash\HH\to \CC$ extends to an isomorphism 
$X_\Gamma\to \CC\cup\{i\infty\}$ (cf.\ (\ref{eqn:XGamma})) of Riemann surfaces. In particular, such a function exists if and only if $\Gamma$ is genus zero.

Secondly we define some spaces of modular forms that satisfy Kohnen's plus-space condition. 
For this let $D$ be congruent to $0$ or $1$ modulo $4$, and for $n$ coprime to $D$ define the Kronecker symbol $(\frac{D}{n})$ by requiring that 
$n\mapsto (\frac{D}{n})$ is completely multiplicative, 
and that $(\frac{D}{n})$ evaluates to $1$ or $-1$ when $n=p$ is prime, according as $D$ is or is not a square modulo $4p$, 
and by taking $(\frac{D}{-1})$ to be the sign of $D$.
Also, for $d$ odd define $\epsilon_d$ to be $1$ or $i$
depending on whether $d$ is $1$ or $3$ modulo $4$.
Then
$\theta(\tau) := 
\sum_{n}q^{n^2}$ 
satisfies the transformation 
rule 
\begin{align}\label{eqn:pre-mod:thetatransform}
\epsilon_d\left(\frac{c}{d_{}}\right)\theta\left(\frac{a\tau+b}{c\tau+d}\right)
\frac1{\sqrt{c\tau+d}}
= 
\theta(\tau),
\end{align}
for $\left(\begin{smallmatrix}a&b\\c&d\end{smallmatrix}\right)\in \Gamma_0(4)$,
where $\left(\frac{c}{d}\right)$ is the Kronecker symbol just defined, and 
$\sqrt{\cdot}$ is the branch of the square root function determined by requiring that $\sqrt{e^{2\pi it}}=e^{\pi i t}$ for $-\frac12 < t\leq \frac12$.

Now let $\chi:\Gamma_0(4mN) \to \CC^\ast$ be a character of $\Gamma_0(4mN)$ for some positive integers $m$ and $N$. 
We use $M^{\wh,+}_{\frac12,m}(N,\chi)$ to denote the space of 
holomorphic functions $f:\HH\to \CC$ 
which satisfy
\begin{align}
\chi(\gamma)\frac{f(\gamma\tau)}{f(\tau)} = \frac{\theta(\gamma\tau)}{\theta(\tau)} 
\end{align}
for $\gamma\in\Gamma_0(4mN)$, and whose Fourier developments 
$f(\tau) = \sum_{D}c(D)q^D$
are such that $c(D) = 0$ unless $D$ is a square modulo $4m$. 
We also require that $u(\tau)^{-1}f(\gamma\tau)=O(e^{C\Im(\tau)})$ as $\Im(\tau)\to \infty$, for some $C>0$, for every $(\gamma,u)\in \widetilde{\SL}_2(\ZZ)$. 
The  condition that $c(D)=0$ unless $D$ is a square modulo $4$
is referred to as the \emph{Kohnen plus-space condition} \cite{MR575942,MR660784}, which is why we have placed a ``plus'' in the superscript of $M_{\frac12,m}^{\wh,+}(N,\chi)$.
We write $M^{+}_{\frac12,m}(N,\chi)$ for the subspace of $M^{\wh,+}_{\frac12,m}(N,\chi)$ composed of forms $f$ such that
$u(\tau)^{-1}f(\gamma\tau)$ remains bounded as $\Im(\tau)\to \infty$, for every $(\gamma,u)\in \widetilde{\SL}_2(\ZZ)$, and call these forms holomorphic.

The Thompson moonshine of \cite{Harvey:2015mca} was originally formulated in terms of weight $\frac12$ modular forms satisfying Kohnen's plus-space condition, but it was emphasized in \cite{pmo} that the theory can alternatively be expressed in terms of skew-holomorphic Jacobi forms, or equivalently, in terms of vector-valued modular forms transforming with respect to a particular Weil representation of $\widetilde{\SL}_2(\ZZ)$. 
This motivates the third kind of modular form we consider in this work, being that which transforms like the vector-valued theta series $\theta_m^0=(\theta_{m,r}^0)$, for some positive integer $m$, where the component functions
\begin{gather}\label{eqn:pre-mod:thetamr0}
	\theta_{m,r}^0(\tau):=\sum_{s\equiv r\xmod 2m} q^{\frac{s^2}{4m}},
\end{gather}
indexed by $r\xmod 2m$,
are the Thetanullwerte. 
To put this precisely we use the $\theta_{m,r}^0$ to define a $2m$-dimensional representation $\varrho_m$ of $\widetilde{\SL}_2(\ZZ)$ by requiring that
\begin{gather}\label{eqn:pre-mod:varrhomgammauthetam0}
	u(\tau)^{-1}\varrho_m(\gamma,u)\theta_m^0(\gamma\tau)
	=\theta^0_m(\tau)
\end{gather}
for every $(\gamma,u)\in \widetilde{\SL}_2(\ZZ)$.
Next we fix positive integers $m$ and $N$ and suppose that $\chi:\Gamma_0(N) \to \CC^\ast$ is a character of $\Gamma_0(N)$.
Then we write $\textsl{V}^{\wh}_{\frac12,m}(N,\chi)$ for the space of $2m$-vector-valued holomorphic functions $F=(F_r)$ on $\HH$ that satisfy
\begin{gather}\label{eqn:pre-mod:varrhomgammauF}
	u(\tau)^{-1}\varrho_m(\gamma,u)F(\gamma\tau)
	=\chi(\gamma)F(\tau)
\end{gather}
for all $(\gamma,u)\in \widetilde{\Gamma}_0(N)$, and are also such that the left-hand side of (\ref{eqn:pre-mod:varrhomgammauF}) is $O(e^{C\Im(\tau)})$ as $\Im(\tau)\to \infty$, 
for some $C>0$, when we take $(\gamma,u)$ to be an arbitrary element in $\widetilde{\SL}_2(\ZZ)$.
We abbreviate $\textsl{V}^{\wh}_{\frac12,m}=\textsl{V}^{\wh}_{\frac12,m}(1,1)$, and note that 
the component $F_r$ of a function $F=\big(F_r\big)$ in $\textsl{V}^{\wh}_{\frac12,m}(N,\chi)$ 
admits a Fourier expansion of the form 
\begin{align}\label{eqn:pre-mod:Frtau}
F_r(\tau) = \sum_{{D\equiv r^2\xmod 4m}} C_F(D,r) q^{\frac{D}{4m}},
\end{align}
where $C_F(D,r)=0$ for $-D$ sufficiently large.

The vector-valued modular forms in $\textsl{V}^\wh_{\frac12,m}(N,\chi)$ are closely related to Kohnen plus-space modular forms. 
To explicate this define a map $\iota_k$ from $\Gamma_0(kN)$ to $\Gamma_0(N)$
by setting
\begin{align}\label{eqn:iotak}
\iota_k\left(\begin{matrix} a & b \\ c & d\end{matrix}\right) 
:=\left(\begin{matrix} a & kb \\ c/k & d \end{matrix}\right).
\end{align}
Then 
we obtain a linear map 
\begin{gather}\label{eqn:prebreveF}
\textsl{V}^{\wh}_{\frac12,m}(N,\chi)\to M^{\wh,+}_{\frac12,m}(N,\chi\circ\iota_{4m}),
\end{gather} 
denoted $F\mapsto \breve{F}$, 
by setting
\begin{align}\label{eqn:breveF}
\breve{F}(\tau) := \sum_{r\xmod 2m} F_r(4m\tau).
\end{align}
Moreover, if $m$ is not composite then this map (\ref{eqn:prebreveF}--\ref{eqn:breveF}) is an isomorphism. 
(The $m=N=1$ case of this statement is contained in the results of \cite{eichler1985theory}. 
See also Example 2.4 in \cite{Borcherds:1996uda}.)

It will be useful to have at hand certain refinements of the spaces $\textsl{V}^{\wh}_{\frac12,m}(N,\chi)$. 
For this set 
\begin{align}\label{eqn:Om}
    O_m := \{a\xmod 2m \mid a^2\equiv 1\xmod 4m\},
\end{align} 
regard $O_m$ as a group under multiplication modulo $2m$, and let $\alpha:O_m\to\mathbb{C}^\ast$ be a character. 
We define $\textsl{V}^{\wh,\alpha}_{\frac12,m}(N,\chi)$ to be the subspace of $\textsl{V}^{\wh}_{\frac12,m}(N,\chi)$ composed of those functions $F=(F_r)$ such that 
\begin{align}\label{eqn:Fartau}
    F_{ar} = \alpha(a)F_r
\end{align}
for each $r\xmod 2m$ and each $a\in O_m$. 
We abbreviate $\textsl{V}^{\wh,\a}_{\frac12,m}=\textsl{V}^{\wh,\a}_{\frac12,m}(1,1)$.

We will construct functions in the spaces $\textsl{V}^{\wh,\alpha}_{\frac12,m}(N,\chi)$ by projection. 
To put this precisely note that the space of matrices that commute with $\varrho_m$ (cf.\ (\ref{eqn:pre-mod:varrhomgammauthetam0})) admits as a basis the matrices $\{\Omega_m(n)\}_{n\vert m}$ whose entries are 
\begin{align}\label{eqn:omegamatrices}
\Omega_m(n)_{r,r'} := \begin{cases}
1 & \text{if }r\equiv-r' \xmod 2n \text{ and }r\equiv r'  \xmod 2m/n, \\
0 & \text{otherwise}
\end{cases}
\end{align}
(see \cite{Gepner:1986hr}).
For $\a$ a character of $O_m$ we define the projection matrix $\Omega^\alpha_m$ by setting
\begin{align}\label{eqn:Omegaalpham}
\Omega^\alpha_m := \frac{1}{| \Ex_m|}\sum_{n\in\Ex_m}\alpha(a(n))\Omega_m(n),
\end{align}
where in each summand $a(n)$ is the unique $a\xmod 2m$ such that $a\equiv -1\xmod 2n$ and $a\equiv 1\xmod \frac{2m}{n}$. 
Then multiplication by $\Omega_m^\a$ defines a projection $\textsl{V}^{\wh}_{\frac12,m}(N,\chi)\to\textsl{V}^{\wh,\a}_{\frac12,m}(N,\chi)$.
(See \cite{Cheng:2016klu} for a related discussion in the context of optimal mock Jacobi forms, and \cite{pmo} for a discussion in the context of skew-holomorphic Jacobi forms.)

We conclude this section by commenting that a modular form in $M^{\wh,+}_{\frac12,1}(N,\chi)$ is uniquely determined by the singular terms in its $q$-expansion, up to a form in the space $M^{+}_{\frac12,1}(N,\chi)$. 
Assuming that $\ker\chi = \Gamma_0(Nh)$ for some $h$, as will be the case for all characters $\chi$ in this paper, we can use the fact that $M^+_{\frac12,1}(N,\chi)\subset M^+_{\frac12,1}(Nh,1)$ to express any such ambiguity as a linear combination of the theta series $\{\theta(k^2\tau)\}_{k^2\vert Nh}$, according to the main result of \cite{serrestark}. (There is an analogous statement for the space $M^{\wh,+}_{\frac12,m}(N,\chi)$, but we will not need it in what follows.) 
Therefore, when convenient we may specify a modular form in $M^{\wh,+}_{\frac12,1}(N,\chi)$ by the data of its poles, as well as some low order terms in its $q$-expansion (so as to fix the contributions from theta series $\theta(k^2\tau)$).

\subsection{CM Points}\label{sec:pre-cmp}

In this section we make some definitions which we will use in \S~\ref{sec:res-tra} to give conjectural expressions for most of the McKay--Thompson series of $\breve{W}_{\mathrm{3C}}$ (cf.\ 
(\ref{eqn:breveW3C}--\ref{eqn:breveF3Cg})) as generating functions for traces of singular moduli. 
More specifically, we prepare in this section to give expressions for the Fourier coefficients $\breve{C}_{{\rm 3C},g}(D)$ of the $\breve{F}_{{\rm 3C},g}$ as sums of values of the weight $0$ functions $\widetilde{T}_{\mathrm{3C},g}^3$ over CM points.

We start by reviewing some basics of integer-coefficient binary quadratic forms, closely following the conventions of \S~I of \cite{GKZ87}. 
To begin say that $D\in\ZZ$ is a discriminant if $D\equiv 0,1\xmod 4$, and say that such a $D$ is a fundamental discriminant if it is odd and square-free, or if $D=4d$ where $d$ is square-free and $d\equiv 2,3\xmod 4$.
For $A,B,C\in \ZZ$ we let $Q$ stand for the binary quadratic form $Q(x,y) = Ax^2+Bxy+Cy^2$ and call $D=B^2-4AC$ the discriminant of $Q$, and restrict attention to those $Q$ for which $D$ is negative.
We write $\a_Q$ for the CM point associated to $Q$. That is, we define $\a_Q$ to be the unique root of $Q(x,1)=Ax^2+Bx+C$ in the upper half-plane. 

We denote the space of binary quadratic forms of discriminant $D<0$ with $A>0$ by $\mathcal{Q}_D$, and for $N$ a positive integer set 
$\mathcal{Q}^{(N)}_D = 
\left\{ Q \in \mathcal{Q}_D\mid A\equiv 0  \xmod N  \right\}$. 
This set $\mc{Q}^{(N)}_D$ is acted on by 
$\Gamma_0(N)+$ according to the rule that
\begin{align}\label{eqn:pre-cmp:Qgammaxy}
(Q\vert \gamma)(x,y) := \frac{1}{n}Q(an x + by, cNx + dny), \ \ \ \ \ \  \gamma = \frac{1}{\sqrt{n}}\left(\begin{matrix} an & b \\ cN & dn \end{matrix}\right)\in \Gamma_0(N){+},
\end{align}
and the orbit space $\mc{Q}^{(N)}_D/\Gamma_0(N)$ 
is finite. 

Now suppose that $D_0$ is a negative fundamental discriminant 
and 
$D$ is a positive discriminant. 
Applying Proposition 1 of \S~I.2 of op.\ cit.\ we may define the 
genus character 
$\chi_{D_0}$ on $\mathcal{Q}_{D_0D}$
by setting
\begin{gather}\label{eqn:pre-cmp:chiD0N}
\chi_{D_0}(Q) := \left(
\frac{D_0'}{A}
\right)
\left(
\frac{D_0''}{C}
\right)
\end{gather}
in case $Q(x,y)=Ax^2+Bxy+Cy^2$,
when there exist discriminants $D_0'$, $D_0''$ 
such that $D_0=D_0'D_0''$ 
and 
$\gcd(D_0',A)=\gcd(D_0'',C)=1$, and by setting 
$\chi_{D_0}(Q) := 0$ 
when no such $D_0'$, $D_0''$ 
exist.

We will actually only be interested in the genus character $\chi_{D_0}$ with $D_0=-3$ in this work. 
For this case, and only for this case, we define a twist $\widetilde{\chi}_{-3}$ of 
$\chi_{-3}$, but restricted to $\mc{Q}^{(3)}_{-3D}$ for positive discriminants $D$, by setting
\begin{gather}\label{eqn:pre-cmp:widetildechiD03}
\widetilde{\chi}_{-3}(Q) := 
\begin{cases}
\left(
\frac{-3}{A}
\right)
&\text{ if $(3,A)=1$,}\\
\left(\frac{-3}{C}\right)
&\text{ if $A\equiv 0 \xmod 3$ and $(3,C)=1$ and $D\equiv 0,1,3,4,6,9\xmod 12$,}\\
-\left(\frac{-3}{C}\right)
&\text{ if $A\equiv 0 \xmod 3$ and $(3,C)=1$ and $D\not\equiv 0,1,3,4,6,9\xmod 12$,}\\
0&\text{ if $A\equiv C\equiv 0\xmod 3$,}
\end{cases}
\end{gather}
in case $Q(x,y)=3Ax^2+Bxy+Cy^2$ and $B^2-12AC=-3D$.
Note that this definition recovers the generalized genus character on $\mc{Q}^{(3)}_{-3D}$ that is denoted $\chi_{-3}$ in op.\ cit. (and denoted $\chi^{(3)}_{-3}$ in \cite{pmp}), when $D$ is a square modulo $12$ (i.e.\ $D\equiv 0,1,4,9\xmod 12$).
Our prescription (\ref{eqn:pre-cmp:widetildechiD03}) extends this generalized genus character to arbitrary positive discriminants $D$.

The primary role of binary quadratic forms here will be to define distinguished divisors in the upper half-plane. 
To put this precisely let $\Gamma=\Gamma_0(N)+S$ for some $S<\Ex_m$,
 let $D_0$ be a negative fundamental discriminant that is a square modulo $4N$, 
 and assume that $\chi$ is a $\Gamma$-invariant function on $\mc{Q}^{(N)}_{D_0D}$, for 
 $D$ a positive discriminant.
Then for $T$ a $\Gamma$-invariant modular function on $\HH$ 
we define the following {trace} of singular moduli,
\begin{align}\label{sec:pre-cmp:TrGammachiD0}
\Tr_{D_0,\Gamma}(T,D,\chi) := \sum_{Q\in \mc{Q}^{(N)}_{D_0D}/\Gamma} \chi(Q)\frac{T(\a_Q)}{\left|\overline{\Gamma}_Q\right| },
\end{align}
where $\left|\overline{\Gamma}_Q\right|$ is the order of the stabilizer of $Q$ when acted on by $\overline{\Gamma} = \Gamma/\{\pm \Id\}$.
For convenience we set 
$\Tr_{D_0,\Gamma}(T,D,\chi) :=0$ 
in case $D$ is not a 
positive discriminant.

\subsection{Rademacher Sums}\label{sec:pre-rad}

We will employ Rademacher sums in 
\S\S~\ref{app:ser-thm} and \ref{app:ser-gen}.
The particular expressions we use appear to be new in some cases, but their validity may be verified by standard techniques, such as appear in e.g.\ \cite{Cheng:2012qc,Duncan:2009sq,2014arXiv1406.0571W}.
For this reason we do not indulge in detailed derivations.

Our interest will exclusively lie in functions belonging to the spaces $\textsl{V}^{\wh}_{\frac12,m}(N,\chi)$ of \S~\ref{sec:pre-mod}. 
The input data is as follows.

\begin{enumerate}
	\item A weight $w$ (which we always take to be $\frac{1}{2}$).
	\item An index $m$ which specifies a Weil representation $\varrho_m$.
	\item A level $N$ which determines the congruence subgroup $\widetilde{\Gamma}_0(N)$ with respect to which the Rademacher sum will covariantly transform.
	\item A character $\chi:\Gamma_0(N)\to\CC^\ast$ with kernel containing $\Gamma_0(Nh)$ for some integer $h$. Together with $\varrho_m$ it determines the modular transformation properties.
	\item An exact divisor $n\in\Ex_N$ which places a pole of the function at the cusp $W_n\cdot i\infty$ where $W_n$ is an Atkin--Lehner coset of $\Gamma_0(N)$.
	\item A vector $\vec{\mu}$ of length $2m$ which specifies the components in which poles will appear, as well as their degrees.
\end{enumerate}
The output is a vector-valued function 
$R^{[\vec{\mu}],(n)}_{m,N,\chi,w}=(R^{[\vec{\mu}],(n)}_{m,N,\chi,w,r})$ (which we abbreviate to $R^{[\vec{\mu}],(n)}_{m,N,\chi}=(R^{[\vec{\mu}],(n)}_{m,N,\chi,r})$ since we are taking $w=\frac{1}{2}$) which satisfies 
\begin{align}
\chi(\gamma)R^{[\vec{\mu}],(n)}_{m,N,\chi}(\gamma\tau) = u(\tau)\varrho_m(\gamma,u)R^{[\vec{\mu}],(n)}_{m,N,\chi}(\tau),   \ \ \ \ \ \ (\gamma,u)\in \widetilde{\Gamma}_0(N).
\end{align}

Set $\psi(\gamma):=\chi(\gamma)\varrho_m({\gamma},\sqrt{c\tau+d})^{-1}$ when $\gamma=\left(\begin{smallmatrix}*&*\\c&d\end{smallmatrix}\right)\in\Gamma_0(N)$ and define the numbers $0\leq \alpha_r<1$ for $r=0,\dots,2m-1$ through the equation $\ex(\alpha_r) = \psi_{rr}(T)$. We will define 
$R^{[\vec{\mu}],(n)}_{m,N,\chi}$ by specifying the coefficients of the Fourier expansions
\begin{align}\label{eqn:rademacherqexpansion}
R^{[\vec{\mu}],(n)}_{m,N,\chi,r}(\tau) = \epsilon_{\vec{\mu},n,r} q^{\mu_r}+ \sum_{\nu \in \ZZ^+-\alpha_r}c^{[\vec{\mu}],(n)}_{m,N,\chi,r}(\nu)q^\nu
\end{align}
of its component functions $R^{[\vec{\mu}],(n)}_{m,N,\chi,r}$.
The symbol $\epsilon_{\vec{\mu},n,r}$ in (\ref{eqn:rademacherqexpansion}) is defined by setting $\epsilon_{\vec{\mu},n,r} := 1$ if $n=1$ and $\mu_r\neq 0$, and $\epsilon_{\vec{\mu},n,r}:=0$ otherwise. 

To specify the coefficients $c^{[\vec{\mu}],(n)}_{m,N,\chi,r}(\nu)$ in (\ref{eqn:rademacherqexpansion}) we first define the ``effective singular vector'' 
$\vec{\mu}^{(n)}=(\vec{\mu}^{(n)}_r)$ to have components 
\begin{gather}\label{eqn:effectivesingularvector}
\mu^{(n)}_r:=-M+n(\mu_r-\floor*{\mu_r}),
\end{gather} 
where $M$ is an integer chosen so as to make $\ceil*{\mu_r}-1<\mu_r^{(n)}\leq \ceil*{\mu_r}$, and $\floor{x}$  and $\ceil{x}$ denote the floor and ceiling of $x$, respectively. 
Next define the ``effective multiplier system'' $\psi^{(n)}=(\psi^{(n)}_{rs})$, a matrix-valued function on $W_n$, as
\begin{align}\label{eqn:effectivemultipliersystem}
\psi^{(n)}\left(\gamma\right)  := n^{\frac{1}{4}}\chi(A^{-1}\gamma)\varrho_m^{-1}\left( \left(\begin{matrix} an & b \\ c/n & d \end{matrix}\right), \sqrt{\frac{c}{n}\tau+d} \right), \ \ \ \ 
\gamma =\frac{1}{\sqrt{n}} \left(\begin{matrix} an & b \\ c & dn \end{matrix}\right)\in W_n,
\end{align}
where $A$ is a fixed coset representative for $W_n$. 

The choice of $A$ in (\ref{eqn:effectivemultipliersystem}) entails an ambiguity by an overall $h$-th root of unity, where $h$ is the order of $\chi$. 
If $n=1$ we choose $A$ to be the identity, and the choice of $A$ is irrelevant when the character $\chi$ is trivial. 
In the few cases that the ambiguity actually arises, we make a specific choice (see \S~\ref{app:ser-thm}).

Now for 
$\gamma =\frac{1}{\sqrt{n}} \left(\begin{smallmatrix} an & b \\ c & dn \end{smallmatrix}\right)\in W_n$ we put
\begin{align}
\begin{split}\label{eqn:RadKB}
K_{\gamma,\psi}^{(n)}(\vec{\mu},\nu)_{rs} &:= \ex\Big(\nu \frac{dn}{c}\Big)\psi_{rs}^{(n)}(\gamma)\ex\Big(\mu^{(n)}_r\frac{an}{c}\Big), \\
B^{(n)}_{\gamma,\frac{1}{2}}(\vec{\mu},\nu)_r &:= \ex\Big(-\frac{1}{8}\Big)\left(-\frac{\mu_r^{(n)}}{\nu}   \right)^{\frac{1}{4}} \frac{2\pi \sqrt{n}}{c} I_{\frac{1}{2}}\left(\frac{4\pi \sqrt{-n\mu_r^{(n)}\nu}}{c} \right),
\end{split}
\end{align}
where $I_{\alpha}(x)$ is the modified Bessel function of the first kind. For $\alpha=\frac12$, which is the only case we will make use of in this paper, it is defined away from $x=0$ by 
\begin{align}
I_{\frac12}(x) = \sqrt{\frac{2}{\pi x}}\sinh(x).
\end{align}

With the conventions (\ref{eqn:effectivesingularvector}--\ref{eqn:RadKB}) in place the Fourier coefficients $c^{[\vec{\mu}],(n)}_{m,N,\chi,r}(\nu)$ in (\ref{eqn:rademacherqexpansion}) can now be expressed as
\begin{align}
c^{[\vec{\mu}],(n)}_{m,N,\chi,r}(\nu) := 
\lim_{K\to\infty}\sum_{s \xmod 2m} \sum_{\Gamma_\infty \backslash (W_n^\times)_K / \Gamma_\infty} B_{\gamma,\frac{1}{2}}^{(n)}(\vec{\mu},\nu)_s K^{(n)}_{\gamma,\psi}(\vec{\mu},\nu)_{sr} 
\end{align}
in the case that $\nu\neq 0$, and
\begin{align}
c^{[\vec{\mu}],(n)}_{m,N,\chi,r}(0) 
:= \frac{1}{2}\ex(-\tfrac{1}{8})\lim_{K\to\infty}\sum_{s \xmod 2m}\frac{(2\pi)^{\frac{3}{2}}\sqrt{-\mu_s^{(n)}}}{\Gamma(\frac{3}{2})}\sum_{\Gamma_\infty \backslash (W_n^\times)_K / \Gamma_\infty} \left(\frac{\sqrt{n}}{c}\right)^{\frac{3}{2}}K_{\gamma,\psi}^{(n)}(\vec{\mu},0)_{sr}
\end{align}
when $\nu = 0$. In the above, $\Gamma(z)$ is the Gamma function, $\Gamma_\infty := \left\{ T^n \mid n \in \ZZ\right\}$, and 
\begin{align}
(W_n^\times)_K := \left.\left\{ \left(\begin{matrix} a&b \\ c & d\end{matrix}\right)\in W_n \,\right|\,  0 < |c| < K \right\}.
\end{align}
The sum over double cosets can be made more explicit, leading to the formulae 
\begin{align}
\begin{split}
c^{[\vec{\mu}],(n)}_{m,N,\chi,r}(\nu) &= 
\sum_{s \xmod 2m} \sum_{\substack{c>0 \\ c\equiv 0  \xmod N}}\sum_{\substack{0\leq a < c/n\\ \gcd(an,c/n) = 1}} 
B_{\gamma,\frac{1}{2}}^{(n)}(\vec{\mu},\nu)_s 
K^{(n)}_{\gamma,\psi}(\vec{\mu},\nu)_{sr}  \ \ \ \ \ \ \ \  (\nu\neq 0),\\
c^{[\vec{\mu}],(n)}_{m,N,\chi,r}(0) &= 
\frac{1}{2}\ex(-\tfrac{1}{8})\sum_{s \xmod 2m}\frac{(2\pi)^{\frac{3}{2}}\sqrt{-\mu_s^{(n)}}}{\Gamma(\frac{3}{2})}\sum_{\substack{c>0 \\ c\equiv 0  \xmod N}}\sum_{\substack{0\leq a < c/n \\ \gcd(an,c/n) = 1}} 
\left(\frac{\sqrt{n}}{c}\right)^{\frac{3}{2}}
K_{\gamma,\psi}^{(n)}(\vec{\mu},0)_{sr}.
\end{split}
\end{align}

If $\Gamma$ is a Fuchsian group of the form $\Gamma = \Gamma_0(N)+S$ with $S$ a subgroup of $\Ex_N$ then we put 
\begin{gather}\label{eqn:unnormalizedrademacher}
\begin{split}
R^{[\vec{\mu}]}_{m,\Gamma,\chi} 
&:= \sum_{n\in S} R^{[\vec{\mu}],(n)}_{m,N,\chi},\\
c^{[\vec{\mu}]}_{m,\Gamma,\chi,r}(\nu)
&:= \sum_{n\in S} c^{[\vec{\mu}],(n)}_{m,N,\chi,r}.
\end{split}
\end{gather}
We will also use notation like e.g.\ $3+\bar{3}$ to denote the combination $R^{[\vec{\mu}]}_{m,3+\bar{3},\chi}=R^{[\vec{\mu}],(1)}_{m,3,\chi}-R^{[\vec{\mu}],(3)}_{m,3,\chi}$, i.e.\ the bar indicates that the Rademacher sum attached to the exact divisor $3$ should be subtracted rather than added in the sum. (Cf.\ the notation of \cite{ferenbaugh1993}.)

%---------------------------------------------------------------------------------------%
\section{Results}\label{sec:res}
%---------------------------------------------------------------------------------------%

In \S~\ref{sec:res-pro} we recall the main result of \cite{pmp}. Then in \S~\ref{sec:res-ava} we apply this result to Thompson moonshine phenomena in weight one-half and in weight zero. In particular, we produce two new avatars of Thompson moonshine: the \emph{weight one-half $\mathrm{3C}$-generalized Thompson moonshine module }${W}_{\mathrm{3C}}$ and the \emph{weight zero penumbral Thompson moonshine module }$V^{(-3,1)}$, which are related to the weight zero 3C-generalized Thompson module $V^\natural_{\mathrm{3C}}$ and the weight one-half penumbral Thompson module ${W}^{(-3,1)}$ respectively through our construction $\SQ$.
Specifically, $V^{(-3,1)} = \SQ({W}^{(-3,1)})$ and $(V^\natural_{\mathrm{3C}})^{\otimes 3} = \SQ({W}_{\mathrm{3C}})$. 
In \S~\ref{sec:res-tra} we present conjectural expressions for (most of) the McKay--Thompson series of the $\Th$-module $W_{\rm 3C}$ in terms of traces of singular moduli associated to the McKay--Thompson series of $V_{\rm 3C}^\natural$, and thereby take a step towards an inverse to $\SQ$ in the 3C case.
Finally, in \S~\ref{sec:res-gen} we offer some evidence that the story we have told for the Thompson group will generalize.

\subsection{Products}\label{sec:res-pro}

Here we recall the defining features of the construction $\SQ$, whose validity is established in \cite{pmp}.
To formulate these features let $G$ be a finite group, 
and let $W$ be a virtual graded $G$-module with a grading of the form 
\begin{gather}\label{eqn:SQ-BP:W}
	W=\bigoplus_{r\xmod 2m}\bigoplus_{D\equiv r^2\xmod 4m}W_{r,\frac{D}{4m}}
\end{gather}
for some positive integer $m$. (See \S~3.4 of op.\ cit.\ for a discussion of virtual graded $G$-modules.)
Define the
McKay--Thompson series 
$F^W_g=(F^W_{g,r})$
associated to $W$ and $g\in G$ by setting
\begin{gather}\label{eqn:SQ-BP:FWgr}
	F^W_{g,r}(\tau):=
	\sum_{D\equiv r^2\xmod 4m} \tr(g|W_{r,\frac{D}{4m}})q^{\frac{D}{4m}}
\end{gather}
for $r\xmod 2m$.
With this notation say that $W$ as in (\ref{eqn:SQ-BP:W}) is a {weakly holomorphic (virtual graded) $G$-module of weight $\frac12$ and index $m$} 
if for each $g\in G$ there exists a positive integer multiple $N_g$ of $o(g)$ such that $\frac{N_g}{o(g)}$ divides $o(g)$, and 
\begin{gather}\label{eqn:SQ-BP:FWgnVwh}
F^W_{g^n}\in V^\wh_{\frac12,m}\left(\tfrac{N_g}{n}\right)
\end{gather}
whenever $n$ is a divisor of $N_g$. Also,
say that such a $G$-module $W$ is rational if the coefficients 
\begin{gather}\label{eqn:SQ-BP:CWgDr}
C^W_{g}(D,r):=\tr(g|W_{r,\frac{D}{4m}})
\end{gather}
of the McKay--Thompson series $F^W_g$ of (\ref{eqn:SQ-BP:FWgr}) are rational integers, 
for all $g\in G$ and $D,r\in \ZZ$.

To a weakly holomorphic $G$-module $W$ of weight $\frac12$ and index $m$ in the above sense we attach a generalized class number $H^W$, following \cite{pmp}, by setting
\begin{gather}\label{eqn:HW}
	H^W
	:=
	\sum_{r\xmod 2m}\sum_{\substack{D\equiv r^2\xmod 4m\\D\leq 0}}
	C^W(D,r)H_m(D,r)
\end{gather}
where 
$C^W(D,r)=C^W_e(D,r)$ (cf.\ (\ref{eqn:SQ-BP:CWgDr})) and $H_m(D,r)$ is 
the coefficient of $q^{-\frac{D}{4m}}$ in the Fourier expansion of $G_{m,r}$, 
for $G_m=(G_{m,r})$ 
as defined in \S~9 of \cite{Borcherds:1996uda}. In what follows we will mostly be concerned with the case that $m=1$, whereby $H_m(D,r)=H_1(D,D)$ is the Hurwitz class number of $D$, as defined, e.g.\ in 
\cite{ZagierTSM}.
The generating function $\breve{G}_1(\tau):=G_{1,0}(4\tau)+G_{1,1}(4\tau)$ (cf.\ (\ref{eqn:breveF})) in this case satisfies
\begin{gather}\label{eqn:breveG1}
	\breve{G}_1(\tau)=\sum_{D\leq 0}H_1(D,D)q^D=-\frac1{12}+\frac13q^3+\frac12q^4+q^7+q^8+\dots.
\end{gather}

Now for $W$ a rational weakly holomorphic $G$-module of weight $\frac12$ and index $m$ we consider the product 
\begin{align}\label{eqn:res-pro:PsiWg}
\Psi^W_g(\tau) := q^{-H}\exp\left(-\sum_{n>0}\sum_{k>0} C^W_{g^k}(n^2,n)\frac{q^{nk}}{k}   \right)
\end{align}
for each $g\in G$ (with $H=H^W$). 
A priori it is not clear that this makes sense, 
but the fact that the right-hand side of (\ref{eqn:res-pro:PsiWg}) converges for $\Im(\tau)$ sufficiently large, and extends by analytic continuation to a holomorphic function on $\HH$, is established in \cite{pmp} by applying the singular theta lift developed in \cite{Harvey:1995fq,Borcherds:1996uda}.
Henceforth it is 
this analytically continued function that we have in mind when we write $\Psi^W_g$.

Next define integers $v_b(g|W_{0,0})$ for $b>0$ and $g\in G$ by requiring that
\begin{gather}\label{eqn:res-pro:pigW00}
	\pi(g|W_{0,0}) := \prod_{b>0}b^{v_b(g|W_{0,0})}
\end{gather}
is the Frame shape defined by the action of $g$ on $W_{0,0}$ (cf.\ (\ref{eqn:SQ-BP:W})), and set
\begin{gather}\label{eqn:SQ-BP:etaWg}
	\eta^W_g(\tau) := \prod_{b>0}\eta(b\tau)^{2v_b(g|W_{0,0})}.
\end{gather}
(See e.g.\ \S~3.4 of \cite{pmp} for background on Frame shapes.)
Then it develops (see \S~4.4 of \cite{pmp}) that the function $T^W_g$, defined for $g\in G$ by setting
\begin{gather}\label{eqn:SQ-BP:TWg}
	T^W_g:=\frac{\Psi^W_g}{\eta^W_g},
\end{gather}
is a weakly holomorphic modular form of weight $0$ with level for each $g$, 
and we may say that the purpose of the construction $\SQ$ is to provide a virtual graded $G$-module that realizes these functions (\ref{eqn:SQ-BP:TWg}) as its McKay--Thompson series. 
To put this precisely 
consider 
a virtual 
$\ZZ$-graded $G$-module 
$V=\bigoplus_nV_n$,
and say that such an object is weakly holomorphic of weight $0$ if there exists a constant $h$ such that 
\begin{gather}\label{eqn:SQ-BP:fVg}
f^V_g(\tau):=
	\sum_{n}
	\tr\left(\left.g\right|V_{n}\right)q^{n-h}
\end{gather} 
is a weakly holomorphic modular form of weight $0$, with level depending upon $g$, for each $g\in G$.
Then the main result of \cite{pmp} is the following.
\begin{thm}\label{thm:SQ}
Let $G$ be a finite group, 
and let $W$ be a rational weakly holomorphic $G$-module of weight $\frac12$ and index $m$, for some positive integer $m$.
Then there exists a unique weakly holomorphic 
$G$-module $V=V^W$ of weight $0$ such that 
	$
	f^V_g
	= 
	T^W_g 
	$
for all $g\in G$.
\end{thm}

Following \cite{pmp} we define $\SQ(W)$ to be the weakly holomorphic $G$-module $V^W$ whose existence and uniqueness is asserted by Theorem \ref{thm:SQ},
\begin{gather}\label{eqn:SQ-BP:SQW}
	\SQ(W) := V^W.
\end{gather}

\subsection{Avatars}\label{sec:res-ava}

As we have stated in \S~\ref{sec:int-com}, a main motivation for this paper is the comparison of two previously known examples of Thompson moonshine. 
The first, denoted $V^\natural_{\mathrm{3C}}$, has been known since at least 1980 \cite{queen}, and is a special case of generalized monstrous moonshine; 
its McKay--Thompson series are the weight $0$ principal moduli ${T}_{\mathrm{3C},g}$ described in \S~\ref{app:ser-thz}.
The second, 
denoted $\breve{W}^{(-3,1)}$ in \S~\ref{sec:int}, 
was introduced in \cite{Harvey:2015mca}, and 
is a special case of penumbral moonshine; its McKay--Thompson series 
are distinguished as Rademacher sums of weight $\frac12$ and are described in
\S~\ref{app:ser-thm}.

In this section we compare these two $\Th$-modules using the construction $\SQ$ developed in \cite{pmp}. 
To do this we first need to introduce a counterpart 
\begin{gather}\label{eqn:WtobreveW}
W\mapsto \breve W
\end{gather}
 for $G$-modules, to the operation $F\mapsto \breve F$ on vector-valued modular forms, defined by (\ref{eqn:prebreveF}--\ref{eqn:breveF}). 
The input for this operation is a virtual graded $G$-module $W$ as in (\ref{eqn:SQ-BP:W}). The output is the virtual graded $G$-module $\breve W=\bigoplus_D \breve W_D$ where
\begin{gather}\label{eqn:breveWD}
	\breve{W}_D :=
	\bigoplus_{\substack{r\xmod 2m\\D\equiv r^2\xmod 4m}}
	W_{r,\frac{D}{4m}},
\end{gather}
for $m$ as in (\ref{eqn:SQ-BP:W}). The connection to the operation $F\mapsto \breve{F}$ is that if $F^W_g=(F^W_{g,r})$ is the McKay--Thompson series arising from the action of $g\in G$ on $W$, as defined by (\ref{eqn:SQ-BP:FWgr}), then its image $\breve F^W_g$ under this map (\ref{eqn:prebreveF}--\ref{eqn:breveF})
satisfies 
\begin{gather}\label{eqn:breveFWg}
	\breve F^W_g(\tau) = \sum_D \tr(g|\breve{W}_D)q^{D}.
\end{gather}
That is to say, the operation $W\mapsto \breve W$ defined by (\ref{eqn:WtobreveW}--\ref{eqn:breveWD}) is such that the McKay--Thompson series associated to $\breve{W}$ are obtained by applying the map $F\mapsto \breve{F}$ of (\ref{eqn:breveF}) to the McKay--Thompson series associated to $W$.

In this section we are interested in the case that $m=1$, whereby the map $W\mapsto\breve{W}$ of (\ref{eqn:WtobreveW}--\ref{eqn:breveWD}) is an isomorphism. 
We henceforth take $W^{(-3,1)}$ to be the unique virtual graded $\Th$-module as in (\ref{eqn:SQ-BP:W}), with $m=1$, whose image $\breve{W}^{(-3,1)}$ under (\ref{eqn:WtobreveW}--\ref{eqn:breveWD}) agrees with the $\Th$-module denoted as such in \S~\ref{sec:int} (see \S~\ref{sec:int-woh}, and especially (\ref{eqn:breveW31})). In this way we recover the meaning of the notation $W^{(-3,1)}$ in \cite{pmo}, and also obtain a $\Th$-module to which the construction $\SQ$ of \cite{pmp} (see \S~\ref{sec:res-pro}) may be applied. So we apply it to $W^{(-3,1)}$, and on the strength of 
Theorem \ref{thm:SQ}
obtain a 
weakly holomorphic $\Th$-module
\begin{gather}\label{eqn:V-31}
	V^{(-3,1)}:=\SQ(W^{(-3,1)})
\end{gather}
of weight $0$, with associated McKay--Thompson series $f^V_g$ that admit product formulae $f^V_g=T^W_g$, where $V=V^{(-3,1)}$ and $W=W^{(-3,1)}$. 
To ease notation as we put this precisely let us write $T^{(-3,1)}_g$ in place of $T^W_g$, when $W=W^{(-3,1)}$. 
\begin{thm}[Weight zero avatar of penumbral Thompson moonshine]\label{thm:res-ava:wz}
There exists a weakly holomorphic $\Th$-module $V^{(-3,1)}$ of weight $0$ whose McKay--Thompson series are the functions $T^{(-3,1)}_g$. 
Moreover, we have
\begin{align}\label{eqn:res-ava:Tg31}
    T_g^{(-3,1)}
=    \widetilde{T}_{\mathrm{3C},g}^2
\end{align}
for $3$-regular $g\in\Th$.
\end{thm}

\begin{proof}
As we have explained, the first statement follows from Theorem \ref{thm:SQ} and the definition (\ref{eqn:V-31}) of $V^{(-3,1)}$.
It remains to verify that (\ref{eqn:res-ava:Tg31}) holds when $o(g)$ is coprime to $3$. 
For this it suffices to check the leading terms in the respective expansions as $\Im(\tau)\to \infty$, and the divisors on each side.

The divisor of the right-hand side of (\ref{eqn:res-ava:Tg31}) is known, as the function in question is the square of an explicitly determined principal modulus (see \S~\ref{app:ser-thz}). 
For the divisor on the left-hand side we apply 
Item 2 of Theorem 13.3 of \cite{Borcherds:1996uda}
to the image $\check{F}=\check{F}^W_g$ of the collection $\{F^{W}_{g^n}\}_{n|N}$ under the repackaging map
\begin{gather}\label{eqn:FWgntocheckFWg}
\left\{F^W_{g^n}\right\}_{n\vert N_g} \mapsto \check{F}^W_g=\left(\check{F}^W_{g,i,j,r}\right)
\end{gather} 
defined in \S~4.2 of \cite{pmp}, where $W=W^{(-3,1)}$, and this determines the zeros and poles of $T_g^{(-3,1)}$ on the upper half-plane. (See \S~4.3 of \cite{pmp} for the specialization of Theorem 13.3 of \cite{Borcherds:1996uda} to the situation of relevance here.)
We then compute the 
generalized class numbers $H^W_\ell(g)$, as defined in \S~4.3 of \cite{pmp},
in order to determine the zeros and poles at the various cusps. 

For the remainder of this proof let us take $W=W^{(-3,1)}$. Then $W_{0,0}$ is the unique irreducible $\Th$-module of dimension $248$, so we have $\eta^W_g=\eta_g^2$ for all $g\in \Th$, where $\eta_g$ is as in (\ref{eqn:int-com:etag}). In particular, $\eta^W_g(\tau)=\prod_{b>0}\eta(b\tau)^{2v_b(g)}$, where the Frame shapes $\prod_{b>0}b^{v_b(g)}$ are as in Table \ref{tab:res-ava:pig}.

It develops that $H^W_\ell(g)=H^W$ when $\ell$ represents the infinite cusp (see (\ref{eqn:HW})), so that in particular, $H^W_\ell(g)$ is independent of $g$ in this case. 
To compute $H^W$ we first observe that $F^{(-3,1)}=(F^{(-3,1)}_0,F^{(-3,1)}_1)$ satisfies 
\begin{gather}\label{eqn:F310F311}
F^{(-3,1)}_0(\tau)=248+O(q),\quad
F^{(-3,1)}_1(\tau)=2q^{-\frac34}+O(q^{\frac14})
\end{gather} 
(cf.\ (\ref{eqn:int-com:breveF-31})).
Then we note, from (\ref{eqn:breveG1}), that $G_1=(G_{1,0},G_{1,1})$ satisfies 
\begin{gather}\label{eqn:G10G11}
G_{1,0}(\tau)=-\frac1{12}+O(q),\quad
G_{1,1}(\tau)=\frac13q^{\frac34}+O(q^{\frac74}).
\end{gather} 
Now plugging (\ref{eqn:F310F311}--\ref{eqn:G10G11}) into the definition (\ref{eqn:HW}) of $H^W$ we obtain that 
\begin{gather}\label{eqn:res-ava:HW}
H^W=-\frac{248}{12}+\frac23=-20.
\end{gather} 
Since $W_{0,0}$ is $248$-dimensional we have $\sum_{b>0}bv_b(g|W_{0,0})=248$ for all $g\in \Th$, so from the definition 
(\ref{eqn:SQ-BP:etaWg})
we obtain that $\eta^W_g(\tau)=q^{\frac{62}3}(1+O(q))$ as $\Im(\tau)\to \infty$, and thus $T^{(-3,1)}_g(\tau)=q^{-\frac23}+O(q^{\frac13})$ as $\Im(\tau)\to \infty$ (cf.\ (\ref{eqn:SQ-BP:TWg})), in precise agreement with $\widetilde{T}_{{\rm 3C},g}^2$.
Thus the leading terms on both sides of (\ref{eqn:res-ava:Tg31}) agree (and actually agree for all $g\in \Th$).

For an example of the computation of $H^W_\ell(g)$ where $\ell$ does not represent the infinite cusp let $\ell=\lambda(1,0,0)$, in the notation of \S~4.4 of \cite{pmp}, so that $\ell$ represents the cusp at $0$. Proceeding as in loc.\ cit.\ we find that 
\begin{gather}\label{eqn:HWellg}
	H^W_\ell(g)
	:=
	\sum_{r\xmod 2}\sum_{\substack{D\equiv r^2\xmod 4\\D\leq 0}} \check{C}_K(D,r)H_{1}(D,r),
\end{gather}
where for $r\xmod 2$ 
we have
\begin{gather}\label{eqn:checkFKat0}
\check{F}_{K,r}=\sum_{i\xmod N}\check{F}_{i,0,r},
\end{gather} 
where $\check{F}_{i,j,r}$ is a shorthand for $\check{F}^W_{g,i,j,r}$ (cf.\ (\ref{eqn:FWgntocheckFWg})).
(In particular, $K$ is a copy of $\sqrt{2}\ZZ$ in this case.)

Now let  
$g$ be an element of order $2$ in the Thompson group, and take $\{F^{W}_{g^n}\}_{n|2}$ as input for the map (\ref{eqn:FWgntocheckFWg}).
Then from (\ref{eqn:checkFKat0}) and the examples of repackaging given in \S~4.2 of \cite{pmp} we have $\check{F}_K=F^{(-3,1)}-120\theta_1^0$, so that 
the counterpart to (\ref{eqn:F310F311}) is
\begin{gather}\label{eqn:checkFK0checkFK1}
\check{F}_{K,0}(\tau)=128+O(q),\quad
\check{F}_{K,1}(\tau)=2q^{-\frac34}+O(q^{\frac14}).
\end{gather} 
Thus, applying (\ref{eqn:G10G11}) and (\ref{eqn:checkFK0checkFK1}) to the definition (\ref{eqn:HWellg}) of $H^W_\ell(g)$ we find that 
\begin{gather}
H^W_\ell(g)=-\frac{128}{12}+\frac23=-10,
\end{gather} 
so that the leading term in the expansion at $0$ of the Borcherds product $\Psi^W_g$, for $W=W^{(-3,1)}$ and $o(g)=2$, is proportional to $q^{10}$. 

Now the Frame shape of $g$ defined by its action on $W_{0,0}$ is $1^{-8}2^{128}$ according to Table \ref{tab:res-ava:pig}, so $\eta_g^W$contributes a constant times 
$\eta(\tau)^{-256}\eta(2\tau)^{16}=q^{-\frac{28}{3}}(1+O(q))$
to the expansion of $T^{(-3,1)}_g$ at $0$. 
We conclude that the expansion of $T^{(-3,1)}_g$ at $0$ is proportional to $q^{\frac23}(1+O(q))$, and in particular, $T^{(-3,1)}_g$ does not have a pole at $0$.
It follows that (\ref{eqn:res-ava:Tg31}) holds when $o(g)=2$. 
 
The computation of the other generalized class numbers $H^W_\ell(g)$ is similar, and a case-by-case check concludes the proof.
\end{proof}

\begin{table}\begin{center}
\begin{small}
\caption{\label{tab:res-ava:pig} {The Frame shapes $\pi(g)=\pi(g|{\bf 248})$.}}
\begin{tabular}{cc}\toprule
$[g]$ & $\pi(g)$  \\\midrule  
$1\mathrm{A}$ & $1^{248}$  \\
$2\mathrm{A}$ & $2^{128}/1^{8}$  \\
$3\mathrm{A}$ & $3^{78} 1^{14}$ \\
$3\mathrm{B}$ & $3^{81} 1^{5}$  \\
$3\mathrm{C}$ & $3^{84}/1^{4}$  \\
$4\mathrm{A}$ & $4^{64} 1^8/ 2^{8}$  \\
$4\mathrm{B}$ & $ 4^{64}/2^4$  \\
$5\mathrm{A}$ & $5^{50}/1^{2}$  \\
$6\mathrm{A}$ & $6^{44} 1^4/3^{4}2^{4}$  \\
$6\mathrm{B}$ & $6^{40} 2^8 /3^{2} 1^{2}$  \\
$6\mathrm{C}$ & $ 6^{42}2^2 1^1/3^{3} $  \\
$7\mathrm{A}$ & $7^{35} 1^3$  \\
$8\mathrm{A}$ & $8^{32}2^4/4^{4}$  \\
$8\mathrm{B}$ & $ 8^{32}/4^{2} $  \\
$9\mathrm{A}$ &  $9^{27} 1^5$                        \\
$9\mathrm{B}$ &   $9^{27}3^3/1^4$                        \\
$9\mathrm{C}$ &   $9^{27} 3^1 1^2$                       \\
$10\mathrm{A}$ &     $10^{26} 1^2/5^{2}2^{2}  $                    \\
$12\mathrm{A}\mathrm{B}$ & $12^{20} 4^4 3^2 1^2/6^{2} 2^2  $                      \\
\bottomrule
\end{tabular}
\begin{tabular}{cc}\toprule
$[g]$ & $\pi(g)$  \\\midrule  
$12\mathrm{C}$ &     $12^{21}4^1 3^3 2^1/6^{3} 1^1$                  \\
$12\mathrm{D}$ &     $12^{22} 2^2/6^2 4^2    $                 \\
$13\mathrm{A}$ &   $13^{19} 1^1$             \\
$14\mathrm{A}$ &   $14^{18}2^2/7^1 1^1 $               \\
$15\mathrm{A}\mathrm{B}$ &  $15^{17} 1^1/5^1 3^1 $          \\
$18\mathrm{A}$ & $18^{14}2^2 1^1/9^1$                        \\
$18\mathrm{B}$ & $18^{14} 3^1 2^2/9^1 1^2 $                          \\
$19\mathrm{A}$ &  $19^{13} 1^1 $                          \\
$20\mathrm{A}$ &   $20^{13} 2^1 /10^1 4^1 $                       \\
$21\mathrm{A}$ &     $21^{11} 7^2 3^1 $                       \\
$24\mathrm{A}\mathrm{B}$ &    $24^{10} 8^2 6^1 2^1 /12^1 4^1$                       \\
$24\mathrm{C}\mathrm{D}$ &  $ 24^{11} 4^1 / 12^1 8^1$                         \\
$27\mathrm{A}$ &    $27^9 9^1 1^2/3^2$                       \\
$27\mathrm{B}\mathrm{C}$ &   $27^9 9^1/3^1 1^1 $                      \\
$28\mathrm{A}$ &      $28^9 7^1 4^1 1^1/14^1 2^1 $                   \\
$30\mathrm{A}\mathrm{B}$ &    $30^9 5^1 3^1 2^1 / 15^1 10^1 6^1 1^1$                     \\
$31\mathrm{A}\mathrm{B}$ &   $31^{8} $                       \\
$36\mathrm{A}\mathrm{B}\mathrm{C}$ &   $36^7 9^1 4^1 2^1 / 18^1 1^1 $                         \\
$39\mathrm{A}\mathrm{B}$ &   $39^6 13^1 1^1 $                        \\
\bottomrule
\end{tabular}
\end{small}
\end{center}
\end{table}

\begin{rmk}
See Tables \ref{w0PTcoeffs1}--\ref{w0PTcoeffs3} for the low-lying Fourier coefficients of the $T^{(-3,1)}_g$, and see Tables \ref{w0PTdecomps1}--\ref{w0PTdecomps8} for decompositions of the low-lying spaces $V^{(-3,1)}_{n}$ into irreducible modules for $\Th$.
\end{rmk}

\begin{rmk}
From the computation (\ref{eqn:res-ava:HW}) of $H=H^W$ for $W=W^{(-3,1)}$, and from the fact that $\eta^W_g=\eta_g^2$ where $\eta_g$ is as in (\ref{eqn:int-com:etag}), we have
\begin{gather}
	T^{(-3,1)}_g(\tau)
	=
	q^{20}\exp\left(-\sum_{n>0}\sum_{k>0}{C}^{(-3,1)}_{g^k}(n^2,n)\frac{q^{nk}}{k}     \right)\eta_g(\tau)^{-2}
\end{gather}
for all $g\in \Th$, where $C^{(-3,1)}_g(D,r)$
is the coefficient of $q^D$ in the Fourier expansion of $F^{(-3,1)}_{g,r}$. 
From the definition (\ref{eqn:breveWD}) we have 
$\breve{C}^{(-3,1)}_g(D) = C^{(-3,1)}_g(D,D)$ for all $D$.
It follows that (\ref{eqn:int-com:T31g}) holds for all $g\in \Th$.
\end{rmk}

Despite the fact that the identity (\ref{eqn:res-ava:Tg31}) holds for $o(g)$ coprime to $3$, the spaces
$V^{(-3,1)}$ 
and 
$(V_{\mathrm{3C}}^\natural)^{\otimes 2}$ 
do not agree as (virtual graded) $\Th$-modules, because
$T^{(-3,1)}_g \neq \widetilde{T}_{\mathrm{3C},g}^2$
for some choices of $g$ whose orders are divisible by 3. For example, the conjugacy class labelled 3A (in the conventions of \cite{atlas}) is the first class for which this happens, as one can check by $q$-expanding both sides to low order (cf.\ Tables \ref{w03Ccoeffs1} and \ref{w0PTcoeffs1}). Therefore, $V^{(-3,1)}$ constitutes a new $\Th$-module in weight zero, which we henceforth refer to as the \emph{weight zero avatar of penumbral Thompson moonshine}. 

Having established that there is a weight zero avatar of penumbral Thompson moonshine that does not agree precisely with $\mathrm{3C}$-generalized Thompson moonshine, it is natural to ask whether or not there is correspondingly a \emph{weight one-half avatar of $\mathrm{3C}$-generalized Thompson moonshine}. That is, is there a $\Th$-module $W_{{\rm 3C}}$ which is related to $V^\natural_{\mathrm{3C}}$ 
by the construction $\SQ$? We answer this question in the affirmative with the following.

\begin{thm}[Weight one-half avatar of 3C-generalized Thompson moonshine]\label{thm:res-ava:wh}
There exists a rational weakly holomorphic $\Th$-module $W_{\rm 3C}$ of weight $\frac12$ and index $1$ such that
$\SQ(W_{\rm 3C})=(\vn_{\rm 3C})^{\otimes 3}$, and such that
\begin{gather}\label{eqn:res-ava:3F31g=2F3Cgplustheta}
3F^{(-3,1)}_{g}= 2F_{{\rm 3C},g}+3\tr (g|{\bf 248})\theta_1^0
\end{gather}
for $3$-regular $g\in \Th$,
where $F_{{\rm 3C},g}:=F^W_g$ for $W=W_{{\rm 3C}}$.
\end{thm}

\begin{proof}
The $\Th$-module structure on $W_{\rm 3C}$ is determined by the associated McKay--Thompson series $F_{{\rm 3C},g}$, which are in turn specified concretely in \S~\ref{app:ser-thm}. 
Given these specifications we confirm that they define a virtual graded $\Th$-module by applying Thompson's reformulation of  Brauer's characterization of virtual characters (see \cite{MR822245}), and Sturm's theorem from \cite{MR894516}. The former result reduces the proof to the verification of certain congruences amongst the coefficients of the McKay--Thompson series, $F_{{\rm 3C}, g}$. The latter result reduces this to a finite check. The implementation of this idea is explained in \S~3.3 of \cite{pmo}, and we refer the reader there for more detail. Sufficient congruences for the case of the Thompson group can be found in \cite{Griffin2016}, and Sturm's result tells us that it is enough to check these congruences up to $O(q^{1920})$ (cf.\ Table 6 of \cite{pmo}). We performed this check using \cite{PARI2} and \cite{SAGE}.

Having established the existence of the virtual graded $\Th$-module $W_{\rm 3C}$ with associated McKay--Thompson series $F_{{\rm 3C},g}$, 
we obtain the identity $\SQ(W_{\rm 3C})=(\vn_{\rm 3C})^{\otimes 3}$ by checking that
\begin{align}\label{eqn:W3C}
q^{-1}\exp\left( -\sum_{n>0} \sum_{k>0} C_{\mathrm{3C},g^k}(n^2,n)\frac{q^{nk}}{k}      \right)
= 
\widetilde{T}_{\mathrm{3C},g}(\tau)^3 
\end{align}
for all $g\in\Th$. This follows from a directly similar argument to that presented for the proof of Theorem \ref{thm:res-ava:wz}.

Finally, the verification of (\ref{eqn:res-ava:3F31g=2F3Cgplustheta}) is a case-by-case check using the specifications in \S~\ref{app:ser-thm}.
\end{proof}

\begin{rmk} 
Note that $W_{{\rm 3C}}$ is the preimage under (\ref{eqn:WtobreveW}) of the $\Th$-module $\breve{W}_{\mathrm{3C}}$ of (\ref{eqn:breveW3C}). 
See Tables \ref{whalf3Ccoeffs1}--\ref{whalf3Ccoeffs3} for the low-lying Fourier coefficients of the $F_{{\rm 3C},g}$, and see Tables \ref{whalf3Cdecomps1}--\ref{whalf3Cdecomps4} for the decompositions of the $(W_{{\rm 3C}})_{r,\frac{D}{4}}$, for small values of $D$, into irreducible modules for $\Th$. 
\end{rmk}

\begin{rmk}
In \S~\ref{app:ser-thm} we report Rademacher expressions for every McKay--Thompson series $F_{{\rm 3C},g}$, except for those with $o(g)=27$. Although Rademacher expressions exist for these cases, we find it more convenient to characterize the corresponding series 
by specifying the singular parts of the Fourier expansions of the corresponding scalar-valued functions (i.e.\ the $\breve{F}_{\mathrm{3C},g}$) about every cusp. 
This data is reported in Table \ref{tab:mts-thm:sngF3C} (see also Table \ref{tab:mts-thm:sngF31}). 
Taking this together with their multiplier systems (see Table \ref{tab:mts-thm:mod}) and the low-lying coefficients in their Fourier expansions (see \S~\ref{app:cff}), 
we completely determine the $F_{{\rm 3C},g}$.
\end{rmk}

In total, we arrive at four avatars of Thompson moonshine, as summarized in Figure \ref{fouravatars}. 
We now discuss some similarities and differences between 
these avatars. 
We begin by comparing the weight zero incarnations 
$V^{(-3,1)}$
and 
$(V_{\mathrm{3C}}^\natural)^{\otimes 2}$.
As we have mentioned, their McKay--Thompson series are closely related in that they agree on 3-regular elements of $\Th$ (see (\ref{eqn:res-ava:Tg31})).
This suggests that the modules may have a common origin in characteristic $3$ (cf.\ modular moonshine \cite{ryba1996modular,borcherds1996modular,borcherds1998modular}, and also \S~3.5 of \cite{Carnahan2012}). One difference between the two modules is that while $V_{\mathrm{3C}}^\natural$ decomposes into honest modules for $\Th$, the components of $V^{(-3,1)}$ are generally virtual. However, the virtual-ness is highly structured: The multiplicity of any given irreducible module for $\Th$ always appears with the same sign across the entire module. In particular, the irreducible representations $V_i$ for $i=4,5,14,15,27,28,29,30,34,39,45,47$ always appear with non-positive multiplicity, whereas the rest appear with non-negative multiplicity. Those that appear with non-positive multiplicity are intriguingly characterized\footnote{We thank Scott Carnahan for this observation.} as the irreducible representations on which the 3A conjugacy class takes a negative value, $\chi_i(\mathrm{3A})<0$.

Next we compare the weight one-half avatars, $W^{(-3,1)}$ and $W_{{\rm 3C}}$. One obvious similarity is that the McKay--Thompson series of both modules belong to the same kinds of spaces, and in some cases---and in particular for $3$-regular elements---they are identical up to holomorphic theta functions and an overall multiplicative factor (see (\ref{eqn:res-ava:3F31g=2F3Cgplustheta})). 
However, in general the series of $W^{(-3,1)}$ transform with different characters 
to those of $W_{{\rm 3C}}$ (cf.\ Table \ref{tab:mts-thm:mod}). Moreover, as can be seen by examining the Rademacher sum expressions in \S~\ref{app:ser-thm}, the module $W_{{\rm 3C}}$ gives rise to functions with poles at cusps inequivalent to $i\infty$ under the action of $\Gamma_0(o(g))$. This contrasts with the more easily characterized functions of penumbral Thompson moonshine, which only admit poles at cusps represented by infinity. 

Using the McKay--Thompson series $F_{{\rm 3C},g}$ to decompose the graded components $(W_{{\rm 3C}})_{r,\frac{D}{4}}$ into irreducible modules for $\Th$ reveals that they are virtual for small values of $D$, as can be seen by inspecting Tables \ref{whalf3Cdecomps1}--\ref{whalf3Cdecomps4}. On the other hand, all graded components $W^{(-3,1)}_{r,\frac{D}{4}}$ of penumbral Thompson moonshine are honest modules for $\Th$, up to an overall sign. The virtual-ness of $W_{{\rm 3C}}$ is of a different nature than that of $V^{(-3,1)}$. Namely, the components $(W_{{\rm 3C}})_{r,\frac{D}{4}}$ become honest $\Th$-modules for sufficiently large $D$, whereas $V^{(-3,1)}_{n}$ is a virtual $\Th$-module for all but finitely many $n$.

Both modules in weight one-half are number theoretically interesting, but in different ways. For example, the penumbral Thompson moonshine module $W^{(-3,1)}$ possesses a discriminant property, which relates the number fields over which the $\Th$-module $W^{(-3,1)}_{r,\frac{D}{4}}$ is defined to the number theoretic properties of the discriminant $D$ (cf.\ \cite{Harvey:2015mca} for details, or \cite{pmo} for a description of the discriminant property in penumbral moonshine more generally). Although the 3C-generalized Thompson module appears not to possess any such property, it is linked via a physically and algebraically suggestive construction (\ref{eqn:SQ-BP:SQW}) to generalized monstrous moonshine in weight zero, and the procedure of inverting this construction, it turns out, leads us to a conjectural interpretation of the $F_{{\rm 3C},g}$ as generating functions for the traces of $\widetilde{T}_{\mathrm{3C},g}^3$ over certain CM points in the upper half-plane. Next, we make this more precise.

\subsection{Traces}\label{sec:res-tra}

As we have established in \cite{pmp} (see Theorem \ref{thm:SQ}), singular theta lifts provide a method for transferring moonshine in weight one-half to moonshine in weight zero. In Theorem \ref{thm:res-ava:wh} we characterized the ``inverse image'' of generalized monstrous moonshine for the Thompson group under this map, but somewhat indirectly. It is worthwhile to ask whether or not this inverse can be obtained more systematically. 

We consider this question in a general way in \S~5.1 of \cite{pmp}, and in particular sketch how 
traces of singular moduli, as introduced by Zagier in \cite{ZagierTSM}, can be applied to it.
Our purpose here is to formulate a concrete proposal in the above mentioned case. 
Specifically, 
we presently offer conjectural expressions for the McKay--Thompson series $\breve{F}_{\mathrm{3C},g}$ as generating functions for traces of the principal moduli $\widetilde{T}_{\mathrm{3C},g}^3$,
for
every $g\in \Th$ that is not in the classes 18B or 27ABC. 

To formulate these expressions let $g\in \Th$ (not in the classes 18B or 27ABC), and let 
$\Gamma_{\mathrm{3C},g}^{(3)}=\Gamma_0(N_g\vert h_g)+S_g$ (cf.\ (\ref{eqn:Gamma0NverthplusS})) be the invariance group of $\widetilde{T}_{\mathrm{3C},g}^3$ as specified in Table \ref{tab:mts-T3C}. 
Define an auxiliary group 
\begin{gather}\label{eqn:res-tra:hatGammag}
\hat{\Gamma}_{g}:= \Gamma_0(N_gh_g)+\hat{S}_g
\end{gather} 
by letting $\hat{S}_g$ be 
the preimage of $S_g$ under the map $\Ex_{N_g|h_g}\to \Ex_{N_g/h_g}$ described in \S~\ref{sec:pre-fuc} (cf.\ (\ref{eqn:pre-fuc:WEsubsetwn})).
For example, $\Gamma^{(3)}_{\mathrm{3C},\mathrm{8A}} = 8\vert 2+$, so $\hat{\Gamma}_{\mathrm{8A}} = 16+$ (cf.\ (\ref{eqn:Nverthplusnnprime})). 

Also associate a (generalized) genus character $\chi_g$ to each $g\in \Th$ as follows. 
Take $\chi_g=\chi_{-3}$ as in (\ref{eqn:pre-cmp:chiD0N})
when $g$ is not in the classes $\mathrm{3C}$, $\mathrm{9A}$, $\mathrm{18AB}$, $\mathrm{27ABC}$ or $\mathrm{36ABC}$, 
take $\chi_g=\widetilde{\chi}_{-3}$ as in (\ref{eqn:pre-cmp:widetildechiD03})
when $g$ belongs to $\mathrm{9A}$, $\mathrm{18A}$ or $\mathrm{36ABC}$, 
and define $\chi_g(Q):=\chi_{-3}(Q)-\chi_{-3}(Q|r_3)$ where 
$r_3=\frac1{\sqrt{3}} 
\left(
\begin{smallmatrix}
1&0\\0&3
\end{smallmatrix}
\right)$
in case $g$ is in the class 3C.
Now our conjectural 
expressions for the $\breve{F}_{{\rm 3C},g}$ read as follows, where 
the trace $\Tr_{D_0,\Gamma}(T,D,\chi)$ is as in (\ref{sec:pre-cmp:TrGammachiD0}).

\begin{con}\label{con:res-tra}
For $g$ in $\Th$ not in the classes 
$\mathrm{9B}$, $\mathrm{18B}$ or $\mathrm{27ABC}$ we have
\begin{align}
\begin{split}
&\breve{F}_{\mathrm{3C},g}(\tau) = 3q^{-3}
- 3{\tr (g|\boldsymbol{248})}\frac{\theta(\tau)-1}2
+3\sum_{n>1}\kappa_{g,n^2}
\frac{\theta(n^2\tau)-1}2  \\
&\hspace{1.5in} 
+3
\sum_{D>0}
\frac1{\sqrt{D}}{\Tr_{-3,\hat{\Gamma}_g}(\widetilde{T}_{\mathrm{3C},g}^3,D,\chi_g)} 
q^{D},
\end{split}
\end{align}
where the $\kappa_{g,n^2}$ all vanish except 
$\kappa_{\mathrm{9A},9} = -{9}$, $\kappa_{\mathrm{18A},9}=-{1}$, $\kappa_{\mathrm{36ABC},9}=-{3}$, and $\kappa_{\mathrm{36ABC},36}=4$.
For $g$ in the class $\mathrm{9B}$ we have 
\begin{align}
\begin{split}
&\breve{F}_{\mathrm{3C},\mathrm{9B}}(\tau)= 3q^{-3} 
-3{\tr (g|\boldsymbol{248})}\frac{\theta(\tau)-1}2\\
& \hspace{1in} 
+3
\sum_{D>0}
\frac{1}{\sqrt{D}}\left(\sum_{Q\in \mathcal{Q}^{(9)}_{-3D}/\Gamma_0(3)} 
\chi_{-3}(Q)\frac{\widetilde{T}_{\mathrm{3C},g}(\a_Q)^3}{\vert\overline{\Gamma_0(3)}_Q\vert}  \right)q^{D}.
\end{split}
\end{align}
\end{con}

We take Conjecture \ref{con:res-tra} as evidence that the Thompson group is not only organizing the Fourier coefficients of the $\widetilde{T}_{\mathrm{3C},g}^3$, but also their traces over distinguished divisors in the upper half-plane.

\begin{rmk}
The need to pass from $\Gamma^{(3)}_{\mathrm{3C},g}$ to the subgroup $\hat\Gamma_g$ of (\ref{eqn:res-tra:hatGammag}) stems from the fact that the former doesn't always act naturally on the requisite spaces of binary quadratic forms, while the latter does. However, in this regard, $\hat{\Gamma}_g$ is smaller than might be expected because we have removed Atkin--Lehner cosets associated to exact divisors which are divisible by 3, in spite of the fact that they do act sensibly on the relevant sets of binary quadratic forms. The reason we need to do this is somewhat mysterious. 
\end{rmk}

\begin{rmk}
The space $\mc{Q}^{(9)}_{D}$ is not generally invariant under the action of $\Gamma_0(3)$, but is so if $D$ is divisible by $3$. 
\end{rmk}

Although the problem of constructing higher level analogs of traces of singular moduli has been addressed in the literature (see e.g.\ \cite{MillerPixton,kim_2006,CHOI2008700,beneish_larson}), we are not aware of any results that fully encapsulate our formulae. 
We have verified the first few terms of each of the formulae of Conjecture \ref{con:res-tra} numerically, and we note that the $g=e$ case can be deduced from known results (see e.g.\ Theorems 1.1 and 1.2 of \cite{BringmannOno2007}, or Corollary 1.7 of \cite{MillerPixton}).

\subsection{Generalizations}\label{sec:res-gen}

One naturally wonders whether or not there are other instances of moonshine in weight one-half that are related to generalized monstrous moonshine in weight zero via Borcherds products and traces of singular moduli, in the manner explored in \S\S~\ref{sec:res-pro}--\ref{sec:res-tra}. 
With this in mind we recall from \cite{pmo} that the weight one-half Thompson moonshine of \cite{Harvey:2015mca} generalizes to a family of related phenomena which was given the name penumbral moonshine. 
Could other instances of penumbral moonshine possess generalized monstrous moonshine avatars, and thus extend our observations about the Thompson group? 

It is indeed our expectation that the story we have told for the Thompson group will generalize. To explain how, first recall that instances of penumbral moonshine are labeled by ``lambdencies'' $\pd=(D_0,\ell)$. In the present context, we will stick to cases for which $D_0$ is a negative fundamental discriminant and 
\begin{gather}\label{eqn:ell}
\ell=m+n,n',\dots
\end{gather}
denotes a genus zero Fricke Atkin--Lehner group with square-free level $m$. (The discriminant $D_0$ should further be ``$\ell$-admissible'' which, among other things, means that it should be a square modulo $4m$; see \cite{pmo} for further details.) Penumbral moonshine then assigns (in at least the cases that $D_0=-3$ or $D_0=-4$) a finite group $G^{(\pd)}$ and a collection of 
vector-valued modular forms $F^{(\pd)}_g=(F^{(\pd)}_{g,r})$ 
of weight $\frac12$
which belong to spaces of the shape $V^{\wh,\alpha}_{\frac12,m}(N,\chi)$ (see \S~\ref{sec:pre-mod}), and which satisfy
\begin{align}\label{eqn:breveFuplambdag}
\breve{F}^{(\pd)}_g(\tau)
:= \sum_{r\xmod 2m}F^{(\pd)}_{g,r}(4m\tau) 
=Cq^{D_0}+O(1)
\end{align}
(cf.\ (\ref{eqn:prebreveF}--\ref{eqn:breveF})) for some constant $C=C^{(\pd)}$. 
The exact divisors $n,n',\dots$ in (\ref{eqn:ell}) instruct us to take the character $\alpha:O_m\to\mathbb{C}^\ast$ to be the unique one whose kernel is $\ker\alpha = \{a(1),a(n),a(n'),\dots\}$, where $a(n)$ is the unique $a\xmod 2m$ such that $a\equiv -1\xmod 2n$ and $a\equiv 1 \xmod \frac{2m}{n}$. (Note that the assignment $n\mapsto a(n)$ furnishes an isomorphism between $\Ex_m$ and $O_m$.) These modular forms $F^{(\pd)}_g$ are interpreted as the McKay--Thompson series of a weakly holomorphic virtual $G^{(\pd)}$-module $W^{(\pd)}$ of weight $\frac12$ and index $m$, where $m$ is as in (\ref{eqn:ell}). 

The case of penumbral Thompson moonshine corresponds to the lambdency $\pd = (-3,1)$, and we have seen that it admits a generalized monstrous moonshine avatar, $W_{\rm 3C}$ (cf.\ Theorem \ref{thm:res-ava:wh}). The question we would like to ask is if penumbral lambdencies beyond $\pd=(-3,1)$ support generalized monstrous moonshine avatars as well. 
Let us first consider some cases with $D_0=-3$. In Table \ref{tab:penumbralgroupsD0m3} we associate, to each lambdency 
\begin{gather}\label{eqn:pd3}
\pd=(-3,m+n,n',\dots)
\end{gather} 
for which $m$ is 3-regular, a conjugacy class $h=h^{(\pd)}$ of the monster which has order $3m$. For each $\pd$ in Table \ref{tab:penumbralgroupsD0m3} define 
$F_{h} := aF^{(\pd)}+b\theta_m^0$,
for $h=h^{(\pd)}$, where $a$ and $b$ are chosen so that 
\begin{align}\label{eqn:breveFhuplambda}
\breve{F}_{h}(\tau) := \sum_{r\xmod 2m} F_{h,r}(4m\tau)= 3q^{-3}+0+O(q).
\end{align}
We claim that an application of Theorem \ref{thm:SQ} to the rational weakly holomorphic $G$-module $W_h$ of weight $\frac12$ and index $m$ for $G=\{e\}$ the trivial group that is defined by $F_h=(F_{h,r})$ recovers the graded vector space structure on $(\vn_h)^{\otimes 3}$. More concretely, we claim that
\begin{gather}\label{eqn:res-gen:Th3}
q^{-1}\exp\left( -\sum_{n>0} \sum_{k>0} C_{h}(n^2,n)\frac{q^{nk}}{k}      \right)
=
q^{-1}\prod_{n>0}(1-q^n)^{C_h(n^2,n)}
= 
\widetilde{T}_{h,e}(\tau)^3 
\end{gather}
(cf.\ (\ref{eqn:W3C})), for $\Im(\tau)$ sufficiently large, for $C_h(D,r)$ the coefficient of $q^D$ in $F_{h,r}$, 
where $\widetilde{T}_{h,e}(\tau):=T_{h,e}(3\tau)$.

As a consequence of Carnahan's proof 
of generalized monstrous moonshine, $\vn_h$ admits an action of a finite group $G_{h}$ which is closely related to the centralizer of $h$ in the monster. We find that the penumbral groups $G^{(\pd)}$ and the generalized monstrous moonshine groups $G_{h}$ \emph{precisely} coincide with one another for the $\pd$ and $h$ in Table \ref{tab:penumbralgroupsD0m3}. 
Thus, for each such $\pd$, we speculate that there exists a 
rational weakly holomorphic virtual $G^{(\pd)}$-module $W_h$ of weight $\frac12$ and index $m$ 
such that $\SQ(W_h) = (V_h^\natural)^{\otimes 3}$, in analogy with Theorem \ref{thm:res-ava:wh}.
\begin{table}
\begin{small}
\begin{center} 
\caption{\label{tab:penumbralgroupsD0m3} {Speculative relations between instances of penumbral moonshine with $D_0=-3$ and generalized monstrous moonshine. }}
\begin{tabular}{c | c c c c c }
\toprule
$\pd$ & $(-3,1)$ &  $(-3,7+7)$ & $(-3,13+13)$ & $(-3,19+19)$ & $(-3,31+31)$  \\
$G^{(\pd)}$ & $\Th$ &  $L_2(7)$ & $\ZZ/3\ZZ$ & $\mathds{1}$ & $\mathds{1}$  \\\midrule
$h$& 3C & 21C & 39B & 57A & 93A\\
$G_{h}$& $\Th$ & $L_2(7)$ & $\ZZ/3\ZZ$ & $\mathds{1}$ & $\mathds{1}$ \\
\bottomrule
\end{tabular}
\end{center}
\end{small}
\end{table}

The $D_0=-4$ family may be treated similarly. In Table \ref{tab:penumbralgroupsD0m4} we associate, to each lambdency 
\begin{gather}\label{eqn:pd4}
\pd = (-4,m+n,n',\dots)
\end{gather}
for which $m$ is even (except for $m=34$), a conjugacy class $h$ of the monster which has order $4m$. 
(The reason the case $m=34$ does not admit such an association is that there is no element of the monster group of order $4\times 34 = 136$.) 
For $\pd$ and $h=h^{(\pd)}$ as in Table \ref{tab:penumbralgroupsD0m4} we define $F_{h}=aF^{(\pd)}+b\theta_m^0$ so that 
\begin{align}\label{eqn:breveFh}
\breve{F}_{h}(\tau):= \sum_{r\xmod 2m} F_{h,r}(4m\tau) = 4q^{-4}+0+O(q).
\end{align}
Then we find that 
\begin{gather}\label{eqn:res-gen:Th4}
q^{-1}\exp\left( -\sum_{n>0} \sum_{k>0} C_{h}(n^2,n)\frac{q^{nk}}{k}      \right)
=
q^{-1}\prod_{n>0}(1-q^n)^{C_h(n^2,n)}
= 
\widetilde{T}_{h,e}(\tau)^4, 
\end{gather}
(cf.\ (\ref{eqn:W3C}), (\ref{eqn:res-gen:Th3})), for $\Im(\tau)$ sufficiently large, for $C_h(D,r)$ the coefficient of $q^D$ in $F_{h,r}$ as before, 
but where $\widetilde{T}_{h,e}(\tau):=T_{h,e}(4\tau)$.
This time, we find that the generalized monstrous moonshine group is a subgroup of the penumbral group with index two, i.e.\ $[G^{(\pd)}:G_{h}] = 2$. 
That notwithstanding, we take this and the coincidence (\ref{eqn:res-gen:Th4}) as reasonable evidence that the $D_0=-4$ lambdencies in Table \ref{tab:penumbralgroupsD0m4} will support generalized monstrous moonshine avatars in weight one-half as well.

\begin{table}
\begin{small}
\begin{center} 
\caption{\label{tab:penumbralgroupsD0m4} {Speculative relations between instances of penumbral moonshine with $D_0=-4$ and generalized monstrous moonshine. }}
\begin{tabular}{c | c c c c }
\toprule
$\pd$ & $(-4,2+2)$ &  $(-4,10+2,5,10)$ & $(-4,26+2,13,26)$ & $(-4,34+2,17,34)$   \\
$G^{(\pd)}$ & ${^2}\textsl{F}_4(2)$ &  $\ZZ/5\ZZ:\ZZ/4\ZZ$ & $\ZZ/2\ZZ$ & $\mathds{1}$  \\\midrule
$h$& 8C & 40A & 104A &\\
$G_{h}$& ${^2}\textsl{F}_4(2)'$ & $\ZZ/5\ZZ:\ZZ/2\ZZ$ & $\mathds{1}$ & \\
\bottomrule
\end{tabular}
\end{center}
\end{small}
\end{table}

Theorem \ref{thm:SQ} applies directly to the penumbral moonshine modules $W^{(\pd)}$, producing counterparts $V^{(\pd)}:=\SQ(W^{(\pd)})$ in weight zero. 
A comparison of the $G_h$-module structures on $V^{(\pd)}$ and $\vn_h$, for the $\pd$ and $h$ that appear in Tables \ref{tab:penumbralgroupsD0m3}--\ref{tab:penumbralgroupsD0m4}, should be illuminating.
We expect the challenge of extending (\ref{eqn:res-gen:Th3}) and (\ref{eqn:res-gen:Th4}) to statements about $G_h$-modules to be no less interesting.

It is also interesting to ask whether or not the observations we have offered can be pushed further. For example, there are other $D_0=-3$ and $D_0=-4$ lambdencies with known penumbral modules that we did not mention in the above discussion, and there are also lambdencies with $D_0<-4$ for which the existence or not of corresponding penumbral moonshine is yet to be determined. As a first step in the direction of counterparts to the present work in such cases, we offer a final result. In a nutshell, it computes the Borcherds products, as in (\ref{eqn:res-gen:Th3}) and (\ref{eqn:res-gen:Th4}), associated to a more general class of functions that might arise as graded dimensions of penumbral moonshine modules, and confirms the existence of particular examples for which the corresponding Borcherds product might admit a relation to generalized monstrous moonshine, in a manner similar to what we have explained above.

For the statement of the result we define a generalized class number $H_\Gamma(D_0)$, for $\Gamma$ a group of the form $\Gamma_0(m)+S$, and $D_0$ a negative discriminant that is a square modulo $4m$, by setting

\begin{align}\label{eqn:HGammaD0}
    H_\Gamma(D_0) := \sum_{Q\in\mathcal{Q}^{(m)}_{D_0}/\Gamma} \frac{1}{|\bar{\Gamma}_Q|}.
\end{align}
We also write $T_\Gamma$ for the unique normalized Hauptmodul of $\Gamma$ when $\Gamma$ is genus zero, and write 
$C_F(D,r)$ for the coefficient of $q^{\frac{D}{4m}}$ in the expansion of $F_r$, when $F=(F_r)$ belongs to $\textsl{V}^{\wh}_{\frac12,m}$.

\begin{pro}\label{pro:generalBorPrds}
Let $m$ be a square-free integer such that $\Gamma=\Gamma_0(m)+$ is genus zero 
and
let $D_0$ be a negative discriminant that is a square modulo $4m$. Then there is a unique function $F=(F_r)$
in $\textsl{V}^{\wh,\alpha}_{\frac12,m}$, where $\a$ is the trivial character of $O_m$, such that $\breve{F} = q^{D_0}+0+O(q)$, and such that
\begin{align}\label{eqn:generalBorPrds}
    q^{-H_{\Gamma}(D_0)}\prod_{n>0} (1-q^n)^{C_F(n^2,n)} = \prod_{Q\in\mathcal{Q}^{(m)}_{D_0}/\Gamma}(T_\Gamma(\tau)-T_\Gamma(\a_Q))^{\frac{1}{|\bar{\Gamma}_Q|}}
\end{align}
holds for $\Im(\tau)$ sufficiently large.
Furthermore, for every such $m$ there exists a choice of $D_0$ 
such that the Borcherds product (\ref{eqn:generalBorPrds}) takes the form 
\begin{gather}\label{eqn:generalBorPrds_aomega}
(T_\Gamma+a)^{\frac{1}{\omega}}
\end{gather}
for some integers $a$ and $\omega$.
\end{pro}

\begin{proof}

Note first that the uniqueness of $F$ follows because the relevant space of theta series is $1$-dimensional when $m$ is square-free. Note also that $F$ has rational coefficients $C_F(D,r)$ according to the main result of \cite{MR1981614}. From here on in we make heavy use of Theorem 13.3 of \cite{Borcherds:1996uda} (or the specialized version that appears in \S~4.3 of \cite{pmp}). In particular, we deduce that the Borcherds product has weight $0$ by applying Item 1 of that result. 
By virtue of the hypothesis that $F$ belongs to the space $\textsl{V}^{\wh,\alpha}_{\frac12,m}$ with $\alpha$ the trivial character (cf.\ \eqref{eqn:Fartau}) we obtain that the associated Borcherds product is a weight $0$ function, transforming under $\Gamma_0(m)+$ according to a unitary character $\chi:\Gamma_0(m)+\to \CC^*$.

Now consider the smallest integer $k$ such that $\chi(\gamma)^k = 1$ for all $\gamma\in\Gamma_0(m)+$. Then the $k$-th power of the Borcherds product defined by $F$, i.e.\ the Borcherds product of $kF$, will have a trivial character, and thus be a weight $0$ modular function for the group $\Gamma_0(m)+$. Since $\Gamma_0(m)+$ is assumed to have genus zero this function defined by $kF$ is a rational function of $T_\Gamma$, and is completely determined by its zeros and poles. Since $m$ is square-free the group $\Gamma_0(m)+$ has only a single cusp, and the behavior of the Borcherds product at this cusp is controlled by the Weyl vector. The divisor in the interior of the upper half-plane can be read off from Theorem 13.3 of \cite{Borcherds:1996uda}. Using the fact that the action of $O_m$ on the $r\xmod 2m$ such that $D_0 \equiv r^2\xmod 4m$ is transitive when $m$ is square-free, we discover that this divisor is simply 
\begin{align}
    \sum_{Q\in\mathcal{Q}^{(m)}_{D_0}/\Gamma}\frac{k}{|\bar{\Gamma}_Q|}\cdot \a_Q.
\end{align}
Therefore, we have that
\begin{align}\label{eqn:generalBorPrds_kpow}
    q^{-kH}\prod_{n=1}^\infty(1-q^n)^{kC(n^2,n)} = \prod_{Q\in\mathcal{Q}^{(m)}_{D_0}/\Gamma} (T_\Gamma(\tau)-T_\Gamma(\a_Q))^{\frac{k}{|\bar{\Gamma}_Q|}}
\end{align}
for some $H$, 
and comparison of the leading powers of $q$ on both sides forces $H = H_\Gamma(D_0)$. Taking the $k$-th root on both sides of (\ref{eqn:generalBorPrds_kpow}) then gives us (\ref{eqn:generalBorPrds}).

The final claim of the theorem is confirmed by a finite check, the results of which we report in \S~\ref{app:ser-gen}.
\end{proof}

\begin{rmk}
We note a ``class number one'' property that underpins Proposition \ref{pro:generalBorPrds}. Namely, the Borcherds product (\ref{eqn:generalBorPrds}) takes the form (\ref{eqn:generalBorPrds_aomega}) whenever 
$h_\Gamma(D_0):=|\mathcal{Q}^{(m)}_{D_0}/\Gamma|
=1$.
\end{rmk}

\begin{rmk}
We have restricted to groups with square-free level $m$ and with all Atkin--Lehner cosets present. 
It would be interesting to see similar results for more general choices of $\Gamma=\Gamma_0(m)+S$ with genus zero. 
\end{rmk}

We conclude by mentioning that Proposition \ref{pro:generalBorPrds} may be motivated independently of moonshine. Indeed, a natural problem in the theory of automorphic forms is to characterize the image of the singular theta lift. Proposition \ref{pro:generalBorPrds} provides a positive answer, in the $O_{2,1}$ setting (cf.\ e.g.\ \cite{pmp}), for some of the replicable functions that arise in generalized monstrous moonshine (cf.\ \cite{for-mck-nor}). Do all such functions arise as Borcherds products $\Psi^W_g$ defined by 
weakly holomorphic $G$-modules $W$ of weight $\frac12$, for suitable groups $G$, via the construction $\SQ$?

%---------------------------------------------------------------------------------------%

\section{Summary and Outlook}\label{sec:sum}

%---------------------------------------------------------------------------------------%

Motivated by
two distinct manifestations 
of 
the Thompson sporadic simple group 
in moonshine (see \S\S~\ref{sec:int-woh}--\ref{sec:int-com}),
and noting that the corresponding moonshine modules are nonetheless related 
(\ref{eqn:int-com:breveF-31}--\ref{eqn:int-com:infiniteproductT3C})
by an infinite product formula, we have explored the relationship between these modules in detail in this work.

The relationship 
(\ref{eqn:int-com:breveF-31}--\ref{eqn:int-com:infiniteproductT3C})
has in turn motivated the lift $\SQ$ of the $O_{2,1}$-type singular theta lift of \cite{Harvey:1995fq,Borcherds:1996uda} (see \S~\ref{sec:int-con} and Theorem \ref{thm:SQ}), that we have established in \cite{pmp}. 
Here we have applied this lift (\ref{eqn:int-con:SQ}), and in so doing produced two new avatars of Thompson moonshine. 
One of these 
realizes a counterpart to penumbral Thompson moonshine in weight zero (see Theorem \ref{thm:res-ava:wz}).
The other 
serves as a counterpart to 3C-generalized monstrous moonshine in weight one-half (see Theorem \ref{thm:res-ava:wh}). 
We have discussed the similarities and differences between these modules in \S~\ref{sec:res-ava}.

In the hope that our work will help to illuminate the algebraic and physical origins of Thompson moonshine in weight one-half, 
and penumbral moonshine more generally, we have emphasized an analogy with monstrous moonshine (see \S~\ref{sec:int-com}), whereby the $\Th$-modules in weight one-half act as counterparts to the moonshine module, and the $\Th$-modules in weight zero act as root spaces for analogues of the monster Lie algebra. Curiously, the weight-zero counterpart to penumbral Thompson moonshine has the same graded dimension as the moonshine module, up to the constant term. 
Thus we have arrived at the possibility of two apparently different roles for the elliptic modular invariant, in connection with exceptional finite groups.

We have also proposed an inverse to our lift of the singular theta lift (\ref{eqn:int-con:SQ}), in the special setting of the 3C-generalized avatars (see Conjecture \ref{con:res-tra}), in terms of traces of singular moduli. We refer to \S~5 of \cite{pmp} for a more general discussion of the application of TSM to the problem of inverting $\SQ$, and comments on the import of the relationship between $\SQ$ and TSM for penumbral and umbral moonshine.
As explained in loc.\ cit., the problem of inverting $\SQ$ motivates the problem of extending the main result of \cite{pmp}, so as to lift---to the level of $G$-modules---the twisted Borcherds products that were introduced by Zagier \cite{ZagierTSM}, and reformulated in terms of twisted $O_{2,1}$-type singular theta lifts by Bruinier and Ono \cite{BruinierOno2010}.

There are a number of open paths for future research. 
Probably the most important of these are the two that we have highlighted in \S~\ref{sec:int-two}.
Namely, 
the illumination of the algebraic and physical origins of the four avatars of Thompson moonshine that appear in Figure \ref{fouravatars},
and the elucidation of counterparts to Figure \ref{fouravatars} that relate other cases of generalized monstrous moonshine to lambdencies of penumbral moonshine.
We have taken some steps towards a better understanding of the latter of these in \S~\ref{sec:res-gen}. 

Also in \S~\ref{sec:res-gen}, we have applied the idea that singular theta lifts have a special part to play in penumbral moonshine, so as to take
a step towards the discovery of penumbral moonshine at more general lambdencies (i.e.\ $\pd=(D_0,\ell)$ with $D_0<-4$) than those for which moonshine is confirmed in \cite{pmo} (see Proposition \ref{pro:generalBorPrds}).

With regard to the origins of Figure \ref{fouravatars}, 
we note that it follows from our results that the four avatars of Thompson moonshine 
essentially collapse to two, when considered in characteristic $3$. 
In weight zero, there is a vertex algebra defined over $\FF_3$ whose graded Brauer characters agree with the 
the corresponding McKay--Thompson series $\widetilde{T}_{\mathrm{3C},g}$ that arise in 3C-generalized monstrous moonshine.
In particular, this modular vertex algebra realizes moonshine for the Thompson group, and moreover its McKay--Thompson series agree with those of both 3C-generalized and penumbral Thompson moonshine in weight zero. 
This suggests that the relationship between 3C-generalized and penumbral Thompson moonshine in weight one-half might be made more clear by working in characteristic 3. 
For example, the two different avatars of Thompson moonshine in weight one-half might be two different lifts to characteristic zero, of some common modular $\Th$-module, which is related via Borcherds products to modular moonshine in weight zero.

Finally, 
we reiterate that the fact that Thompson moonshine in weight zero is related to Thompson moonshine in weight one-half via singular theta lifts suggests an analogy with the 
relationship between 
the monster Lie algebra
and 
the moonshine module VOA.
Is there an analogous interpretation in the present case? Is there a BKM with $\Th$-symmetry whose twined denominator formulae agree with the graded characters of $(V_{\mathrm{3C}}^\natural)^{\otimes  3}$? Can a string theory setup, perhaps along the lines of \cite{Paquette:2016xoo,Paquette:2017xui} (see also \cite{Harrison:2018joy,Harrison:2020wxl,Harrison:2021gnp} for related work), shed additional light on these questions?

\appendix 

\section{Forms}\label{app:ser}

Here we provide concrete specifications for the modular forms that appear in this work. 
We consider the McKay--Thompson series of Thompson moonshine in weight one-half 
in \S~\ref{app:ser-thm}, 
in \S~\ref{app:ser-thz}
we consider the McKay--Thompson series of Thompson moonshine in weight zero,
and 
we consider forms 
that arise from the discussion of \S~\ref{sec:res-gen} in \S~\ref{app:ser-gen}.

\subsection{Weight One-Half Thompson Moonshine}\label{app:ser-thm}

In this section we explain how to explicitly construct the functions $F^{(-3,1)}_g$ and $F_{{\rm 3C},g}$, for each $g\in \Th$. 
We begin with the 
$F^{(-3,1)}_g$, 
which are specified in terms of Rademacher sums (see \S~\ref{sec:pre-rad}) by the formula
\begin{align}\label{eqn:mts-thm:radF31}
F^{(-3,1)}_g(\tau)
=
2R^{[\vec\mu]}_{1,\Gamma_g,\chi_g}(\tau)
-c^{[\vec\mu]}_{1,\Gamma_g,\chi_g,1}(\tfrac14)\theta^0_1(\tau)
+\sum_{k>0}k^{(-3,1)}_g\theta^0_1(k^2\tau).
\end{align}
Here $\vec{\mu} = (0,-\frac34)$, and $R_{1,\Gamma_g,\chi_g}^{[\vec\mu]}$ and $c^{[\vec\mu]}_{1,\Gamma_g,\chi_g,1}(\tfrac14)$ are as in \eqref{eqn:unnormalizedrademacher}, 
where the $\Gamma_g=\Gamma^{(-3,1)}_g$ and $\chi_g=\chi^{(-3,1)}_g$ and $k_g^{(-3,1)}$ are determined as follows. 
Given $g\in \Th$, inspect the rows of Table \ref{tab:mts-thm:mod} labelled $(-3,1)$ in order to arrive at a corresponding symbol $N|h_v$, and for this purpose regard $N$ as a shorthand for $N|1_1$. 
Then take $\Gamma_g=\Gamma^{(-3,1)}_g=\Gamma_0(N)$ and $\chi_g=\chi^{(-3,1)}_g=\chi_{N,v,h}$, where 
\begin{gather}\label{eqn:chiNvh}
\chi_{N,v,h}
\left(\begin{matrix} a & b \\ c & d \end{matrix}\right) 
:= \ex\left(- v \frac{cd}{Nh}\right)
\end{gather}
for $\left(\begin{smallmatrix}a&b\\c&d\end{smallmatrix}\right)$ in $\Gamma_0(N)$ (so that $\chi^{(-3,1)}_g$ is trivial if the $|h_v$ is suppressed in Table \ref{tab:mts-thm:mod}).
Finally, inspect Table \ref{tab:mts-thm:thtF31}, where $k_g (k^2)$ appears as a shorthand for $k_g\theta^0_1(k^2\tau)$, for the non-zero $k_g=k^{(-3,1)}_g$.
(If a conjugacy class does not appear in Table \ref{tab:mts-thm:thtF31} then the corresponding $k_g^{(-3,1)}$ all vanish. 
Also, Rademacher sum specifications equivalent to those given here for $F^{(-3,1)}_g$ appeared earlier in \cite{Harvey:2015mca}, but with different notational conventions.)

\begin{table}[h!]
\begin{center}
\begin{small}
\caption{
Modularity of the $F^{(-3,1)}_g$ and $F_{{\rm 3C},g}$.
\label{tab:mts-thm:mod}}
\begin{tabular}{c|cccccccccc}
\toprule
$[g]$	&1A	&2A	&3A	&3B	&3C	&4A	&4B	&5A	&6A	&6B		\\
	\midrule
${(-3,1)}$ &$1$  &$2$ &$3|3_2$ &$3$ &$3\vert 3_1$ &$4$ &$4\vert 2_1$ &$5$ &$6\vert 3_2$ & $6\vert 3_1$   
\\[1pt]
3C 
&$1$  &$2$ &$3+3$ &$3$ &$3+\bar{3}$ &$4$ &$4\vert 2_1$ &$5$ &$6+\bar{3}$ & $6+3$   
\\
\midrule 
\midrule
$[g]$	&6C &7A & 8A & 8B & 9A & 9B & 9C & 10A & 12AB & 12C  	\\
\midrule
$(-3,1)$ &$6$  &$7$ &$8\vert 2_1$ &$8\vert 4_3$ &$9$ &$9$ &$9\vert 3_2$ &$10$ &$12\vert 3_2$ & $12$    \\[1pt]
3C &$6$  &$7$ &$8\vert 2_1$ &$8\vert 4_3$ &$9+9$ &$9$ &$9$ &$10$ &$12+3$ & $12$    \\
\midrule
\midrule
$[g]$	&  12D & 13A & 14A & 15AB & 18A & 18B & 19A &20A & 21A & 24AB  	\\
\midrule
$(-3,1)$ &$12\vert 6_5$  &$13$  &$14$ &$15\vert 3_2$ &$18$ &$18\vert 3_1$ &$19$ &$20\vert 2_1$ &$21\vert 3_2$ &$24\vert 6_5$       \\[1pt]
3C &$(12+3)|2_1$  &$13$  &$14$ &$15+ \bar{3}$ &$18+9$ &$18$ &$19$ &$20\vert 2_1$ &$21+3$ &$(24+3)|2_1$       \\
	\midrule
\midrule
$[g]$ & 24CD & 27A & 27BC & 28A & 30AB & 31AB & 36A & 36BC & 39AB &  \\
\midrule
$(-3,1)$ &$24\vert 12_{11}$  &$27\vert 3_2$ &$27\vert 3_2$ &$28$ &$30\vert 3_1$ &$31$ &$36$ &$36$ &$39\vert 3_2$ &      \\[1pt]
3C &$(24+3)|4_1$  &$27+?$ &$27+?$ &$28$ &$30+\bar{3}$ &$31$ &$36+9$ &$36+9$ &$39+3$ &      \\
\bottomrule
\end{tabular}
\end{small}
\end{center}
\end{table}

\begin{table}[h!]
\begin{center}
\begin{small}
\caption{\small Theta corrections for the $F^{(-3,1)}_g$. 
\label{tab:mts-thm:thtF31}
} 
\begin{tabular}{c|cccccccc}
\toprule
$[g]$	&3A	&4A	 & 9A & 9B  
& 12AB & 12C  &21A	
\\
	\midrule
  $\sum_kk^{(-3,1)}_{g}(k^2)$ & $18(9)$ &$8(4)$  & $ 6(9)$ &$-3(9)$ 
  &$-1(4)+3({36})$ & $-1(4)$  &$-\frac{3}{8}(9) $ 
  \\[2pt]
\midrule
\midrule
$[g]$ & 27A & 27BC & 28A  & 36A & 36BC & 39AB  \\
	\midrule
  $\sum_kk^{(-3,1)}_{g}(k^2)$ &$-1(9)+3(81)$ &$\frac{1}{2}(9)-\frac{3}{2}({81})$ &$1(4)$  &$2(4)-3({36})$ & $-1(4)$ &$\frac{9}{7}(9)$   \\ [2pt]
  \bottomrule
\end{tabular}
\end{small}
\end{center}
\end{table}

The 
functions $F_{{\rm 3C},g}$ 
are specified, at least for $g$ not of order $27$, by a similar prescription. Namely, when $o(g)\neq 27$ we have
\begin{align}\label{eqn:mts-thm:radF3C}
    F_{{\rm 3C},g}(\tau) = 3\left(R_{1,\Gamma_g,\chi_g}^{[\vec{\mu}]}(\tau)-c_{1,\Gamma_g,\chi_g,0}^{[\vec\mu]}(0)\theta_1^0(\tau) + \sum_{k>0}k_{{\rm 3C},g}\theta^0_1(k^2\tau)\right),
\end{align} 
where $\vec{\mu} = (0,-\frac{3}{4})$ as in (\ref{eqn:mts-thm:radF31}), the $R_{1,\Gamma_g,\chi_g}^{[\vec{\mu}]}$ and $c^{[\vec\mu]}_{1,\Gamma_g,\chi_g,0}(0)$ are again defined by \eqref{eqn:unnormalizedrademacher}, and the non-zero $k_{{\rm 3C},g}$ are read off from Table \ref{tab:mts-thm:thtF3C}. 

\begin{table}[h!]
\begin{center}
\begin{small}
\caption{\small Theta corrections for the $F_{{\rm 3C},g}$. 
\label{tab:mts-thm:thtF3C}
} 
\begin{tabular}{c@{ }|@{ }ccccc}
\toprule
$[g]$	&4A	 & 9AB   & 9C
& 12AB & 12C 
\\
	\midrule
  $\sum_kk_{{\rm 3C},g}(k^2)$ &  $-4(1)+4(4)$ & $\frac32(1)-\frac32(9)$& $-\frac32(1)+\frac32(9)$ 
&$-1(1)+1(4)$ &$\frac12(1)-\frac12(4)$ 
  \\[2pt]
 \midrule
 \midrule
$[g]$ 
&18A & 18B	
& 27A & 27BC  & 36ABC   \\
	\midrule
  $\sum_kk_{{\rm 3C},g}(k^2)$ 
& $-\frac32(1)+\frac32(9)$ & $\frac32(1)-\frac32(9)$
  & $?$   & $?$ & $\frac12(1)-\frac12(4)-\frac32(9)+\frac32(36)$ 
  \\[2pt]
  \bottomrule
\end{tabular}
\end{small}
\end{center}
\end{table}

We interpret the entries of Table \ref{tab:mts-thm:thtF3C} in the same way as those of Table \ref{tab:mts-thm:thtF31}, and $\Gamma_g=\Gamma_{{\rm 3C},g}$ and $\chi_g=\chi_{{\rm 3C},g}$ may be read off from the rows labelled 3C in Table \ref{tab:mts-thm:mod}, 
where a symbol $N|h_v$ is interpreted as in (\ref{eqn:mts-thm:radF31}).
The 3C rows of Table \ref{tab:mts-thm:mod} also feature entries of the form $(N+n)\vert h_v$ and $(N+\bar n)|h_v$, with $n$ an exact divisor of $N$. In this case $\chi_{{\rm 3C},g}=\chi_{N,v,h}$ as before (again taking $N+n$ as a shorthand for $(N+n)|h_v$, \&c.), and $N+n$ specifies $\Gamma_{{\rm 3C},g}$ according to the notational convention (\ref{eqn:Nplusnnprime}). 
The symbol $N+\bar{n}$ is interpreted similarly, except that the contribution of the Rademacher sum attached to $n$ in (\ref{eqn:mts-thm:radF3C}) should be subtracted rather than added. 

For the cases that $g$ belongs to 12D or 24ABCD there is an ambiguity 
in the definition of the contribution to the Rademacher sum attached to the exact divisor $n=3$, which we fix in each case by making the following specific choices of Atkin--Lehner involution $A$ as in \eqref{eqn:effectivemultipliersystem}.
For $g$ in 12D we take $A=\frac{1}{\sqrt{3}}\left(\begin{smallmatrix} 3 & -1 \\ 12 & -3\end{smallmatrix}\right)$, and for $g$ in 24ABCD we take $A=\frac{1}{\sqrt{3}}\left(\begin{smallmatrix} 3 & 1 \\ -48 & -15 \end{smallmatrix}\right)$.

The question marks in Tables \ref{tab:mts-thm:mod} and \ref{tab:mts-thm:thtF3C} emphasize the fact that we do not furnish expressions of the form (\ref{eqn:mts-thm:radF3C}) for $F_{{\rm 3C},g}$, for the $g\in \Th$ of order $27$ in this work.

\begin{table}[ht]
\begin{center}
\caption{Singular parts of the $\breve{F}_{g}^{(-3,1)}$. 
}\label{tab:mts-thm:sngF31}
\begin{small}
\begin{tabular}{c@{ }|@{ }cc@{ }|@{ }cc@{ }|@{ }cc@{ }|@{ }cc@{ }|@{ }cc@{ }|@{ }cc@{ }|@{ }cc@{ }|@{ }cc@{ }|@{ }c}
\toprule
\multicolumn{2}{c}{1A} & 
\multicolumn{2}{c}{3A} & \multicolumn{2}{c}{3B} & \multicolumn{2}{c}{3C} & 
\multicolumn{2}{c}{5A} & 
\multicolumn{2}{c}{7A} &
\multicolumn{2}{c}{9AB} & 
\multicolumn{2}{c}{9C} &
\multicolumn{2}{c}{13A} 
\\\midrule
$\frac{1}{2}$ & $\frac{-1}{q^{3}}$& 
$\frac{1}{18}$ & $\frac{-1}{q^{3}}$& $\frac{1}{6}$ & $\frac{-1}{q^{3}}$& $\frac{1}{18}$ & $\frac{-1}{q^{3}}$& 
$\frac{1}{10}$ & $\frac{-1}{q^{3}}$& 
$\frac{1}{14}$ & $\frac{-1}{q^{3}}$ &
$\frac{1}{18}$ & $\frac{-1}{q^{3}}$& 
$\frac{1}{54}$ & $\frac{-1}{q^{3}}$&
$\frac{1}{26}$ & $\frac{-1}{q^{3}}$ 
\\[3pt]
 \midrule
\end{tabular}

\begin{tabular}{c@{ }|@{ }cc@{ }|@{ }cc@{ }|@{ }cc@{ }|@{ }cc@{ }|@{ }cc@{ }|@{ }c}
\midrule
\multicolumn{2}{c}{15AB} & 
\multicolumn{2}{c}{19A} & 
\multicolumn{2}{c}{21A} & 
\multicolumn{2}{c}{27ABC} &
\multicolumn{2}{c}{31AB} & 
\multicolumn{2}{c}{39AB} 
\\\midrule
$\frac{1}{30}$ & $\frac{\omega^5}{q^3}$& 
$\frac{1}{38}$ & $\frac{-1}{q^{3}}$& 
$\frac{1}{42}$ & $\frac{\omega^5}{q^3}$& 
$\frac{1}{54}$ & $\frac{\omega^5}{q^3}$ &
$\frac{1}{62}$ & $\frac{-1}{q^{3}}$& 
$\frac{1}{78}$ & $\frac{\omega^5}{q^3}$ 
\\[3pt]
 \bottomrule
\end{tabular}
\end{small}
\end{center}
\end{table}

We also specify the functions $F^{(-3,1)}_g$ and $F_{{\rm 3C},g}$, up to theta series, by reporting the singular parts of their scalar-valued counterparts $\breve F^{(-3,1)}_g$ and $\breve F_{{\rm 3C},g}$ at cusps, and in this case we do include data for the $g\in \Th$ such that $o(g)=27$. 
The singular parts of the functions $\frac12\breve{F}^{(-3,1)}_g$ at non-infinite cusps are given in Table \ref{tab:mts-thm:sngF31}, where an entry $\alpha\,|\,\frac{\kappa}{q^{k}}$ indicates that $\kappa q^{-k}$ is the leading term in the Fourier expansion of  
\begin{gather}\label{eqn:mts-thm:breveF31atalpha}
\frac12\breve{F}^{(-3,1)}_g\left(\frac{a\tau+b}{c\tau+d}\right)
\frac1{\sqrt{c\tau+d}},
\end{gather} 
where $\left(\begin{smallmatrix}a&b\\c&d\end{smallmatrix}\right)$ is some element of ${\SL}_2(\mathbb{Z})$ such that $\frac a c=\alpha$. 
Here $\alpha\in \QQ$ is to be regarded as a representative of a cusp of $\Gamma_0(4N)$, where $N$ is as given in Table \ref{tab:mts-thm:mod}. 
In Table \ref{tab:mts-thm:sngF31} we only include entries for representatives of cusps where (\ref{eqn:mts-thm:breveF31atalpha}) is not bounded as $\Im(\tau)\to \infty$, and we also suppress the infinite cusp, 
because $\frac12\breve F^{(-3,1)}_g(\tau)=q^{-3}+O(1)$ as $\Im(\tau)\to\infty$ for all $g$. 
Thus, if a conjugacy class does not appear in Table \ref{tab:mts-thm:sngF31}, then the only cusp of $\Gamma_0(4N)$ at which $\breve{F}^{(-3,1)}_g$ has growth is the infinite cusp, and the expansion there is $2q^{-3}+O(1)$ as $\Im(\tau)\to \infty$.
In Table \ref{tab:mts-thm:sngF31} we use $\omega$ as a shorthand for $\ex(\frac16)$.

\begin{table}[ht]
\caption{Singular parts of the $\breve{F}_{{\rm 3C},g}$ for $3$-singular $g$. 
}\label{tab:mts-thm:sngF3C}
\vspace{-1em}
\begin{center}
\begin{small}
\begin{tabular}{c@{ }|@{ }cc@{ }|@{ }cc@{ }|@{ }cc@{ }|@{ }cc@{ }|@{ }cc@{ }|@{ }cc@{ }|@{ }cc@{ }|@{ }cc@{ }|@{ }cc@{ }|@{ }cc@{ }|@{ }c}\toprule
\multicolumn{2}{c}{3A} & \multicolumn{2}{c}{3B} & \multicolumn{2}{c}{3C} & 
\multicolumn{2}{c}{6A} & \multicolumn{2}{c}{6B} &
\multicolumn{2}{c}{9A} & \multicolumn{2}{c}{9B} & \multicolumn{2}{c}{9C} & 
\multicolumn{2}{c}{12AB} & 
\multicolumn{2}{c}{12D} &
\multicolumn{2}{c}{15AB} 
\\\midrule
$\frac{1}{18}$ & $\frac{-1}{q^{3}}$& $\frac{1}{6}$ & $\frac{-1}{q^{3}}$& $\frac{1}{18}$ & $\frac{-1}{q^{3}}$& 
$\frac{1}{8}$ & $\frac{\omega^{5}}{q}$& $\frac{1}{8}$ & $\frac{\omega^{2}}{q^{3}}$ &
$\frac{1}{18}$ & $\frac{-1}{q^{3}}$& $\frac{1}{18}$ & $\frac{-1}{q^{3}}$& $\frac{1}{54}$ & $\frac{-1}{q^{3}}$& 
$\frac{1}{16}$ & $\frac{\omega^{4}}{q^{3}}$& 
$\frac{1}{16}$ & $\frac{\omega^{4}}{q}$ &
$\frac{1}{30}$ & $\frac{-1}{q^{3}}$
\\[3pt]
 $\frac{1}{4}$ & $\frac{\omega^{4}}{q^{3}}$& \multicolumn{1}{c}{} & \multicolumn{1}{c}{} & $\frac{1}{4}$ & $\frac{\omega}{q^{3}}$& 
 \multicolumn{1}{c}{} & \multicolumn{1}{c}{} & \multicolumn{1}{c}{} & \multicolumn{1}{c}{} &
 $\frac{1}{4}$ & $\frac{\omega^{4}}{q^{3}}$& \multicolumn{1}{c}{} & \multicolumn{1}{c}{} & 
\multicolumn{1}{c}{} & \multicolumn{1}{c}{} & 
\multicolumn{1}{c}{} & \multicolumn{1}{c}{} & 
\multicolumn{1}{c}{} & \multicolumn{1}{c}{} &
$\frac{1}{20}$ & $\frac{\omega^{5}}{q^{3}}$ 
 \\[3pt]
$\frac{1}{2}$ & $\frac{\omega^{5}}{q^{3}}$& \multicolumn{1}{c}{} & \multicolumn{1}{c}{} & $\frac{1}{2}$ & $\frac{\omega^{2}}{q^{3}}$& 
\multicolumn{1}{c}{} & \multicolumn{1}{c}{} & \multicolumn{1}{c}{} & \multicolumn{1}{c}{} &
$\frac{1}{2}$ & $\frac{\omega^{5}}{q^{3}}$& \multicolumn{1}{c}{} & \multicolumn{1}{c}{} & 
\multicolumn{1}{c}{} & \multicolumn{1}{c}{} & 
\multicolumn{1}{c}{} & \multicolumn{1}{c}{} & 
\multicolumn{1}{c}{} & \multicolumn{1}{c}{} &
$\frac{1}{10}$ & $\frac{\omega^{4}}{q^{3}}$
\\[3pt]
\midrule
\end{tabular}

\begin{tabular}{c@{ }|@{ }cc@{ }|@{ }cc@{ }|@{ }cc@{ }|@{ }cc@{ }|@{ }cc@{ }|@{ }cc@{ }|@{ }cc@{ }|@{ }cc@{ }|@{ }cc@{ }|@{ }c}\midrule
\multicolumn{2}{c}{18A} & 
\multicolumn{2}{c}{21A} & 
\multicolumn{2}{c}{24AB} & 
\multicolumn{2}{c}{24CD} &
\multicolumn{2}{c}{27A} 
& \multicolumn{2}{c}{27BC}  
& \multicolumn{2}{c}{30AB} 
& \multicolumn{2}{c}{36ABC}
& \multicolumn{2}{c}{39AB} 
\\\midrule
$\frac{1}{8}$ & $\frac{\omega^{2}}{q^{3}}$& 
$\frac{1}{42}$ & $\frac{-1}{q^{3}}$& 
$\frac{1}{32}$ & $\frac{\omega^{5}}{q}$& 
$\frac{1}{32}$ & $\frac{i \, \omega^{2}}{q}$ &
$\frac{1}{54}$ & $\frac{-1}{q^{3}}$
& $\frac{1}{54}$ & $\frac{-1}{q^{3}}$
& $\frac{1}{40}$ & $\frac{\omega}{q^{3}}$
& $\frac{1}{16}$ & $\frac{\omega^{4}}{q^{3}}$ 
& $\frac{1}{78}$ & $\frac{-1}{q^{3}}$ 
\\[3pt]
\multicolumn{1}{c}{} & \multicolumn{1}{c}{} & 
$\frac{1}{28}$ & $\frac{\omega^{4}}{q^{3}}$& 
\multicolumn{1}{c}{} & \multicolumn{1}{c}{} & 
\multicolumn{1}{c}{} & \multicolumn{1}{c}{} &
$\frac{1}{36}$ & $\frac{\omega^{5}}{q^{3}}$
& $\frac{1}{36}$ & $\frac{\omega^{5}}{q^{3}}$
& \multicolumn{1}{c}{} & \multicolumn{1}{c}{} 
& \multicolumn{1}{c}{} & \multicolumn{1}{c}{} 
& $\frac{1}{52}$ & $\frac{\omega^{4}}{q^{3}}$ 
\\[3pt]
\multicolumn{1}{c}{} & \multicolumn{1}{c}{} & 
$\frac{1}{14}$ & $\frac{\omega^{5}}{q^{3}}$& 
\multicolumn{1}{c}{} & \multicolumn{1}{c}{} & 
\multicolumn{1}{c}{} & \multicolumn{1}{c}{} &
$\frac{1}{18}$ & $\frac{\omega^{4}}{q^{3}}$
& $\frac{1}{18}$ & $\frac{\omega^{4}}{q^{3}}$
& \multicolumn{1}{c}{} & \multicolumn{1}{c}{} 
& \multicolumn{1}{c}{} & \multicolumn{1}{c}{} 
& $\frac{1}{26}$ & $\frac{\omega^{5}}{q^{3}}$ 
\\[3pt]
\multicolumn{1}{c}{} & \multicolumn{1}{c}{} & 
\multicolumn{1}{c}{} & \multicolumn{1}{c}{} & 
\multicolumn{1}{c}{} & \multicolumn{1}{c}{} &
\multicolumn{1}{c}{} & \multicolumn{1}{c}{} &
$\frac{1}{12}$ & $\frac{\zeta^{24}}{q}$
& $\frac{1}{12}$ & $\frac{\zeta^{-12}}{q}$
& \multicolumn{1}{c}{} & \multicolumn{1}{c}{} 
& \multicolumn{1}{c}{} & \multicolumn{1}{c}{} 
& \multicolumn{1}{c}{} & \multicolumn{1}{c}{} 
\\[3pt]
\multicolumn{1}{c}{} & \multicolumn{1}{c}{} & 
\multicolumn{1}{c}{} & \multicolumn{1}{c}{} & 
\multicolumn{1}{c}{} & \multicolumn{1}{c}{} &
\multicolumn{1}{c}{} & \multicolumn{1}{c}{} &
$\frac{5}{36}$ & $\frac{\omega}{q^{3}}$
& $\frac{5}{36}$ & $\frac{\omega}{q^{3}}$
& \multicolumn{1}{c}{} & \multicolumn{1}{c}{} 
& \multicolumn{1}{c}{} & \multicolumn{1}{c}{} 
& \multicolumn{1}{c}{} & \multicolumn{1}{c}{} 
\\[3pt]
\multicolumn{1}{c}{} & \multicolumn{1}{c}{} & 
\multicolumn{1}{c}{} & \multicolumn{1}{c}{} & 
\multicolumn{1}{c}{} & \multicolumn{1}{c}{} &
\multicolumn{1}{c}{} & \multicolumn{1}{c}{} &
$\frac{1}{6}$ & $\frac{\zeta^{-42}}{q}$
& $\frac{1}{6}$ & $\frac{\zeta^{-6}}{q}$
& \multicolumn{1}{c}{} & \multicolumn{1}{c}{} 
& \multicolumn{1}{c}{} & \multicolumn{1}{c}{} 
& \multicolumn{1}{c}{} & \multicolumn{1}{c}{} 
\\[3pt]
\multicolumn{1}{c}{} & \multicolumn{1}{c}{} & 
\multicolumn{1}{c}{} & \multicolumn{1}{c}{} & 
\multicolumn{1}{c}{} & \multicolumn{1}{c}{} &
\multicolumn{1}{c}{} & \multicolumn{1}{c}{} &
$\frac{1}{4}$ & $\frac{\zeta^{-28}}{q}$
& $\frac{5}{18}$ & $\frac{\omega^{2}}{q^{3}}$
& \multicolumn{1}{c}{} & \multicolumn{1}{c}{} 
& \multicolumn{1}{c}{} & \multicolumn{1}{c}{} 
& \multicolumn{1}{c}{} & \multicolumn{1}{c}{} 
\\[3pt]
\multicolumn{1}{c}{} & \multicolumn{1}{c}{} & 
\multicolumn{1}{c}{} & \multicolumn{1}{c}{} & 
\multicolumn{1}{c}{} & \multicolumn{1}{c}{} &
\multicolumn{1}{c}{} & \multicolumn{1}{c}{} &
$\frac{5}{18}$ & $\frac{\omega^{2}}{q^{3}}$
& $\frac{5}{12}$ & $\frac{\zeta^{-6}}{q}$
& \multicolumn{1}{c}{} & \multicolumn{1}{c}{} 
& \multicolumn{1}{c}{} & \multicolumn{1}{c}{} 
& \multicolumn{1}{c}{} & \multicolumn{1}{c}{} 
\\[3pt]
\multicolumn{1}{c}{} & \multicolumn{1}{c}{} & 
\multicolumn{1}{c}{} & \multicolumn{1}{c}{} & 
\multicolumn{1}{c}{} & \multicolumn{1}{c}{} &
\multicolumn{1}{c}{} & \multicolumn{1}{c}{} &
$\frac{5}{12}$ & $\frac{\zeta^{-42}}{q}$
& $\frac{5}{6}$ & $\frac{\zeta^{-21}}{q}$
& \multicolumn{1}{c}{} & \multicolumn{1}{c}{} 
& \multicolumn{1}{c}{} & \multicolumn{1}{c}{} 
& \multicolumn{1}{c}{} & \multicolumn{1}{c}{} 
\\[3pt]
\multicolumn{1}{c}{} & \multicolumn{1}{c}{} & 
\multicolumn{1}{c}{} & \multicolumn{1}{c}{} & 
\multicolumn{1}{c}{} & \multicolumn{1}{c}{} &
\multicolumn{1}{c}{} & \multicolumn{1}{c}{} &
$\frac{1}{2}$ & $\frac{\zeta^{-14}}{q}$
& \multicolumn{1}{c}{} & \multicolumn{1}{c}{} 
& \multicolumn{1}{c}{} & \multicolumn{1}{c}{} 
& \multicolumn{1}{c}{} & \multicolumn{1}{c}{} 
& \multicolumn{1}{c}{} & \multicolumn{1}{c}{} 
\\[3pt]
\multicolumn{1}{c}{} & \multicolumn{1}{c}{} & 
\multicolumn{1}{c}{} & \multicolumn{1}{c}{} & 
\multicolumn{1}{c}{} & \multicolumn{1}{c}{} &
\multicolumn{1}{c}{} & \multicolumn{1}{c}{} &
$\frac{5}{6}$ & $\frac{\zeta^{15}}{q}$
& \multicolumn{1}{c}{} & \multicolumn{1}{c}{} 
& \multicolumn{1}{c}{} & \multicolumn{1}{c}{} 
& \multicolumn{1}{c}{} & \multicolumn{1}{c}{} 
& \multicolumn{1}{c}{} & \multicolumn{1}{c}{} 
\\[3pt]
 \bottomrule
\end{tabular}
\end{small}
\end{center}
\end{table}

The singular parts of the functions $\frac13\breve{F}_{{\rm 3C},g}$ at non-infinite cusps are given for $3$-singular $g$ (i.e.\ $g$ such that $o(g)\equiv 0 \xmod 3$) in Table \ref{tab:mts-thm:sngF3C}, according to the same conventions, except that $\frac12\breve{F}^{(-3,1)}_g$ should be replaced with $\frac13\breve{F}_{{\rm 3C},g}$ in (\ref{eqn:mts-thm:breveF31atalpha}).
We only include $3$-singular $g$ in Table \ref{tab:mts-thm:sngF3C} because $\frac13\breve{F}_{{\rm 3C},g}$ has the same singular parts as $\frac12\breve{F}^{(-3,1)}_g$ if $o(g)$ is coprime to $3$ (see (\ref{eqn:res-ava:3F31g=2F3Cgplustheta})). 
Also, the singular parts of $\frac13\breve F_{{\rm 3C},g}$ and $\frac12\breve{F}^{(-3,1)}_g$ at the infinite cusp of $\Gamma_0(4N)$ coincide for all $g$, and in addition to $\omega=\ex(\frac16)$ we use $\zeta=\ex(\frac1{108})$ in Table \ref{tab:mts-thm:sngF3C}.
Thus, the singular parts of all the $\frac13\breve{F}_{{\rm 3C},g}$, including those with $o(g)=27$, are determined in Tables \ref{tab:mts-thm:sngF31}--\ref{tab:mts-thm:sngF3C}.

Having determined their singular parts, the functions $\breve{F}^{(-3,1)}_g$ and $\breve{F}_{{\rm 3C},g}$ are precisely determined by finitely many of their Fourier coefficients at perfect square exponents. 
The coefficients given in \S~\ref{app:cff} are sufficient for this purpose.

\subsection{Weight Zero Thompson Moonshine}\label{app:ser-thz}

In this section we specify the McKay--Thompson series $T_{{\rm 3C},g}$ of weight zero 3C-generalized monstrous moonshine. 
We do not directly characterize the McKay--Thompson series $T^{(-3,1)}_g$ of weight zero penumbral moonshine, but we note that they are determined explicitly in terms of the $F^{(-3,1)}_g$ 
by the product formula (\ref{eqn:int-com:T31g}), according to Theorem \ref{thm:res-ava:wz}.

The $T_{{\rm 3C},g}$ are specified in Table \ref{tab:mts-T3C},  
where the first column runs over conjugacy class names, 
the second column indicates the invariance group of the principal modulus $\widetilde{T}_{\mathrm{3C},g}(3\tau)$, the third column indicates the invariance group of $\widetilde{T}_{\mathrm{3C},g}(\tau)^3$ (which turns out to be a Hauptmodul in every case), and the last column gives a description of $\widetilde{T}_{\mathrm{3C},g}$.

The notation in the second and third columns of Table \ref{tab:mts-T3C} is as in \S~\ref{sec:pre-fuc}, except for the classes 18B and 27ABC. 
For 27A and 27BC we follow the notational conventions of \cite{for-mck-nor}. 
For 18B we write
$18+\tilde{2},9,\tilde{18}$ 
for the group generated by $18+9$ and the matrix 
$\left(\begin{smallmatrix}4&-5/3\\6&-2\end{smallmatrix}\right)$, 
and we write
$54|3+\tilde{2},9,\tilde{18}$ 
for the group generated by the kernel of the homomorphism $\lambda:\Gamma_0(54|3)+9\to \ZZ/3\ZZ$ (cf.\ (\ref{eqn:Gamma0NverthplusS})) given by 
\begin{gather}
	\lambda\begin{pmatrix}
		ae&b/3\\
		54c&de
	\end{pmatrix}
	:=\begin{cases}
		a(b+c)\xmod 3&\text{for $e=1$,}\\
		-b(a+d)\xmod 3&\text{for $e=9$,}
	\end{cases}
\end{gather}
and the matrix 
$\left(\begin{smallmatrix}2&-1/9\\90&-4\end{smallmatrix}\right)$.
(The invariance groups of $\widetilde{T}_{\mathrm{3C},g}(3\tau)$ and $\widetilde{T}_{\mathrm{3C},g}(\tau)^3$ for $g$ in the class 18B are not given in op.\ cit., although the latter function is denoted there by 18z.)

\begin{table}[h!]
\begin{small}
\begin{center}
\caption{The McKay--Thompson series $T_{{\rm 3C},g}$.}
\label{tab:mts-T3C}
\begin{tabular}{c  c  c  c}
\toprule
$[g]$ & $\Inv(\widetilde{T}_{\mathrm{3C},g}(3\tau))$ & $\Inv(\widetilde{T}_{\mathrm{3C},g}(\tau)^3)$ & $\widetilde{T}_{\mathrm{3C},g}(\tau)$ \\
\midrule
1A    & $3\vert 3$   		& $1$ 	   				& $j(\tau)^{\frac13}$ \\
2A    & $6\vert 3$   		& $2$ 	   				& $1^8/2^8$ \\
3A    & $9\vert 3+$  		& $3+$ 	   				& $(1^{12}/3^{12}+729\cdot 3^{12}/1^{12} + 54)^{\frac13}$  \\
3B    & $9$         		& $3$ 	   				& $(1/3)^3/3^3+3$ \\
3C    & $9\vert 3$   		& $3$ 	   				& $1^4/3^4$ \\
4A    & $12\vert 3+$ 		& $4+$       			& $2^{16}/1^84^8$  \\
4B    & $12\vert 6$  		& $4\vert 2$ 			& $2^4/4^4$ \\
5A    & $15\vert 3$  		& $5$   				& $1^2/5^2$  \\
6A    & $18\vert 3+6$       & $6+6$ 				& $2^43^4/1^46^4$ \\
6B    & $18\vert 3+3$       & $6+3$ 				& $1^23^2/2^26^2$ \\
6C    & $18$        		& $6$  				& $2^13^3/1^1 6^3$\\
7A    & $21\vert 3+$        & $7+$  				& $(1^4/7^4+49\cdot 7^4/1^4+13)^{\frac13}$ \\
8A    & $24\vert 6+$        & $8\vert 2+$ 			& $4^8/2^48^4$\\
8B    & $24\vert 12$       & $8\vert 4$			& $4^2/8^2$  \\
9A    & $9$        		& $3$ 					& $(1/3)^3/3^3+3$ \\
9B    & $9\vert 3$        	& $3$ 					& $1^4/3^4$ \\
9C    & $27\vert 3+$        & $9+$ 					& $3^4/1^29^2$  \\
10A   & $30\vert 3+10$      & $10+10$ 				& $2^25^2/1^210^2$\\
12AB  & $36\vert 3+$       	& $12+$ 				& $2^46^4/1^23^24^212^2$ \\
12C   & $36+4$       		& $12+4$ 				& $1^14^16^6/2^23^312^3$ \\
12D   & $36\vert 6+6$      	& $12\vert 2 + 6$ 		& $4^16^2/2^212^2$\\
13A   & $39\vert 3+$      	& $13+$ 				& $(1^2/13^2+13\cdot 13^2/1^2+5)^{\frac13}$ \\
14A   & $42\vert 3+7$       & $14+7$ 				& $1^17^1/2^114^1$\\
15AB  & $45\vert 3+15$      & $15+15$ 				& $3^15^1/1^115^1$  \\
18A   & $18$       		& $6$    				& $2^13^3/1^16^3$\\
18B   &     $54\vert 3+\tilde{2},9,\tilde{18}$ 				&  $18+\tilde{2},9,\tilde{18}$						& $1^29^26^4/2^218^23^4$  \\
19A   & $57\vert 3+$       	& $19+$ 				& $q^{-\frac13}(G(\tau)G(19\tau)+q^4H(\tau)H(19\tau))$\\
20A   & $60\vert 6+10$      & $20\vert 2 + 10$		& $4^1 10^1/2^1 20^1$  \\
21A   & $63\vert 3+$       	& $21+$ 				& $(1^1 3^1/7^1 21^1 + 7 \cdot 7^1 21^1/3^1 1^1 + 1)^{\frac13}$  \\
24AB  & $72\vert 6+$      	& $24\vert 2+$ 			& $4^2 12^2/2^1 6^1 8^1 24^1$  \\
24CD  & $72\vert 12+6$  	& $24\vert 4+6$ 	   	& $8^1 12^1/4^1 24^1$  \\
27A   & $27\vert\vert 9 +=18.24.57$      			& $9\vert\vert 3+= 12$  & $(3^4/9^4 + 9 \cdot 9^4/3^4 + 6)^{\frac13}$  \\
27BC  & $27\vert\vert 9 +=12.45.78$   				& $9\vert\vert 3+= 12$  & $(3^4/9^4 + 9 \cdot 9^4/3^4 - 3)^{\frac13}$ \\
28A   & $84\vert 3+$    	& $28+$ 				& $2^2 14^2/1^1 4^1 7^1 28^1$  \\
30AB  & $90\vert 3+6,10,15$ & $30+6,10,15$ 			& $1^1 6^1 10^1 15^1/2^1 3^1 5^1 30^1$ \\
31AB  & $93\vert 3+$        & $31+$ 				& $q^{-\frac13}(H(\tau)G(31\tau)-q^6G(\tau)H(31\tau))$ \\
36ABC & $36+4$     			& $12+4$ 				& $1^1 4^1 6^6/2^2 3^3 12^3$ \\
39AB  & $117\vert 3+$    	& $39+$ 				& $(3^1 13^1/1^1 39^1 + 1^1 39^1/3^1 13^1 + 2)^{\frac13}$  
\\\bottomrule
\end{tabular}
\end{center}
\end{small}
\end{table}

For the notation in the last column of Table \ref{tab:mts-T3C} our conventions are as follows. 
We use $j$ to denote the elliptic modular invariant, as usual. A symbol of the form $c\cdot \prod_{b>0} b^{v_b}$ (cf.\ (\ref{eqn:int-com:etag})) is a proxy for the eta product $c\prod_{b>0} \eta(b\tau)^{v_b}$.
Sums and powers of such symbols indicate sums and powers of the corresponding eta products. Finally, $G$ and $H$ are the Rogers--Ramanujan functions, 
given by 
    $G(\tau) := \sum_{n\geq 0} {q^{n^2}}{(q;q)_n^{-1}}$
    and $H(\tau) := \sum_{n\geq 0} {q^{n^2+n}}{(q;q)_n^{-1}}$,
where 
$(a;q)_n := \prod_{k=0}^{n-1}(1-aq^k)$
is the $q$-Pochhammer symbol.

We used the references \cite{queen,for-mck-nor,Sloane_theencyclopedia}, some help from Simon Norton \cite{NortonUnpublished}, and some direct computation to compile Table \ref{tab:mts-T3C}.

\subsection{Genus Zero Borcherds Products}\label{app:ser-gen}

We now specify the modular forms that are promised by Proposition \ref{pro:generalBorPrds}. 
More specifically, for each genus zero group of the form $\Gamma=\Gamma_0(m){+}$ with $m$ square-free, we construct a function $F=(F_r)$ in $\textsl{V}^{\wh,\alpha}_{\frac12,m}$, for $\a$ the trivial character, 
for which the corresponding Borcherds product (\ref{eqn:generalBorPrds}) takes the form \eqref{eqn:generalBorPrds_aomega}.

Our results are given in Table \ref{tab:bpd:bpp}. The first three columns specify the modular form $F$, which takes the shape ·
\begin{align}\label{eqn:preimageJGamma}
   F= \Omega^\alpha_m (bR_{m,\SL_2(\ZZ),1}^{[\vec{\mu}]}+c\theta_m^0).
\end{align}
Here $\Omega^\alpha_m$ is as in (\ref{eqn:Omegaalpham}), the Rademacher sum $R_{m,\SL_2(\ZZ),1}^{[\vec{\mu}]}$ is as defined in \S~\ref{sec:pre-rad}, we take $\vec{\mu}=(\mu_r)$ to be the vector whose components are 
\begin{gather}\label{eqn:mur}
\mu_r := \frac{D_0}{4m}(\delta_{r,r_0}+\delta_{r,-r_0})
\end{gather} 
for $r\xmod 2m$, 
and the constants $b$ and $c$ are uniquely determined by requiring that the resulting series satisfies $Aq^{\frac{D_0}{4m}}(e_{r_0}+e_{-r_0})+O(q^{\frac{1}{4m}})$ as $\Im(\tau)\to\infty$, where $A$ is given in the third column of the tables. The value of $D_0$ we choose is the least negative one for which the Borcherds product of \eqref{eqn:preimageJGamma} takes the form \eqref{eqn:generalBorPrds_aomega}.

The last three columns describe properties of the Borcherds product (\ref{eqn:generalBorPrds}), after $\tau$ has been rescaled so that its $q$-expansion has integer exponents. 
The invariance group is $\Inv(\Psi)$, and $n\mathrm{Z}$ is the corresponding symbol in the notation of \cite{for-mck-nor}. The last column designates the constant term in the $q$-expansion if it is not zero. 

Note that the $a$ and $\omega$ in \eqref{eqn:generalBorPrds_aomega} may be extracted from Table \ref{tab:bpd:bpp} as follows. For a given $\Gamma$, the value of $\omega$ is the value of $A$ for which $\Inv(\Psi)=\Gamma$, and
the value of $a$ is the constant term associated to this value of $A$.
E.g.\ for $\Gamma=\Gamma_0(3)+$ we have $a=42$ and $\omega=6$.

\begin{center}
\begingroup%  to keep \arraystretch local
\renewcommand{\arraystretch}{0.9}
\begin{small}
\begin{longtable}{cccccc}
\caption{Borcherds Products for Proposition \ref{pro:generalBorPrds}.}\label{tab:bpd:bpp} \\

\multicolumn{1}{c}{$\Gamma$} & 
\multicolumn{1}{c}{$(D_0,r_0)$} & 
\multicolumn{1}{c}{$A$} & 
\multicolumn{1}{c}{$\Inv(\Psi)$} & 
\multicolumn{1}{c}{$n{\rm Z}$} &
\multicolumn{1}{c}{const.}
\\ \hline\hline
\endfirsthead

\multicolumn{6}{c}%
{{\tablename\ \thetable{}, \small{continued.}}} \\
\multicolumn{1}{c}{$\Gamma$} & 
\multicolumn{1}{c}{$(D_0,r_0)$} & 
\multicolumn{1}{c}{$A$} & 
\multicolumn{1}{c}{$\Inv(\Psi)$} & 
\multicolumn{1}{c}{$n{\rm Z}$} &
\multicolumn{1}{c}{const.}
\\ \hline\hline
\endhead

\hline
\endlastfoot

$1-$ & $(-3,1)$ 	 & $1$ &  $3\vert 3$ 				& $3\mathrm{C}$ \\ 
	 &				 & $3$ &  $1$                      & $1\mathrm{A}$ & $744$ \\ 
\hline
$2+$ & $(-4,2)$   & $1$ &  $8\vert 4+$ 				& $8\mathrm{C}$ 	\\ 
	 & 					 & $2$ &  $4\vert 2+$  				& $4\mathrm{B}$ 	\\
	 &					 & $4$ &  $2+$						& $2\mathrm{A}$  & $104$ \\
\hline
$3+$ & $(-3,3)$   & $1$ &  $18\vert 6+\bar{3}$ 			& $18\mathrm{b}$ 	\\
	 &					 & $2$ &  $9\vert 3+$				& $9\mathrm{a}$ 	\\
	 &					 & $3$ &  $6\vert 2+\bar{3}$				& $6\mathrm{b}$ 	\\
	 &					 & $6$ &  $3+$ 						& $3\mathrm{A}$  & $42$ \\
\hline
$5+$ & $(-4,4)$ & $1$ &  $10\vert 2+\bar{5}$ 			& $10\mathrm{a}$ &  \\
     &		 & $2$ &  $5+$ 						& $5\mathrm{A}$  & $16$ \\
\hline
$6+$ & $(-8,4)$  & $1$ & $12\vert 2+\bar{2}$ & $12\mathrm{b}$ &     \\
	 & 					 & $2$ &  $6+$						& $6\mathrm{A}$  & $10$ \\
\hline
$7+$ & $(-3,5)$   & $1$     &  $21\vert 3+$          & $21\mathrm{C}$ \\
     &                 & $3$     &  $7+$       			& $7\mathrm{A}$  & $9$ \\
\hline
$10+$& $(-4,6)$  & $1$ & $40\vert 4+$  & $40\mathrm{A}$ & \\
     &                 & $2$     & $20\vert 2+$           & $20\mathrm{B}$ & \\
     &                  & $4$     & $10+$ 					& $10\mathrm{A}$ & $4$ \\
\hline
$11+$& $(-7,9)$ 	 & $1$     & $11+$					& $11\mathrm{A}$ & $5$ \\
\hline
$13+$& $(-3,7)$  & $1$	   & $39\vert 3+$			& $39\mathrm{B}$ & $3$\\
	&					 & $3$	   & $13+$					& $13\mathrm{A}$ & \\
\hline
$14+$& $(-7,7)$ & $1$  & $28\vert 2+\bar{2},\bar{7},14$   & $28\mathrm{a}$ & \\
    &                   & $2$ 	   & $14+$					& $14\mathrm{A}$ & $2$ \\
\hline
$15+$& $(-11,7)$ & $1$  & $15+$ 				& $15\mathrm{A}$ & $3$ \\
\hline
$17+$& $(-4,8)$  & $1$      & $34\vert 2+\overline{17}$		& $34\mathrm{a}$ & \\
	&					& $2$		& $17+$ 				& $17\mathrm{A}$ & $2$ \\
\hline
$19+$   & $(-3,15)$    & $1$  & $57\vert 3+$ 		& $57\mathrm{A}$ &  \\
		&						& $3$ & $19+$				& $19\mathrm{A}$ & $3$ \\
\hline
$21+$   & $(-3,9)$   & $1$  & $126\vert 6+\bar{3},7,\overline{21}$ & $126\mathrm{a}$ \\
		&						 & $2$	& $63\vert 3+$& $63\mathrm{a}$ & \\
		&						 & $3$	& $42\vert 2+\bar{3},7,\overline{21}$& $42\mathrm{a}$ & \\
		&						 & $6$	& $21+$				& $21\mathrm{A}$ &  \\
\hline
$22+$	& $(-7,9)$   & $1$ & $22+$				& $22\mathrm{A}$ & $1$ \\
\hline
$23+$ 	& $(-7,19)$ 	 & $1$  & $23+$ 			& $23\mathrm{A}$ & $2$ \\
\hline
$26+$   & $(-4,10)$  & $1$ & $104\vert 4+$ & $104\mathrm{A}$ & \\
		&						& $2$  & $52\vert 2+$ 	& $52\mathrm{A}$ & \\
		&						& $4$	& $26+$			& $26\mathrm{A}$ & \\
\hline
$29+$ 	& $(-4,24)$ 	&	$1$	&$58\vert 2+\overline{29}$ & $58\mathrm{a}$ & \\
		&						&	$2$	& $29+$			 & $29\mathrm{A}$ & $2$ \\
\hline
$30+$   & $(-15,15)$ & $1$ & $\begin{array}{c}60\vert 2+\bar{2},\bar{3}, 5,\\ \ \ \ 6,\overline{10},\overline{15},30\end{array}$ & $60\mathrm{b}$ & \\
		&						&	$2$	  & $30+$		& $30\mathrm{B}$ & \\
\hline
$31+$   & $(-3,11)$  & $1$ 			& $93\vert 3+$				& $93\mathrm{A}$ & \\
		&					 & $3$			& $31+$						& $31\mathrm{A}$ & $3$ \\
\hline
$33+$   & $(-8,16)$  & $1$		& $33+$						& $33\mathrm{B}$ & $1$ \\
\hline
$34+$	& $(-4,26)$  & $2$		& $68\vert2+$				& $68\mathrm{A}$ & \\
		&						&	$4$		& $34+$						& $34\mathrm{A}$ & $2$ \\
\hline
$35+$ 	& $(-19,11)$&	$1$		& $35+$						& $35\mathrm{A}$ & $1$ \\
\hline
$38+$	& $(-8,12)$ & $1$     & $76\vert 2+\bar{2},\overline{19},38$		& $76\mathrm{a}$ & \\
		&							& $2$	& $38+$						& $38\mathrm{A}$ & \\
\hline
$39+$ 	& $(-3,33)$ &	$2$		& $117\vert 3+$				& $117\mathrm{a}$ & \\
		&						 &  $6$		& $39+$						& $39\mathrm{A}$ & $3$ \\ 
\hline
$41+$	& $(-4,18)$ 	  & $1$		& $82\vert 2+\overline{41}$			& $82\mathrm{a}$ & \\
		&					  & $2$		& $41+$						& $41\mathrm{A}$ & \\
\hline
$42+$   & $(-12,18)$ & $1$ & $84\vert 2+$ 		& $84\mathrm{A}$ & \\
		&						& $2$		& $42+$					& $42\mathrm{A}$ & \\
\hline
$46+$ 	& $(-7,19)$  & $1$		& $46+$						& $46\mathrm{CD}$&  \\
\hline
$47+$ 	& $(-11,41)$ 	  & $1$		& $47+$						& $47\mathrm{A}$ & $1$ \\
\hline
$51+$   & $(-8,14)$  & $1$		& $51+$ 					& $51\mathrm{A}$ & \\
\hline
$55+$  	& $(-11,33)$  & $2$ 	& $55+$						& $55\mathrm{A}$ & $1$ \\
\hline
$59+$ 	& $(-8,46)$  		& $1$ 	& $59+$						& $59\mathrm{A}$ & $1$ \\
\hline
$62+$ 	& $(-12,22)$ 	& $1$ 	& $62+$ 					& $62\mathrm{AB}$ & \\
\hline
$66+$   & $(-8,16)$  & $2$ & $66+$				& $66\mathrm{A}$ & $-1$ \\
\hline
$69+$	& $(-11,37)$  & $1$ & $69+$ 	& $69\mathrm{AB}$ & $1$ \\
\hline
$70+$ 	& $(-20,50)$  & $2$ & $70+$ 	& $70\mathrm{A}$ & $1$ \\
\hline
$71+$   & $(-7,63)$   			& $1$	&$71+$		& $71\mathrm{AB}$ & $1$\\
\hline
$78+$ 	& $(-12,66)$  & $2$ & $78+$	& $78\mathrm{A} $ & $1$ \\
\hline
$87+$   & $(-24,18)$ & $1$ & $87+$ & $87\mathrm{AB}$ & \\
\hline
$94+$   & $(-20,58)$  & $1$ & $94+$ & $94\mathrm{AB}$ & \\
\hline
$95+$   & $(-15,55)$  & $1$ & $95+$ & $95\mathrm{AB}$ & \\
\hline
$105+$  & $(-24,54)$  & $1$ & $105+$ & $105\mathrm{A}$ & $1$ \\
\hline
$110+$  & $(-24,36)$  & $1$ & $110+$ & $110\mathrm{A}$ & \\
\hline
$119+$  & $(-19,95)$  & $1$ & $119+$ & $119\mathrm{AB}$ & $1$ \\
\hline
\end{longtable}
\end{small}
\end{center}
\endgroup

\clearpage

\thispagestyle{tables}

\pagestyle{tables}

\section{Characters}\label{app:tgd}

We reproduce the character table of $\Th$ in Tables \ref{character_table_th_1}--\ref{character_table_th_2}.
In these tables we use the notation $A = -1+2i\sqrt{3}$, $B = -2+4i\sqrt{3}$, $C = \frac{-1+i\sqrt{15}}{2}$, $D = -i\sqrt{3}$, $E = i\sqrt{6}$, $F = \frac{-1+3i\sqrt{3}}{2}$, $G = \frac{-1+i\sqrt{31}}{2}$, $H = -1+i\sqrt{3}$, and $I = \frac{-1+i\sqrt{39}}{2}$.

\begin{table}[h!]
\begin{center}
\captionsetup{justification=centering}
\caption{Character table for the Thompson group, $\Th$, Part 1.}\label{character_table_th_1}
\begin{tiny}
\begin{tabular}{c@{ }r@{ }r@{ }r@{ }r@{ }r@{ }r@{ }r@{ }r@{ }r@{ }r@{ }r@{ }r@{ }r@{ }r@{ }r@{ }r@{ }r@{ }r@{ }r@{ }r@{ }r@{ }r@{ }r@{ }r} \toprule
$[g]$ & 1A & 2A & 3A & 3B & 3C & 4A & 4B & 5A & 6A & 6B & 6C & 7A & 8A & 8B & 9A & 9B & 9C & 10A & 12A & 12B & 12C & 12D & 13A & 14A \\\midrule
$[g^{2}]$ & 1A & 1A & 3A & 3B & 3C & 2A & 2A & 5A & 3C & 3A & 3B & 7A & 4A & 4B & 9A & 9B & 9C & 5A & 6B & 6B & 6C & 6A & 13A & 7A \\ 
$[g^{3}]$ & 1A & 2A & 1A & 1A & 1A & 4A & 4B & 5A & 2A & 2A & 2A & 7A & 8A & 8B & 3B & 3B & 3B & 10A & 4A & 4A & 4A & 4B & 13A & 14A \\ 
$[g^{5}]$ & 1A & 2A & 3A & 3B & 3C & 4A & 4B & 1A & 6A & 6B & 6C & 7A & 8A & 8B & 9A & 9B & 9C & 2A & 12B & 12A & 12C & 12D & 13A & 14A \\ 
$[g^{7}]$ & 1A & 2A & 3A & 3B & 3C & 4A & 4B & 5A & 6A & 6B & 6C & 1A & 8A & 8B & 9A & 9B & 9C & 10A & 12A & 12B & 12C & 12D & 13A & 2A \\ 
$[g^{13}]$ & 1A & 2A & 3A & 3B & 3C & 4A & 4B & 5A & 6A & 6B & 6C & 7A & 8A & 8B & 9A & 9B & 9C & 10A & 12A & 12B & 12C & 12D & 1A & 14A \\ 
$[g^{19}]$ & 1A & 2A & 3A & 3B & 3C & 4A & 4B & 5A & 6A & 6B & 6C & 7A & 8A & 8B & 9A & 9B & 9C & 10A & 12A & 12B & 12C & 12D & 13A & 14A \\ 
$[g^{31}]$ & 1A & 2A & 3A & 3B & 3C & 4A & 4B & 5A & 6A & 6B & 6C & 7A & 8A & 8B & 9A & 9B & 9C & 10A & 12A & 12B & 12C & 12D & 13A & 14A \\
\bottomrule
\end{tabular}
\end{tiny}
\end{center}
\end{table}

\begin{table}
\begin{center}
\captionsetup{justification=centering}
\caption{Character table for the Thompson group, $\Th$, Part 2.}\label{character_table_th_2}
\begin{tiny}
\begin{tabular}{c@{ }r@{ }r@{ }r@{ }r@{ }r@{ }r@{ }r@{ }r@{ }r@{ }r@{ }r@{ }r@{ }r@{ }r@{ }r@{ }r@{ }r@{ }r@{ }r@{ }r@{ }r@{ }r@{ }r@{ }r} \toprule
$[g]$ & 15A & 15B & 18A & 18B & 19A & 20A & 21A & 24A & 24B & 24C & 24D & 27A & 27B & 27C & 28A & 30A & 30B & 31A & 31B & 36A & 36B & 36C & 39A & 39B \\\midrule
$[g^{2}]$ & 15A & 15B & 9A & 9C & 19A & 10A & 21A & 12A & 12B & 12D & 12D & 27A & 27C & 27B & 14A & 15A & 15B & 31A & 31B & 18A & 18A & 18A & 39A & 39B \\ 
$[g^{3}]$ & 5A & 5A & 6C & 6C & 19A & 20A & 7A & 8A & 8A & 8B & 8B & 9B & 9B & 9B & 28A & 10A & 10A & 31B & 31A & 12C & 12C & 12C & 13A & 13A \\ 
$[g^{5}]$ & 3C & 3C & 18A & 18B & 19A & 4B & 21A & 24B & 24A & 24C & 24D & 27A & 27C & 27B & 28A & 6A & 6A & 31A & 31B & 36A & 36C & 36B & 39A & 39B \\ 
$[g^{7}]$ & 15B & 15A & 18A & 18B & 19A & 20A & 3A & 24A & 24B & 24C & 24D & 27A & 27B & 27C & 4A & 30B & 30A & 31A & 31B & 36A & 36B & 36C & 39B & 39A \\ 
$[g^{13}]$ & 15B & 15A & 18A & 18B & 19A & 20A & 21A & 24A & 24B & 24D & 24C & 27A & 27B & 27C & 28A & 30B & 30A & 31B & 31A & 36A & 36B & 36C & 3A & 3A \\ 
$[g^{19}]$ & 15A & 15B & 18A & 18B & 1A & 20A & 21A & 24A & 24B & 24D & 24C & 27A & 27B & 27C & 28A & 30A & 30B & 31A & 31B & 36A & 36B & 36C & 39B & 39A \\ 
$[g^{31}]$ & 15A & 15B & 18A & 18B & 19A & 20A & 21A & 24A & 24B & 24C & 24D & 27A & 27B & 27C & 28A & 30A & 30B & 1A & 1A & 36A & 36B & 36C & 39B & 39A \\
\bottomrule
\end{tabular}
\end{tiny}
\end{center}
\end{table}

\clearpage

\begin{sidewaystable}
\section{Coefficients}\label{app:cff}
\subsection{Weight Zero 3C-Generalized Thompson Moonshine}
\begin{scriptsize}
\begin{center}
\caption{Coefficient of $q^{\frac{n}{3}}$ in the Fourier expansion of $\widetilde{T}_{\mathrm{3C},g}$, Part 1.}\label{w03Ccoeffs1}
\medskip
\begin{tabular}{c@{ }@{\;}r@{ }r@{ }r@{ }r@{ }r@{ }r@{ }r@{ }r@{ }r@{ }r@{ }r@{ }r@{ }r@{ }r@{ }r@{ }r@{ }r@{ }r@{ }r}\toprule
 $n$ &1A &2A &3A &3B &3C &4A &4B &5A &6A &6B &6C &7A &8A &8B &9A &9B &9C &10A \\
 \midrule 
 $-1$ & $1$ & $1$ & $1$ & $1$ & $1$ & $1$ & $1$ & $1$ & $1$ & $1$ & $1$ & $1$ & $1$ & $1$ & $1$ & $1$ & $1$ & $1$ \\
$2$ & $248$ & $-8$ & $14$ & $5$ & $-4$ & $8$ & $0$ & $-2$ & $4$ & $-2$ & $1$ & $3$ & $0$ & $0$ & $5$ & $-4$ & $2$ & $2$ \\
$5$ & $4124$ & $28$ & $65$ & $-7$ & $2$ & $28$ & $-4$ & $-1$ & $10$ & $1$ & $1$ & $8$ & $4$ & $0$ & $-7$ & $2$ & $5$ & $3$ \\
$8$ & $34752$ & $-64$ & $156$ & $3$ & $12$ & $64$ & $0$ & $2$ & $20$ & $-4$ & $-1$ & $11$ & $0$ & $0$ & $3$ & $12$ & $6$ & $6$ \\
$11$ & $213126$ & $134$ & $456$ & $15$ & $-21$ & $134$ & $6$ & $1$ & $35$ & $8$ & $-1$ & $25$ & $6$ & $-2$ & $15$ & $-21$ & $12$ & $9$ \\
$14$ & $1057504$ & $-288$ & $1066$ & $-32$ & $4$ & $288$ & $0$ & $4$ & $60$ & $-6$ & $0$ & $35$ & $0$ & $0$ & $-32$ & $4$ & $16$ & $12$ \\
$17$ & $4530744$ & $568$ & $2250$ & $9$ & $36$ & $568$ & $-8$ & $-6$ & $100$ & $10$ & $1$ & $57$ & $8$ & $0$ & $9$ & $36$ & $27$ & $18$ \\
\bottomrule
\end{tabular}
\end{center}
\end{scriptsize}
\end{sidewaystable}

\begin{sidewaystable}
\begin{scriptsize}
\begin{center}
\caption{Coefficient of $q^{\frac{n}{3}}$ in the Fourier expansion of $\widetilde{T}_{\mathrm{3C},g}$, Part 2.}\label{w03Ccoeffs2}
\begin{tabular}{c@{ }@{\;}r@{ }r@{ }r@{ }r@{ }r@{ }r@{ }r@{ }r@{ }r@{ }r@{ }r@{ }r@{ }r@{ }r@{ }r@{ }r@{ }r@{ }r@{ }r@{ }r@{ }r@{ }r@{ }r@{ }r@{ }r@{ }r@{ }r@{ }r@{ }r@{ }r@{ }r}\toprule
 $n$ &12A &12B &12C &12D &13A &14A &15A &15B &18A &18B &19A &20A &21A &24A &24B &24C &24D &27A &27B &27C &28A &30A &30B &31A &31B &36A &36B &36C &39A &39B \\
 \midrule 
 $-1$ & $1$ & $1$ & $1$ & $1$ & $1$ & $1$ & $1$ & $1$ & $1$ & $1$ & $1$ & $1$ & $1$ & $1$ & $1$ & $1$ & $1$ & $1$ & $1$ & $1$ & $1$ & $1$ & $1$ & $1$ & $1$ & $1$ & $1$ & $1$ & $1$ & $1$ \\
$2$ & $2$ & $2$ & $-1$ & $0$ & $1$ & $-1$ & $1$ & $1$ & $1$ & $-2$ & $1$ & $0$ & $0$ & $0$ & $0$ & $0$ & $0$ & $2$ & $-1$ & $-1$ & $1$ & $-1$ & $-1$ & $0$ & $0$ & $-1$ & $-1$ & $-1$ & $1$ & $1$ \\
$5$ & $1$ & $1$ & $1$ & $2$ & $3$ & $0$ & $2$ & $2$ & $1$ & $1$ & $1$ & $1$ & $2$ & $1$ & $1$ & $0$ & $0$ & $-1$ & $2$ & $2$ & $0$ & $0$ & $0$ & $1$ & $1$ & $1$ & $1$ & $1$ & $0$ & $0$ \\
$8$ & $4$ & $4$ & $1$ & $0$ & $3$ & $-1$ & $2$ & $2$ & $-1$ & $2$ & $1$ & $0$ & $2$ & $0$ & $0$ & $0$ & $0$ & $0$ & $3$ & $3$ & $1$ & $0$ & $0$ & $1$ & $1$ & $1$ & $1$ & $1$ & $0$ & $0$ \\
$11$ & $8$ & $8$ & $-1$ & $3$ & $4$ & $1$ & $4$ & $4$ & $-1$ & $-4$ & $3$ & $1$ & $1$ & $0$ & $0$ & $1$ & $1$ & $3$ & $0$ & $0$ & $1$ & $0$ & $0$ & $1$ & $1$ & $-1$ & $-1$ & $-1$ & $1$ & $1$ \\
$14$ & $6$ & $6$ & $0$ & $0$ & $6$ & $-1$ & $4$ & $4$ & $0$ & $0$ & $2$ & $0$ & $2$ & $0$ & $0$ & $0$ & $0$ & $-2$ & $1$ & $1$ & $1$ & $0$ & $0$ & $1$ & $1$ & $0$ & $0$ & $0$ & $0$ & $0$ \\
$17$ & $10$ & $10$ & $1$ & $4$ & $10$ & $1$ & $6$ & $6$ & $1$ & $7$ & $4$ & $2$ & $3$ & $2$ & $2$ & $0$ & $0$ & $3$ & $3$ & $3$ & $1$ & $0$ & $0$ & $1$ & $1$ & $1$ & $1$ & $1$ & $1$ & $1$ \\
\bottomrule
\end{tabular}
\end{center}
\end{scriptsize}
\end{sidewaystable}

\clearpage

\begin{sidewaystable}
\subsection{Weight Zero Penumbral Thompson Moonshine}
\begin{scriptsize}
\begin{center}
\caption{Coefficient of $q^{\frac{n}{3}}$ in the Fourier expansion of $T^{(-3,1)}_{g}$, Part 1.}\label{w0PTcoeffs1}
\medskip
\begin{tabular}{c@{ }@{\;}r@{ }r@{ }r@{ }r@{ }r@{ }r@{ }r@{ }r@{ }r@{ }r}\toprule
 $n$ &1A &2A &3A &3B &3C &4A &4B &5A &6A \\
 \midrule 
 $-2$ & $1$ & $1$ & $1$ & $1$ & $1$ & $1$ & $1$ & $1$ & $1$ \\
$1$ & $496$ & $-16$ & $28$ & $10$ & $-8$ & $16$ & $0$ & $-4$ & $8$ \\
$4$ & $69752$ & $120$ & $488$ & $11$ & $20$ & $120$ & $-8$ & $2$ & $36$ \\
$7$ & $2115008$ & $-576$ & $6668$ & $-64$ & $8$ & $576$ & $0$ & $8$ & $120$ \\
$10$ & $34670620$ & $2076$ & $77626$ & $109$ & $-134$ & $2076$ & $28$ & $-5$ & $330$ \\
$13$ & $394460000$ & $-6304$ & $798848$ & $44$ & $224$ & $6304$ & $0$ & $0$ & $800$ \\
$16$ & $3499148224$ & $17344$ & $7421086$ & $-503$ & $100$ & $17344$ & $-64$ & $-26$ & $1780$ \\
\bottomrule
\end{tabular}
\end{center}
\end{scriptsize}
\end{sidewaystable}

\begin{sidewaystable}
\begin{scriptsize}
\begin{center}
\caption{Coefficient of $q^{\frac{n}{3}}$ in the Fourier expansion of $T^{(-3,1)}_{g}$, Part 2.}\label{w0PTcoeffs2}
\begin{tabular}{c@{ }@{\;}r@{ }r@{ }r@{ }r@{ }r@{ }r@{ }r@{ }r@{ }r@{ }r@{ }r@{ }r@{ }r@{ }r@{ }r@{ }r}\toprule
 $n$ &6B &6C &7A &8A &8B &9A &9B &9C &10A &12A &12B &12C &12D &13A &14A \\
 \midrule 
 $-2$ & $1$ & $1$ & $1$ & $1$ & $1$ & $1$ & $1$ & $1$ & $1$ & $1$ & $1$ & $1$ & $1$ & $1$ & $1$ \\
$1$ & $-4$ & $2$ & $6$ & $0$ & $0$ & $10$ & $-8$ & $4$ & $4$ & $4$ & $4$ & $-2$ & $0$ & $2$ & $-2$ \\
$4$ & $24$ & $3$ & $25$ & $8$ & $0$ & $65$ & $20$ & $14$ & $10$ & $12$ & $12$ & $3$ & $4$ & $7$ & $1$ \\
$7$ & $-84$ & $0$ & $70$ & $0$ & $0$ & $314$ & $8$ & $50$ & $24$ & $36$ & $36$ & $0$ & $0$ & $12$ & $-2$ \\
$10$ & $330$ & $-3$ & $180$ & $28$ & $-4$ & $1270$ & $-134$ & $145$ & $51$ & $102$ & $102$ & $-3$ & $10$ & $23$ & $4$ \\
$13$ & $-1024$ & $-4$ & $396$ & $0$ & $0$ & $4472$ & $224$ & $392$ & $96$ & $256$ & $256$ & $4$ & $0$ & $38$ & $-4$ \\
$16$ & $3262$ & $1$ & $845$ & $64$ & $0$ & $14239$ & $100$ & $985$ & $174$ & $598$ & $598$ & $1$ & $20$ & $65$ & $5$ \\
\bottomrule
\end{tabular}
\end{center}
\end{scriptsize}
\end{sidewaystable}

\begin{sidewaystable}
\begin{scriptsize}
\begin{center}
\caption{Coefficient of $q^{\frac{n}{3}}$ in the Fourier expansion of $T^{(-3,1)}_{g}$, Part 3.}\label{w0PTcoeffs3}
\begin{tabular}{c@{ }@{\;}r@{ }r@{ }r@{ }r@{ }r@{ }r@{ }r@{ }r@{ }r@{ }r@{ }r@{ }r@{ }r@{ }r@{ }r@{ }r@{ }r@{ }r@{ }r@{ }r@{ }r@{ }r@{ }r@{ }r@{ }r}\toprule
 $n$ &15A &15B &18A &18B &19A &20A &21A &24A &24B &24C &24D &27A &27B &27C &28A &30A &30B &31A &31B &36A &36B &36C &39A &39B \\
 \midrule 
 $-2$ & $1$ & $1$ & $1$ & $1$ & $1$ & $1$ & $1$ & $1$ & $1$ & $1$ & $1$ & $1$ & $1$ & $1$ & $1$ & $1$ & $1$ & $1$ & $1$ & $1$ & $1$ & $1$ & $1$ & $1$ \\
$1$ & $2$ & $2$ & $2$ & $-4$ & $2$ & $0$ & $0$ & $0$ & $0$ & $0$ & $0$ & $4$ & $-2$ & $-2$ & $2$ & $-2$ & $-2$ & $0$ & $0$ & $-2$ & $-2$ & $-2$ & $2$ & $2$ \\
$4$ & $5$ & $5$ & $9$ & $6$ & $3$ & $2$ & $-2$ & $2$ & $2$ & $0$ & $0$ & $14$ & $-1$ & $-1$ & $1$ & $1$ & $1$ & $2$ & $2$ & $-3$ & $3$ & $3$ & $7$ & $7$ \\
$7$ & $8$ & $8$ & $18$ & $-6$ & $4$ & $0$ & $4$ & $0$ & $0$ & $0$ & $0$ & $38$ & $-1$ & $-1$ & $2$ & $0$ & $0$ & $2$ & $2$ & $6$ & $-6$ & $-6$ & $12$ & $12$ \\
$10$ & $16$ & $16$ & $54$ & $9$ & $9$ & $3$ & $3$ & $4$ & $4$ & $2$ & $2$ & $97$ & $7$ & $7$ & $4$ & $0$ & $0$ & $3$ & $3$ & $6$ & $12$ & $12$ & $29$ & $29$ \\
$13$ & $24$ & $24$ & $104$ & $-16$ & $12$ & $0$ & $-6$ & $0$ & $0$ & $0$ & $0$ & $224$ & $5$ & $5$ & $4$ & $0$ & $0$ & $4$ & $4$ & $-8$ & $-20$ & $-20$ & $50$ & $50$ \\
$16$ & $40$ & $40$ & $247$ & $25$ & $19$ & $6$ & $8$ & $10$ & $10$ & $0$ & $0$ & $493$ & $-8$ & $-8$ & $5$ & $0$ & $0$ & $5$ & $5$ & $-5$ & $31$ & $31$ & $101$ & $101$ \\
\bottomrule
\end{tabular}
\end{center}
\end{scriptsize}
\end{sidewaystable}

\begin{table}
\subsection{Weight One-Half 3C-Generalized Thompson Moonshine}
\begin{tiny}
\begin{center}
\centering
\caption{Coefficient of $q^{D}$ in the Fourier expansion of $\breve{F}_{{\rm 3C},g}$, Part 1.}\label{whalf3Ccoeffs1}
\medskip
\begin{tabular}{c@{ }@{\;}r@{ }r@{ }r@{ }r@{ }r@{ }r@{ }r@{ }r}\toprule
$D $ &1A &2A &3A &3B &3C &4A &4B \\
\midrule
$-3$ & $3$ & $3$ & $3$ & $3$ & $3$ & $3$ & $3$ \\
$1$ & $-744$ & $24$ & $-42$ & $-15$ & $12$ & $-24$ & $0$ \\
$4$ & $80256$ & $384$ & $120$ & $66$ & $12$ & $0$ & $0$ \\
$5$ & $-257985$ & $63$ & $0$ & $81$ & $162$ & $63$ & $-33$ \\
$8$ & $5121792$ & $-2304$ & $0$ & $-162$ & $-324$ & $0$ & $0$ \\
$9$ & $-12288744$ & $24$ & $-234$ & $9$ & $252$ & $-24$ & $0$ \\
$12$ & $132991488$ & $10752$ & $504$ & $18$ & $-468$ & $0$ & $0$ \\
\bottomrule
\end{tabular}
\end{center}
\end{tiny}
\end{table}

\clearpage
\begin{table}
\begin{tiny}
\begin{center}
\caption{Coefficient of $q^{D}$ in the Fourier expansion of $\breve{F}_{{\rm 3C},g}$, Part 2.}\label{whalf3Ccoeffs2}
\begin{tabular}{c@{ }@{\;}r@{ }r@{ }r@{ }r@{ }r@{ }r@{ }r@{ }r@{ }r@{ }r@{ }r@{ }r@{ }r@{ }r@{ }r@{ }r@{ }r@{ }r}\toprule
$D$ &5A &6A &6B &6C &7A &8A &8B &9A &9B &9C &10A &12A &12B &12C &12D &13A &14A \\
\midrule
$-3$ & $3$ & $3$ & $3$ & $3$ & $3$ & $3$ & $3$ & $3$ & $3$ & $3$ & $3$ & $3$ & $3$ & $3$ & $3$ & $3$ & $3$ \\
$1$ & $6$ & $-12$ & $6$ & $-3$ & $-9$ & $0$ & $0$ & $-15$ & $12$ & $-6$ & $-6$ & $-6$ & $-6$ & $3$ & $0$ & $-3$ & $3$ \\
$4$ & $6$ & $-12$ & $24$ & $6$ & $-6$ & $0$ & $0$ & $66$ & $12$ & $-6$ & $-6$ & $0$ & $0$ & $0$ & $0$ & $-6$ & $6$ \\
$5$ & $15$ & $18$ & $0$ & $9$ & $0$ & $-9$ & $3$ & $-81$ & $0$ & $0$ & $3$ & $0$ & $0$ & $9$ & $-6$ & $0$ & $0$ \\
$8$ & $42$ & $36$ & $0$ & $18$ & $18$ & $0$ & $0$ & $162$ & $0$ & $0$ & $6$ & $0$ & $0$ & $0$ & $0$ & $0$ & $6$ \\
$9$ & $6$ & $-12$ & $6$ & $-3$ & $15$ & $0$ & $0$ & $-234$ & $9$ & $9$ & $-6$ & $-6$ & $-6$ & $3$ & $0$ & $0$ & $3$ \\
$12$ & $-12$ & $-12$ & $24$ & $6$ & $0$ & $0$ & $0$ & $504$ & $18$ & $18$ & $12$ & $0$ & $0$ & $0$ & $0$ & $6$ & $0$ \\
\bottomrule
\end{tabular}
\end{center}
\end{tiny}
\end{table}

\clearpage
\begin{table}
\begin{tiny}
\begin{center}
\caption{Coefficient of $q^{D}$ in the Fourier expansion of $\breve{F}_{{\rm 3C},g}$, Part 3.}\label{whalf3Ccoeffs3}
\begin{tabular}{c@{ }@{\;}r@{ }r@{ }r@{ }r@{ }r@{ }r@{ }r@{ }r@{ }r@{ }r@{ }r@{ }r@{ }r@{ }r@{ }r@{ }r@{ }r@{ }r@{ }r@{ }r@{ }r@{ }r@{ }r@{ }r@{ }r}\toprule
$D$ &15A &15B &18A &18B &19A &20A &21A &24A &24B &24C &24D &27A &27B &27C &28A &30A &30B &31A &31B &36A &36B &36C &39A &39B \\
\midrule
$-3$ & $3$ & $3$ & $3$ & $3$ & $3$ & $3$ & $3$ & $3$ & $3$ & $3$ & $3$ & $3$ & $3$ & $3$ & $3$ & $3$ & $3$ & $3$ & $3$ & $3$ & $3$ & $3$ & $3$ & $3$ \\
$1$ & $-3$ & $-3$ & $-3$ & $6$ & $-3$ & $0$ & $0$ & $0$ & $0$ & $0$ & $0$ & $-6$ & $3$ & $3$ & $-3$ & $3$ & $3$ & $0$ & $0$ & $3$ & $3$ & $3$ & $-3$ & $-3$ \\
$4$ & $-3$ & $-3$ & $6$ & $6$ & $0$ & $0$ & $-6$ & $0$ & $0$ & $0$ & $0$ & $12$ & $-6$ & $-6$ & $0$ & $3$ & $3$ & $-3$ & $-3$ & $0$ & $0$ & $0$ & $3$ & $3$ \\
$5$ & $-3$ & $-3$ & $-9$ & $0$ & $-3$ & $-3$ & $0$ & $0$ & $0$ & $-6$ & $-6$ & $0$ & $0$ & $0$ & $0$ & $3$ & $3$ & $-3$ & $-3$ & $-9$ & $-9$ & $-9$ & $0$ & $0$ \\
$8$ & $6$ & $6$ & $-18$ & $0$ & $0$ & $0$ & $0$ & $0$ & $0$ & $0$ & $0$ & $0$ & $0$ & $0$ & $0$ & $6$ & $6$ & $3$ & $3$ & $0$ & $0$ & $0$ & $0$ & $0$ \\
$9$ & $-3$ & $-3$ & $6$ & $-3$ & $0$ & $0$ & $-3$ & $0$ & $0$ & $0$ & $0$ & $-18$ & $-18$ & $-18$ & $-3$ & $3$ & $3$ & $-3$ & $-3$ & $-6$ & $-6$ & $-6$ & $0$ & $0$ \\
$12$ & $-3$ & $-3$ & $24$ & $6$ & $0$ & $0$ & $0$ & $0$ & $0$ & $0$ & $0$ & $18$ & $18$ & $18$ & $0$ & $3$ & $3$ & $0$ & $0$ & $0$ & $0$ & $0$ & $-3$ & $-3$ \\
\bottomrule
\end{tabular}
\end{center}
\end{tiny}
\end{table}
\clearpage

\begin{sidewaystable}
\subsection{Weight One-Half Penumbral Thompson Moonshine}
\begin{tiny}
\begin{center}
\caption{Coefficient of $q^{\frac{D}{4}}$ in the Fourier expansion of ${F}^{(-3,1)}_{g,r}$, Part 1.}\label{mequals1D0m3coefficients1}
\begin{tabular}{c@{ }c@{ }@{\;}r@{ }r@{ }r@{ }r@{ }r@{ }r@{ }r@{ }r@{ }r@{ }r@{ }r@{ }r@{ }r@{ }r@{ }r@{ }r@{ }r@{ }r@{ }r}\toprule
$r~~$ &  $D~\backslash ~[g]$ &1A &2A &3A &3B &3C &4A &4B &5A &6A &6B &6C &7A &8A &8B &9A &9B &9C &10A \\
 \midrule 
$0~~$
 &  $0$ & $248$ & $-8$ & $14$ & $5$ & $-4$ & $8$ &  .& $-2$ & $4$ & $-2$ & $1$ & $3$ &  .&  .& $5$ & $-4$ & $2$ & $2$ \\
 & $4$ & $54000$ & $240$ & $-54$ & $54$ &  .& $16$ &  .&  .&  .& $-6$ & $6$ & $2$ &  .&  .&  .&  .&  .&  .\\
 & $8$ & $3414528$ & $-1536$ &  .& $-108$ & $108$ &  .&  .& $28$ & $-12$ &  .& $12$ & $12$ &  .&  .&  .&  .&  .& $4$ \\
 & $12$ & $88660992$ & $7168$ & $336$ & $12$ & $-312$ &  .&  .& $-8$ & $-8$ & $16$ & $4$ &  .&  .&  .& $12$ & $12$ & $-6$ & $8$ \\
 & $16$ & $1417878000$ & $-26640$ & $-702$ & $702$ &  .& $16$ &  .&  .&  .& $18$ & $-18$ & $-14$ &  .&  .&  .&  .&  .&  .\\
 & $20$ & $16555069440$ & $86016$ &  .& $-1512$ & $1512$ &  .&  .& $-60$ & $24$ &  .& $-24$ &  .&  .&  .&  .&  .&  .& $-4$ \\
 & $24$ & $154184576000$ & $-250880$ & $2912$ & $-4$ & $-2920$ &  .&  .&  .& $40$ & $-32$ & $4$ &  .&  .&  .& $-4$ & $-4$ & $-4$ &  .\\
\bottomrule

\end{tabular}

\end{center}
\end{tiny}
\end{sidewaystable}

\clearpage

\begin{sidewaystable}
\begin{tiny}
\begin{center}
\caption{Coefficient of $q^{\frac{D}{4}}$ in the Fourier expansion of ${F}^{(-3,1)}_{g,r}$, Part 2.}\label{mequals1D0m3coefficients2}
\begin{tabular}{c@{ }c@{ }@{\;}r@{ }r@{ }r@{ }r@{ }r@{ }r@{ }r@{ }r@{ }r@{ }r@{ }r@{ }r@{ }r@{ }r@{ }r@{ }r@{ }r@{ }r@{ }r@{ }r@{ }r@{ }r@{ }r@{ }r@{ }r@{ }r@{ }r@{ }r@{ }r@{ }r@{ }r}\toprule
$r~~$ &  $D~\backslash ~[g]$ &12A &12B &12C &12D &13A &14A &15A &15B &18A &18B &19A &20A &21A &24A &24B &24C &24D &27A &27B &27C &28A &30A &30B &31A &31B &36A &36B &36C &39A &39B \\
 \midrule 
$0~~$
 &  $0$ & $2$ & $2$ & $-1$ &  .& $1$ & $-1$ & $1$ & $1$ & $1$ & $-2$ & $1$ &  .&  .&  .&  .&  .&  .& $2$ & $-1$ & $-1$ & $1$ & $-1$ & $-1$ &  .&  .& $-1$ & $-1$ & $-1$ & $1$ & $1$ \\
 & $4$ & $-2$ & $-2$ & $-2$ &  .& $-2$ & $2$ &  .&  .&  .&  .& $2$ &  .& $2$ &  .&  .&  .&  .&  .&  .&  .& $2$ &  .&  .& $-2$ & $-2$ & $4$ & $-2$ & $-2$ & $-2$ & $-2$ \\
 & $8$ &  .&  .&  .&  .&  .& $4$ & $-2$ & $-2$ &  .&  .&  .&  .&  .&  .&  .&  .&  .&  .&  .&  .&  .& $-2$ & $-2$ & $2$ & $2$ &  .&  .&  .&  .&  .\\
 & $12$ &  .&  .&  .&  .& $4$ &  .& $-2$ & $-2$ & $4$ & $-2$ &  .&  .&  .&  .&  .&  .&  .&  .&  .&  .&  .& $2$ & $2$ &  .&  .&  .&  .&  .& $-2$ & $-2$ \\
 & $16$ & $-2$ & $-2$ & $-2$ &  .& $6$ & $2$ &  .&  .&  .&  .& $-2$ &  .& $-2$ &  .&  .&  .&  .&  .&  .&  .& $2$ &  .&  .&  .&  .& $4$ & $-2$ & $-2$ &  .&  .\\
 & $20$ &  .&  .&  .&  .&  .&  .& $-3$ & $-3$ &  .&  .& $4$ &  .&  .&  .&  .&  .&  .&  .&  .&  .&  .& $-1$ & $-1$ & $2$ & $2$ &  .&  .&  .&  .&  .\\
 & $24$ &  .&  .&  .&  .&  .&  .&  .&  .& $4$ & $4$ & $4$ &  .&  .&  .&  .&  .&  .& $-4$ & $-4$ & $-4$ &  .&  .&  .&  .&  .&  .&  .&  .&  .&  .\\
\bottomrule

\end{tabular}

\end{center}
\end{tiny}
\end{sidewaystable}

\clearpage

\begin{sidewaystable}[h!]
\section{Decompositions}\label{app:dec}
\subsection{Weight Zero 3C-Generalized Thompson Moonshine}
\begin{scriptsize}
\begin{center}
\caption{Decomposition of $(V_{\mathrm{3C}}^\natural)_{\frac n9}$ into irreducible modules for $\Th$, Part 1.}\label{w03Cdecomps1}
\medskip
\begin{tabular}{c@{ }@{\;}r@{ }r@{ }r@{ }r@{ }r@{ }r@{ }r@{ }r@{ }r@{ }r@{ }r@{ }r@{ }r@{ }r@{ }r@{ }r@{ }r}\toprule
 $n$ & $V_{1}$ & $V_{2}$ & $V_{3}$ & $V_{4}$ & $V_{5}$ & $V_{6}$ & $V_{7}$ & $V_{8}$ & $V_{9}$ & $V_{10}$ & $V_{11}$ & $V_{12}$ & $V_{13}$ & $V_{14}$ & $V_{15}$ & $V_{16}$ \\
 \midrule 
 $-1$ & $1$ &  &  &  &  &  &  &  &  &  &  &  &  &  &  &  \\
$2$ &  & $1$ &  &  &  &  &  &  &  &  &  &  &  &  &  &  \\
$5$ & $1$ &  & $1$ &  &  &  &  &  &  &  &  &  &  &  &  &  \\
$8$ & $1$ &  & $1$ &  &  & $1$ &  &  &  &  &  &  &  &  &  &  \\
$11$ & $1$ & $2$ & $1$ &  &  &  &  & $1$ &  &  & $1$ &  &  &  &  &  \\
$14$ & $1$ & $1$ & $2$ &  &  & $1$ &  & $1$ &  &  &  &  &  &  &  & $1$ \\
$17$ & $3$ & $1$ & $3$ &  &  & $2$ & $2$ & $1$ &  &  &  &  &  &  &  & $1$ \\
\bottomrule
\end{tabular}
\end{center}
\end{scriptsize}
\end{sidewaystable}

\clearpage

\begin{sidewaystable}
\begin{scriptsize}
\begin{center}
\caption{Decomposition of $(V_{\mathrm{3C}}^\natural)_{\frac n9}$ into irreducible modules for $\Th$, Part 2.}\label{w03Cdecomps2}
\begin{tabular}{c@{ }@{\;}r@{ }r@{ }r@{ }r@{ }r@{ }r@{ }r@{ }r@{ }r@{ }r@{ }r@{ }r@{ }r@{ }r@{ }r@{ }r@{ }r}\toprule
 $n$ & $V_{17}$ & $V_{18}$ & $V_{19}$ & $V_{20}$ & $V_{21}$ & $V_{22}$ & $V_{23}$ & $V_{24}$ & $V_{25}$ & $V_{26}$ & $V_{27}$ & $V_{28}$ & $V_{29}$ & $V_{30}$ & $V_{31}$ & $V_{32}$ \\
 \midrule 
 $-1$ &  &  &  &  &  &  &  &  &  &  &  &  &  &  &  &  \\
$2$ &  &  &  &  &  &  &  &  &  &  &  &  &  &  &  &  \\
$5$ &  &  &  &  &  &  &  &  &  &  &  &  &  &  &  &  \\
$8$ &  &  &  &  &  &  &  &  &  &  &  &  &  &  &  &  \\
$11$ &  &  &  &  &  &  &  &  &  &  &  &  &  &  &  &  \\
$14$ &  &  &  &  &  &  &  &  &  &  &  &  &  &  &  &  \\
$17$ &  &  &  &  & $1$ &  &  &  &  &  &  &  &  &  &  &  \\
\bottomrule
\end{tabular}
\end{center}
\end{scriptsize}
\end{sidewaystable}

\clearpage

\begin{sidewaystable}
\begin{scriptsize}
\begin{center}
\caption{Decomposition of $(V_{\mathrm{3C}}^\natural)_{\frac n9}$ into irreducible modules for $\Th$, Part 3.}\label{w03Cdecomps3}
\begin{tabular}{c@{ }@{\;}r@{ }r@{ }r@{ }r@{ }r@{ }r@{ }r@{ }r@{ }r@{ }r@{ }r@{ }r@{ }r@{ }r@{ }r@{ }r@{ }r}\toprule
 $n$ & $V_{33}$ & $V_{34}$ & $V_{35}$ & $V_{36}$ & $V_{37}$ & $V_{38}$ & $V_{39}$ & $V_{40}$ & $V_{41}$ & $V_{42}$ & $V_{43}$ & $V_{44}$ & $V_{45}$ & $V_{46}$ & $V_{47}$ & $V_{48}$ \\
 \midrule 
 $-1$ &  &  &  &  &  &  &  &  &  &  &  &  &  &  &  &  \\
$2$ &  &  &  &  &  &  &  &  &  &  &  &  &  &  &  &  \\
$5$ &  &  &  &  &  &  &  &  &  &  &  &  &  &  &  &  \\
$8$ &  &  &  &  &  &  &  &  &  &  &  &  &  &  &  &  \\
$11$ &  &  &  &  &  &  &  &  &  &  &  &  &  &  &  &  \\
$14$ &  &  &  &  &  &  &  &  &  &  &  &  &  &  &  &  \\
$17$ &  &  &  &  &  &  &  &  &  &  &  &  &  &  &  &  \\
\bottomrule
\end{tabular}
\end{center}
\end{scriptsize}
\end{sidewaystable}

\clearpage

\begin{sidewaystable}
\subsection{Weight Zero Penumbral Moonshine}
\begin{tiny}
\begin{center}
\caption{Decomposition of $V^{(-3,1)}_{n}$ into irreducible modules for $\Th$, Part 1.}\label{w0PTdecomps1}
\medskip
\begin{tabular}{c@{ }@{\;}r@{ }r@{ }r@{ }r@{ }r@{ }r@{ }r}\toprule
 $n$ & $V_{1}$ & $V_{2}$ & $V_{3}$ & $V_{4}$ & $V_{5}$ & $V_{6}$ \\
 \midrule 
 $-2$ & $1$ &  &  &  &  &  \\
$1$ &  & $2$ &  &  &  &  \\
$4$ & $3$ & $2$ &  & $-1$ & $-1$ & $1$ \\
$7$ & $4$ & $8$ &  &  &  & $4$ \\
$10$ & $14$ & $16$ &  & $-3$ & $-3$ & $8$ \\
$13$ & $20$ & $44$ &  & $-4$ & $-4$ & $40$ \\
$16$ & $56$ & $92$ & $24$ & $-32$ & $-32$ & $113$ \\
\bottomrule
\end{tabular}
\end{center}
\end{tiny}
\end{sidewaystable}

\clearpage

\begin{sidewaystable}
\begin{tiny}
\begin{center}
\caption{Decomposition of $V^{(-3,1)}_{n}$ into irreducible modules for $\Th$, Part 2.}\label{w0PTdecomps2}
\begin{tabular}{c@{ }@{\;}r@{ }r@{ }r@{ }r@{ }r@{ }r@{ }r@{ }r}\toprule
 $n$ & $V_{7}$ & $V_{8}$ & $V_{9}$ & $V_{10}$ & $V_{11}$ & $V_{12}$ & $V_{13}$ \\
 \midrule 
 $-2$ &  &  &  &  &  &  &  \\
$1$ &  &  &  &  &  &  &  \\
$4$ & $3$ &  &  &  &  &  &  \\
$7$ & $4$ &  &  &  & $2$ &  &  \\
$10$ & $17$ & $2$ &  &  & $13$ &  &  \\
$13$ & $40$ & $6$ &  &  & $32$ &  &  \\
$16$ & $179$ & $110$ &  &  & $187$ &  &  \\
\bottomrule
\end{tabular}
\end{center}
\end{tiny}
\end{sidewaystable}

\clearpage

\begin{sidewaystable}
\begin{tiny}
\begin{center}
\caption{Decomposition of $V^{(-3,1)}_{n}$ into irreducible modules for $\Th$, Part 3.}\label{w0PTdecomps3}
\begin{tabular}{c@{ }@{\;}r@{ }r@{ }r@{ }r@{ }r@{ }r@{ }r@{ }r}\toprule
 $n$ & $V_{14}$ & $V_{15}$ & $V_{16}$ & $V_{17}$ & $V_{18}$ & $V_{19}$ & $V_{20}$ \\
 \midrule 
 $-2$ &  &  &  &  &  &  &  \\
$1$ &  &  &  &  &  &  &  \\
$4$ &  &  &  &  &  &  &  \\
$7$ &  &  &  &  &  &  & $4$ \\
$10$ &  &  &  &  &  & $8$ & $9$ \\
$13$ & $-10$ & $-10$ & $32$ &  &  & $20$ & $82$ \\
$16$ & $-95$ & $-95$ & $327$ &  &  & $202$ & $429$ \\
\bottomrule
\end{tabular}
\end{center}
\end{tiny}
\end{sidewaystable}

\clearpage

\begin{sidewaystable}
\begin{tiny}
\begin{center}
\caption{Decomposition of $V^{(-3,1)}_{n}$ into irreducible modules for $\Th$, Part 4.}\label{w0PTdecomps4}
\begin{tabular}{c@{ }@{\;}r@{ }r@{ }r@{ }r@{ }r@{ }r@{ }r@{ }r}\toprule
 $n$ & $V_{21}$ & $V_{22}$ & $V_{23}$ & $V_{24}$ & $V_{25}$ & $V_{26}$ & $V_{27}$ \\
 \midrule 
 $-2$ &  &  &  &  &  &  &  \\
$1$ &  &  &  &  &  &  &  \\
$4$ &  &  &  &  &  &  &  \\
$7$ &  & $1$ & $1$ &  &  & $2$ &  \\
$10$ & $4$ & $2$ & $2$ & $5$ & $6$ & $11$ & $-1$ \\
$13$ & $58$ & $9$ & $9$ & $20$ & $38$ & $56$ & $-26$ \\
$16$ & $503$ & $48$ & $48$ & $128$ & $438$ & $413$ & $-190$ \\
\bottomrule
\end{tabular}
\end{center}
\end{tiny}
\end{sidewaystable}

\clearpage

\begin{sidewaystable}
\begin{tiny}
\begin{center}
\caption{Decomposition of $V^{(-3,1)}_{n}$ into irreducible modules for $\Th$, Part 5.}\label{w0PTdecomps5}
\begin{tabular}{c@{ }@{\;}r@{ }r@{ }r@{ }r@{ }r@{ }r@{ }r}\toprule
 $n$ & $V_{28}$ & $V_{29}$ & $V_{30}$ & $V_{31}$ & $V_{32}$ & $V_{33}$ \\
 \midrule 
 $-2$ &  &  &  &  &  &  \\
$1$ &  &  &  &  &  &  \\
$4$ &  &  &  &  &  &  \\
$7$ &  & $-2$ & $-2$ &  &  &  \\
$10$ & $-1$ & $-3$ & $-3$ & $1$ & $6$ &  \\
$13$ & $-26$ & $-32$ & $-32$ & $48$ & $16$ &  \\
$16$ & $-190$ & $-216$ & $-216$ & $487$ & $240$ &  \\
\bottomrule
\end{tabular}
\end{center}
\end{tiny}
\end{sidewaystable}

\clearpage

\begin{sidewaystable}
\begin{tiny}
\begin{center}
\caption{Decomposition of $V^{(-3,1)}_{n}$ into irreducible modules for $\Th$, Part 6.}\label{w0PTdecomps6}
\begin{tabular}{c@{ }@{\;}r@{ }r@{ }r@{ }r@{ }r@{ }r}\toprule
 $n$ & $V_{34}$ & $V_{35}$ & $V_{36}$ & $V_{37}$ & $V_{38}$ \\
 \midrule 
 $-2$ &  &  &  &  &  \\
$1$ &  &  &  &  &  \\
$4$ &  &  &  &  &  \\
$7$ &  &  &  &  &  \\
$10$ & $-6$ &  &  & $5$ & $4$ \\
$13$ & $-14$ &  &  & $52$ & $64$ \\
$16$ & $-210$ &  &  & $588$ & $559$ \\
\bottomrule
\end{tabular}
\end{center}
\end{tiny}
\end{sidewaystable}

\clearpage

\begin{sidewaystable}
\begin{tiny}
\begin{center}
\caption{Decomposition of $V^{(-3,1)}_{n}$ into irreducible modules for $\Th$, Part 7.}\label{w0PTdecomps7}
\begin{tabular}{c@{ }@{\;}r@{ }r@{ }r@{ }r@{ }r@{ }r}\toprule
 $n$ & $V_{39}$ & $V_{40}$ & $V_{41}$ & $V_{42}$ & $V_{43}$ \\
 \midrule 
 $-2$ &  &  &  &  &  \\
$1$ &  &  &  &  &  \\
$4$ &  &  &  &  &  \\
$7$ &  &  &  &  &  \\
$10$ &  & $1$ &  &  & $5$ \\
$13$ & $-2$ & $8$ &  &  & $38$ \\
$16$ & $-29$ & $116$ & $1$ & $1$ & $437$ \\
\bottomrule
\end{tabular}
\end{center}
\end{tiny}
\end{sidewaystable}

\clearpage

\begin{sidewaystable}
\begin{tiny}
\begin{center}
\caption{Decomposition of $V^{(-3,1)}_{n}$ into irreducible modules for $\Th$, Part 8.}\label{w0PTdecomps8}
\begin{tabular}{c@{ }@{\;}r@{ }r@{ }r@{ }r@{ }r@{ }r}\toprule
 $n$ & $V_{44}$ & $V_{45}$ & $V_{46}$ & $V_{47}$ & $V_{48}$ \\
 \midrule 
 $-2$ &  &  &  &  &  \\
$1$ &  &  &  &  &  \\
$4$ &  &  &  &  &  \\
$7$ &  &  &  &  &  \\
$10$ & $7$ & $-5$ &  & $-8$ &  \\
$13$ & $94$ & $-36$ &  & $-102$ &  \\
$16$ & $880$ & $-427$ & $3$ & $-985$ & $3$ \\
\bottomrule
\end{tabular}
\end{center}
\end{tiny}
\end{sidewaystable}

\clearpage

\begin{sidewaystable}
\subsection{Weight One-Half 3C-Generalized Thompson Moonshine}
\caption{Decomposition of $(\breve{W}_{{\rm 3C}})_{D}$ into irreducible modules for $\Th$, Part 1.}\label{whalf3Cdecomps1}
\begin{center}
\begin{tiny}
\begin{tabular}{c@{ }@{ }r@{ }r@{ }r@{ }r@{ }r@{ }r@{ }r@{ }r@{ }r@{ }r@{ }r@{ }r@{ }r@{ }r}
\toprule 
$D$ &$V_{1}$ & $V_{2}$ & $V_{3}$ & $V_{4}$ & $V_{5}$ & $V_{6}$ & $V_{7}$ & $V_{8}$ & $V_{9}$ & $V_{10}$ & $V_{11}$ & $V_{12}$ & $V_{13}$ & $V_{14}$ \\\midrule
$-3$ & $3$ & $0$ & $0$ & $0$ & $0$ & $0$ & $0$ & $0$ & $0$ & $0$ & $0$ & $0$ & $0$ & $0$ \\
$1$ & $0$ & $3$ & $0$ & $0$ & $0$ & $0$ & $0$ & $0$ & $0$ & $0$ & $0$ & $0$ & $0$ & $0$ \\
$4$ & $0$ & $0$ & $-3$ & $0$ & $0$ & $0$ & $3$ & $0$ & $0$ & $0$ & $0$ & $0$ & $0$ & $0$ \\
$5$ & $2$ & $0$ & $0$ & $0$ & $0$ & $-1$ & $0$ & $-2$ & $0$ & $0$ & $2$ & $0$ & $0$ & $0$ \\
$8$ & $1$ & $0$ & $0$ & $0$ & $0$ & $1$ & $0$ & $2$ & $0$ & $0$ & $1$ & $0$ & $0$ & $0$ \\
$9$ & $3$ & $0$ & $0$ & $0$ & $0$ & $3$ & $0$ & $-3$ & $0$ & $0$ & $3$ & $0$ & $0$ & $0$ \\
$12$ & $3$ & $0$ & $0$ & $0$ & $0$ & $3$ & $3$ & $-3$ & $0$ & $0$ & $3$ & $0$ & $0$ & $0$ \\
\bottomrule
\end{tabular}
\end{tiny}
\end{center}
\end{sidewaystable}

\clearpage
\begin{sidewaystable}
\caption{Decomposition of $(\breve{W}_{{\rm 3C}})_{D}$ into irreducible modules for $\Th$, Part 2.}\label{whalf3Cdecomps2}
\begin{center}
\begin{tiny}
\begin{tabular}{c@{ }@{ }r@{ }r@{ }r@{ }r@{ }r@{ }r@{ }r@{ }r@{ }r@{ }r@{ }r@{ }r}
\toprule 
$D$ &$V_{15}$ & $V_{16}$ & $V_{17}$ & $V_{18}$ & $V_{19}$ & $V_{20}$ & $V_{21}$ & $V_{22}$ & $V_{23}$ & $V_{24}$ & $V_{25}$ & $V_{26}$ \\\midrule
$-3$ & $0$ & $0$ & $0$ & $0$ & $0$ & $0$ & $0$ & $0$ & $0$ & $0$ & $0$ & $0$ \\
$1$ & $0$ & $0$ & $0$ & $0$ & $0$ & $0$ & $0$ & $0$ & $0$ & $0$ & $0$ & $0$ \\
$4$ & $0$ & $0$ & $0$ & $0$ & $0$ & $0$ & $0$ & $0$ & $0$ & $0$ & $0$ & $0$ \\
$5$ & $0$ & $1$ & $0$ & $0$ & $1$ & $0$ & $0$ & $0$ & $0$ & $0$ & $0$ & $-1$ \\
$8$ & $0$ & $-1$ & $0$ & $0$ & $-1$ & $0$ & $0$ & $0$ & $0$ & $0$ & $0$ & $1$ \\
$9$ & $0$ & $-3$ & $0$ & $0$ & $0$ & $0$ & $0$ & $0$ & $0$ & $0$ & $0$ & $3$ \\
$12$ & $0$ & $-3$ & $0$ & $0$ & $0$ & $3$ & $0$ & $0$ & $0$ & $3$ & $0$ & $3$ \\
\bottomrule
\end{tabular}
\end{tiny}
\end{center}
\end{sidewaystable}

\clearpage
\begin{sidewaystable}
\caption{Decomposition of $(\breve{W}_{{\rm 3C}})_{D}$ into irreducible modules for $\Th$, Part 3.}\label{whalf3Cdecomps3}
\begin{center}
\begin{tiny}
\begin{tabular}{c@{ }@{ }r@{ }r@{ }r@{ }r@{ }r@{ }r@{ }r@{ }r@{ }r@{ }r@{ }r}
\toprule 
$D$ &$V_{27}$ & $V_{28}$ & $V_{29}$ & $V_{30}$ & $V_{31}$ & $V_{32}$ & $V_{33}$ & $V_{34}$ & $V_{35}$ & $V_{36}$ & $V_{37}$ \\\midrule
$-3$ & $0$ & $0$ & $0$ & $0$ & $0$ & $0$ & $0$ & $0$ & $0$ & $0$ & $0$ \\
$1$ & $0$ & $0$ & $0$ & $0$ & $0$ & $0$ & $0$ & $0$ & $0$ & $0$ & $0$ \\
$4$ & $0$ & $0$ & $0$ & $0$ & $0$ & $0$ & $0$ & $0$ & $0$ & $0$ & $0$ \\
$5$ & $0$ & $0$ & $0$ & $0$ & $0$ & $0$ & $0$ & $0$ & $0$ & $0$ & $-1$ \\
$8$ & $0$ & $0$ & $0$ & $0$ & $0$ & $0$ & $0$ & $0$ & $0$ & $0$ & $-2$ \\
$9$ & $0$ & $0$ & $0$ & $0$ & $0$ & $0$ & $0$ & $0$ & $0$ & $0$ & $0$ \\
$12$ & $0$ & $0$ & $0$ & $0$ & $-3$ & $0$ & $0$ & $0$ & $0$ & $0$ & $0$ \\
\bottomrule
\end{tabular}
\end{tiny}
\end{center}
\end{sidewaystable}

\clearpage
\begin{sidewaystable}
\caption{Decomposition of $(\breve{W}_{{\rm 3C}})_{D}$ into irreducible modules for $\Th$, Part 4.}\label{whalf3Cdecomps4}
\begin{center}
\begin{tiny}
\begin{tabular}{c@{ }@{ }r@{ }r@{ }r@{ }r@{ }r@{ }r@{ }r@{ }r@{ }r@{ }r@{ }r}
\toprule 
$D$ &$V_{38}$ & $V_{39}$ & $V_{40}$ & $V_{41}$ & $V_{42}$ & $V_{43}$ & $V_{44}$ & $V_{45}$ & $V_{46}$ & $V_{47}$ & $V_{48}$ \\\midrule
$-3$ & $0$ & $0$ & $0$ & $0$ & $0$ & $0$ & $0$ & $0$ & $0$ & $0$ & $0$ \\
$1$ & $0$ & $0$ & $0$ & $0$ & $0$ & $0$ & $0$ & $0$ & $0$ & $0$ & $0$ \\
$4$ & $0$ & $0$ & $0$ & $0$ & $0$ & $0$ & $0$ & $0$ & $0$ & $0$ & $0$ \\
$5$ & $1$ & $0$ & $0$ & $0$ & $0$ & $0$ & $0$ & $0$ & $0$ & $0$ & $0$ \\
$8$ & $2$ & $0$ & $0$ & $0$ & $0$ & $0$ & $0$ & $0$ & $0$ & $0$ & $0$ \\
$9$ & $0$ & $0$ & $0$ & $0$ & $0$ & $0$ & $0$ & $0$ & $0$ & $0$ & $0$ \\
$12$ & $0$ & $0$ & $3$ & $0$ & $0$ & $0$ & $0$ & $0$ & $0$ & $0$ & $0$ \\
\bottomrule
\end{tabular}
\end{tiny}
\end{center}
\end{sidewaystable}

\begin{sidewaystable}
\subsection{Weight One-Half Penumbral Moonshine}
\begin{scriptsize}
\begin{center}
\caption{Decomposition of $\breve{W}^{(-3,1)}_D$ into irreducible modules for $\Th$, Part 1.}\label{mequals1D0m3decomps1}
\medskip
\begin{tabular}{c@{ }c@{ }@{\;}r@{ }r@{ }r@{ }r@{ }r@{ }r@{ }r@{ }r@{ }r@{ }r@{ }r}\toprule
$r~~$ &  $D$ & $V_{1}$ & $V_{2}$ & $V_{3}$ & $V_{4}$ & $V_{5}$ & $V_{6}$ & $V_{7}$ & $V_{8}$ & $V_{9}$ & $V_{10}$ \\
 \midrule 
$0~~$
 &  $0$ &  & $1$ &  &  &  &  &  &  &  &  \\
 & $4$ &  &  &  & $1$ & $1$ &  &  &  &  &  \\
 & $8$ &  &  &  &  &  &  &  &  &  &  \\
 & $12$ &  &  &  &  &  &  &  &  &  &  \\
 & $16$ &  &  &  &  &  &  &  &  &  &  \\
 & $20$ &  &  &  &  &  &  &  &  & $1$ & $1$ \\
 & $24$ &  &  &  &  &  &  &  &  &  &  \\
\bottomrule

\end{tabular}
\end{center}
\end{scriptsize}
\end{sidewaystable}

\clearpage

\begin{sidewaystable}

\begin{scriptsize}
\begin{center}
\caption{Decomposition of $\breve{W}^{(-3,1)}_D$ into irreducible modules for $\Th$, Part 2.}\label{mequals1D0m3decomps2}
\begin{tabular}{c@{ }c@{ }@{\;}r@{ }r@{ }r@{ }r@{ }r@{ }r@{ }r@{ }r@{ }r}\toprule
$r~~$ &  $D$ & $V_{11}$ & $V_{12}$ & $V_{13}$ & $V_{14}$ & $V_{15}$ & $V_{16}$ & $V_{17}$ & $V_{18}$ \\
 \midrule 
$0~~$
 &  $0$ &  &  &  &  &  &  &  &  \\
 & $4$ &  &  &  &  &  &  &  &  \\
 & $8$ &  &  &  &  &  &  & $1$ & $1$ \\
 & $12$ &  &  &  &  &  &  &  &  \\
 & $16$ &  &  &  &  &  &  &  &  \\
 & $20$ &  &  &  &  &  &  &  &  \\
 & $24$ &  &  &  & $2$ & $2$ & $4$ & $4$ & $4$ \\
\bottomrule

\end{tabular}
\end{center}
\end{scriptsize}
\end{sidewaystable}

\clearpage

\begin{sidewaystable}

\begin{scriptsize}
\begin{center}
\caption{Decomposition of $\breve{W}^{(-3,1)}_D$ into irreducible modules for $\Th$, Part 3.}\label{mequals1D0m3decomps3}
\begin{tabular}{c@{ }c@{ }@{\;}r@{ }r@{ }r@{ }r@{ }r@{ }r@{ }r@{ }r@{ }r}\toprule
$r~~$ &  $D$ & $V_{19}$ & $V_{20}$ & $V_{21}$ & $V_{22}$ & $V_{23}$ & $V_{24}$ & $V_{25}$ & $V_{26}$ \\
 \midrule 
$0~~$
 &  $0$ &  &  &  &  &  &  &  &  \\
 & $4$ &  &  &  &  &  &  &  &  \\
 & $8$ &  &  &  &  &  &  &  &  \\
 & $12$ &  &  &  &  &  &  &  &  \\
 & $16$ &  &  &  &  &  &  &  &  \\
 & $20$ & $2$ &  & $2$ &  &  &  & $2$ &  \\
 & $24$ & $2$ & $6$ & $4$ & $8$ & $8$ & $8$ & $4$ & $10$ \\
\bottomrule

\end{tabular}
\end{center}
\end{scriptsize}
\end{sidewaystable}

\clearpage

\begin{sidewaystable}

\begin{scriptsize}
\begin{center}
\caption{Decomposition of $\breve{W}^{(-3,1)}_D$ into irreducible modules for $\Th$, Part 4.}\label{mequals1D0m3decomps4}
\begin{tabular}{c@{ }c@{ }@{\;}r@{ }r@{ }r@{ }r@{ }r@{ }r@{ }r@{ }r@{ }r}\toprule
$r~~$ &  $D$ & $V_{27}$ & $V_{28}$ & $V_{29}$ & $V_{30}$ & $V_{31}$ & $V_{32}$ & $V_{33}$ & $V_{34}$ \\
 \midrule 
$0~~$
 &  $0$ &  &  &  &  &  &  &  &  \\
 & $4$ &  &  &  &  &  &  &  &  \\
 & $8$ &  &  &  &  &  &  &  &  \\
 & $12$ &  &  &  &  &  &  &  &  \\
 & $16$ & $1$ & $1$ &  &  &  &  &  &  \\
 & $20$ &  &  &  &  &  & $2$ & $4$ & $4$ \\
 & $24$ & $14$ & $14$ & $14$ & $14$ & $22$ & $18$ & $22$ & $28$ \\
\bottomrule

\end{tabular}
\end{center}
\end{scriptsize}
\end{sidewaystable}

\clearpage

\begin{sidewaystable}

\begin{scriptsize}
\begin{center}
\caption{Decomposition of $\breve{W}^{(-3,1)}_D$ into irreducible modules for $\Th$, Part 5.}\label{mequals1D0m3decomps5}
\begin{tabular}{c@{ }c@{ }@{\;}r@{ }r@{ }r@{ }r@{ }r@{ }r@{ }r@{ }r}\toprule
$r~~$ &  $D$ & $V_{35}$ & $V_{36}$ & $V_{37}$ & $V_{38}$ & $V_{39}$ & $V_{40}$ & $V_{41}$ \\
 \midrule 
$0~~$
 &  $0$ &  &  &  &  &  &  &  \\
 & $4$ &  &  &  &  &  &  &  \\
 & $8$ &  &  &  &  &  &  &  \\
 & $12$ &  &  &  &  &  & $2$ &  \\
 & $16$ &  &  &  &  &  &  &  \\
 & $20$ & $4$ & $4$ & $8$ & $6$ & $12$ & $10$ & $12$ \\
 & $24$ & $38$ & $38$ & $44$ & $52$ & $56$ & $68$ & $80$ \\
\bottomrule

\end{tabular}
\end{center}
\end{scriptsize}
\end{sidewaystable}

\clearpage

\begin{sidewaystable}

\begin{scriptsize}
\begin{center}
\caption{Decomposition of $\breve{W}^{(-3,1)}_D$ into irreducible modules for $\Th$, Part 6.}\label{mequals1D0m3decomps6}
\begin{tabular}{c@{ }c@{ }@{\;}r@{ }r@{ }r@{ }r@{ }r@{ }r@{ }r@{ }r}\toprule
$r~~$ &  $D$ & $V_{42}$ & $V_{43}$ & $V_{44}$ & $V_{45}$ & $V_{46}$ & $V_{47}$ & $V_{48}$ \\
 \midrule 
$0~~$
 &  $0$ &  &  &  &  &  &  &  \\
 & $4$ &  &  &  &  &  &  &  \\
 & $8$ &  &  &  &  &  &  &  \\
 & $12$ &  &  &  &  &  &  &  \\
 & $16$ & $2$ & $2$ &  & $2$ & $2$ &  & $4$ \\
 & $20$ & $10$ & $12$ & $16$ & $14$ & $16$ & $24$ & $32$ \\
 & $24$ & $132$ & $136$ & $126$ & $140$ & $154$ & $180$ & $332$ \\
\bottomrule

\end{tabular}
\end{center}
\end{scriptsize}
\end{sidewaystable}

%------------------------------------------------------------------%
\clearpage
%------------------------------------------------------------------%

\addcontentsline{toc}{section}{Declarations}

\section*{Declarations}

\subsection*{Competing Interests}

The authors declare no competing financial or non-financial interests directly or indirectly related to this article.

\subsection*{Data Availability}

All relevant data is available in the appendices of this article.

\addcontentsline{toc}{section}{References}

\newcommand{\etalchar}[1]{$^{#1}$}

\end{document}